\documentclass[11pt]{amsart}

\usepackage[numeric]{amsrefs}
\DeclareRobustCommand{\VAN}[3]{#2}
\usepackage[utf8]{inputenc}
\usepackage{csquotes}

\usepackage{tikz-cd}
\usetikzlibrary{decorations.pathmorphing}
\usepackage[font=small,format=plain,labelfont=bf,up,textfont=it,up]{caption}

\newcommand{\TJM}[4]{
\bigl(\begin{smallmatrix}
  #1 & #2 \\
  #3 & #4
\end{smallmatrix} \bigr)}
\usepackage[margin=1in]{geometry} 
\usepackage{fancyhdr}
\usepackage{hyperref}

\usepackage{stmaryrd}

\usepackage{amsmath}
\usepackage{mathtools} 
\usepackage{amsfonts}
\usepackage{amssymb} 
\usepackage{amsthm}
\usepackage{mathrsfs}
\usepackage{calligra}
\newtheorem{Thm}[subsubsection]{Theorem}
\newtheorem{Lem}[subsubsection]{Lemma}
\newtheorem{Prop}[subsubsection]{Proposition}
\newtheorem{Cor}[subsubsection]{Corollary}

\newtheorem{mainThm}{Theorem}

\newtheorem*{IntroThm}{Theorem}

\theoremstyle{definition}
\newtheorem{Def}[subsubsection]{Definition}
\newtheorem{Hyp}[subsubsection]{Hypothesis}

\newtheorem{Rem}[subsubsection]{Remark}

\numberwithin{equation}{subsection}

\DeclareMathOperator{\Hom}{\mathscr{H}\mathrm{\kern -3pt {\calligra\large om}}\,}

\newcommand{\zp}{\mathbb{Z}_p}

\newcommand{\qbar}{\overline{\mathbb{Q}}}
\newcommand{\qlbar}{\overline{\mathbb{Q}}_{\ell}}
\newcommand{\qpbar}{\overline{\mathbb{Q}}_{p}}

\newcommand{\qp}{\mathbb{Q}_p}

\newcommand{\art}{\operatorname{Art}}
\newcommand{\nilp}{\operatorname{Nilp}}
\newcommand{\shginf}{\operatorname{Sh}_{G, K_p}}
\newcommand{\shg}{\operatorname{Sh}_{G, K^pK_p}}
\newcommand{\shgsp}{\operatorname{Sh}_{\mathcal{G}_{V}, \mathcal{K}^p \mathcal{K}_p}}
\newcommand{\shgspinf}{\operatorname{Sh}_{\mathcal{G}_{V},\mathcal{K}_p}}
\newcommand{\shginfov}{\operatorname{Sh}_{G, K_p,\ovfp}}
\newcommand{\shgov}{\operatorname{Sh}_{G, K^pK_p,\ovfp}}
\newcommand{\shgord}{\operatorname{Sh}_{G, \mathrm{ord}}}
\newcommand{\shgb}{\operatorname{Sh}_{G,[b],K^pK_p}}
\newcommand{\tildezcl}{Z^{\circ}_{\ell}}
\newcommand{\tildezl}{{Z}_{\ell}}
\newcommand{\afp}{\mathbb{A}_{f}^{p}}
\newcommand{\af}{\mathbb{A}_{f}}
\newcommand{\afs}{\mathbb{A}_{f}^{\Sigma}}
\newcommand{\ovfp}{\overline{\mathbb{F}}_{p}}
\newcommand{\bgmu}{B(G,\{\mu^{-1}\})}
\newcommand{\fp}{\mathbb{F}_{p}}
\newcommand{\gsc}{G^{\mathrm{sc}}}
\newcommand{\gder}{G^{\mathrm{der}}}
\newcommand{\zcirc}{Z^{\circ}}
\newcommand{\ql}{\mathbb{Q}_{\ell}}
\newcommand{\zl}{\mathbb{Z}_{\ell}}

\newcommand{\qpbr}{\breve{\mathbb{Q}}_{p}}
\newcommand{\zpbr}{\breve{\mathbb{Z}}_{p}}

\newcommand{\spec}{\operatorname{Spec}}
\newcommand{\spf}{\operatorname{Spf}}

\newcommand{\gal}{\operatorname{Gal}}

\newcommand\restr[2]{{
  \left.\kern-\nulldelimiterspace 
  #1 
  \vphantom{\big|} 
  \right|_{#2} 
  }}
\usepackage{float}
\mathtoolsset{showonlyrefs}
\usepackage{bbm}
\usepackage{graphicx}
\usepackage{enumitem}
\usepackage{hyperref}
\author{Pol van Hoften} \email{pol.van.hoften@stanford.edu}
\address{Stanford mathematics department, 450 Jane Stanford, Way Building 380x, Stanford, CA 94305, USA}

\title{On the ordinary Hecke orbit conjecture}

\begin{document}
\maketitle 
\begin{abstract}
  We prove the ordinary Hecke orbit conjecture for Shimura varieties of Hodge type at primes of good reduction. We make use of the global Serre--Tate coordinates of Chai as well as recent results of D'Addezio about the $p$-adic monodromy of isocrystals. The new ingredients in this paper are a general monodromy theorem for Hecke-stable subvarieties for Shimura varieties of Hodge type, and a rigidity result for the formal completions of ordinary Hecke orbits. Along the way we show that classical Serre--Tate coordinates can be described using unipotent formal groups, generalising a result of Howe.
\end{abstract}
\tableofcontents
\section{Introduction}
Let $\mathcal{A}_{g,n}$ be the moduli space of $g$-dimensional principally polarised abelian varieties $(A, \lambda)$ with level $n \ge 3$ structure over $\ovfp$, for a prime number $p$ coprime to $n$. Recall that there are finite \'etale prime-to-$p$ Hecke correspondences from $\mathcal{A}_{g,n}$ to itself, and that two points $x,y \in \mathcal{A}_{g,n}(\ovfp)$ are said to be \emph{in the same prime-to-$p$ Hecke orbit} if they share a preimage under one of these correspondences. Recall the following result of Chai:
\begin{IntroThm}[Chai \cite{ChaiOrdinarySiegel}] \label{Thm:Chai}
Let $x \in \mathcal{A}_{g,n}(\ovfp)$ be a point corresponding to an ordinary principally polarised abelian variety. Then the prime-to-$p$ Hecke orbit of $x$ is Zariski dense in $\mathcal{A}_{g,n,\ovfp}$.
\end{IntroThm}
Our main result is a generalisation of this theorem to Shimura varieties of Hodge type. To state it, we will first to introduce some notation.
\subsection{Main results}
Let $(G,X)$ be a Shimura datum of Hodge type with reflex field $E$ and let $p$ be a prime number. Let $K_p \subset G(\qp)$ be a hyperspecial subgroup and let $K^p \subset G(\afp)$ be a sufficiently small compact open subgroup. Let $\operatorname{Sh}_G$ be the special fiber of the canonical integral model of the Shimura variety of level $K^pK_p$ at a prime $v$ above $p$ of $E$, constructed in \cites{KisinModels,MadapusiPeraKim}.

Let $E_v$ be the $v$-adic completion of $E$, which is a finite extension of $\mathbb{Q}_p$. There is a closed immersion $\operatorname{Sh}_{G,\ovfp} \to \mathcal{A}_{g,n}$ for some $n$ (see \cite{xu2020normalization}), and the intersection $\shgord$ of the ordinary locus of $\mathcal{A}_{g,n}$ with $\operatorname{Sh}_G$ is nonempty if and only if $E_v=\mathbb{Q}_p$ (see
\cite[Cor. 1.0.2]{LeeNewton}). Recall that there are also prime-to-$p$ Hecke correspondences from $\operatorname{Sh}_G$ to itself, and thus we can define prime-to-$p$ Hecke orbits.
\begin{mainThm} \label{Thm:Ordinary}
Suppose that $E_v=\mathbb{Q}_p$. Then the prime-to-$p$ Hecke orbit of a point $x \in \shgord(\ovfp)$ is Zariski dense in $\operatorname{Sh}_G$.
\end{mainThm}
Our result generalises results of Maulik-Shankar-Tang \cite{MaulikShankarTang}, who deal with $\operatorname{GSpin}$ Shimura varieties associated to a quadratic space over $\mathbb{Q}$ and $GU(1,n-1)$ Shimura varieties associated to imaginary quadratic fields $E$ with $p$ split in $E$; their methods are completely disjoint from ours. There is also work of Shankar \cite{ShankarEtranges} for Shimura varieties of type C, using a group-theoretic version of Chai's strategy of using hypersymmetric points and reducing to the case of Hilbert modular varieties. Shankar crucially proves that the Hodge map $\operatorname{Sh}_G \to \mathcal{A}_{g,n}$ is a closed immersion over the ordinary locus via canonical liftings, whereas we use work of Xu \cite{xu2020normalization}. 

Last we mention work of Zhou \cite{ZhouMotivic}, who proves the Hecke orbit conjecture for the $\mu$-ordinary locus of certain quaternionic Shimura varieties. Our results do not imply his, but there is some overlap between the cases that we cover.

A fairly direct consequence of Theorem \ref{Thm:Ordinary} is a density result for prime-to-$p$ Hecke orbits of an $\ovfp$-point in the $\mu$-ordinary locus of a Shimura variety of abelian type, at primes $v$ above $p$ of the reflex field $E$ where $E_v=\mathbb{Q}_p$, see Corollary \ref{Cor:AbelianType}.

\subsection{Monodromy theorems} An important ingredient in our proof is an $\ell$-adic monodromy theorem for prime-to-$p$ Hecke-stable subvarieties of special fibers of Shimura varieties, in the style of \cite[Cor. 3.5]{ChaiElladicmonodromy}. To state it, let $(G,X)$ be as above and assume for simplicity that $G^{\mathrm{ad}}$ is simple over $\mathbb{Q}$. Let $V_{\ell}$ be the rational $\ell$-adic Tate module of the abelian variety $A$ over $\operatorname{Sh}_G$ coming from the map $\operatorname{Sh}_{G,\ovfp} \to \mathcal{A}_{g,n}$; it is an $\ell$-adic local system of rank $2g$.
\begin{mainThm} \label{Thm:MonodromyIntroduction}
Let $Z \subset \operatorname{Sh}_G$ be a smooth locally closed subvariety that is stable under the prime-to-$p$ Hecke operators. Suppose that $Z$ is not contained in the smallest Newton stratum of $\operatorname{Sh}_G$. Let $z \in Z(\ovfp)$ and let $\zcirc \subset Z_{\ovfp}$ be the connected component of $Z$ containing $z$. Then the neutral component $\mathbb{M}_{\mathrm{geom}}$ of the Zariski closure of the image of the monodromy representation 
\begin{align}
    \rho_{\ell,\mathrm{geom}}:\pi_1^{\text{\'et}}(Z^{\circ},z) \to \operatorname{GL}_{2g}(\mathbb{Q}_{\ell})
\end{align}
corresponding to $V_{\ell }$, is isomorphic to $G^{\mathrm{der}}_{\ql}$.
\end{mainThm}
This generalises work of Chai \cite{ChaiElladicmonodromy} in the Siegel case and others \cites{KasprowitzMonodromy,LXiao} in the PEL case. \smallskip 

In the body of the paper, we work with the integral models of Shimura varieties of Hodge type of level $K_p \subset G(\qp)$ constructed in \cite{KMPS}. Here $K_p$ is not required to be hyperspecial, for example it is allowed to be any (connected) parahoric subgroup. Our results, namely Theorem \ref{Thm:ArithmeticMonodromy} and Corollary \ref{Cor:GeometricMonodromy}, are proved under the assumption that Hypothesis \ref{Hyp:Levi} holds. This hypothesis holds for example when $G_{\qp}$ is quasi-split and has no factors of type $D$, or when $K_p$ is hyperspecial. 

We also prove results about irreducible components of smooth locally closed subvarieties that are stable under the prime-to-$p$ Hecke operators, in the style of \cite[Proposition 4.4]{ChaiElladicmonodromy}, see Theorem \ref{Thm:Irreducibility}. These results will be used to prove irreducibility of Ekedahl--Oort strata in an upcoming version of \cite{HoftenZhou}.

\subsubsection{An overview of the proof of Theorem \ref{Thm:MonodromyIntroduction}} Since $Z^{\circ} \subset \operatorname{Sh}_G$ is defined over a finite field $k$ we can write it as $Z^{\circ}_k \otimes_k \ovfp$. We can then consider the Zariski closure $\mathbb{M}$ of the image of
\begin{align}
    \rho_{\ell}:\pi_1^{\text{\'et}}(Z_k^{\circ},z) \to \operatorname{GL}_{2g}(\mathbb{Q}_{\ell}).
\end{align}
An argument from \cite{ChaiElladicmonodromy} proves that $\mathbb{M}$ is (isomorphic to) a normal subgroup of $G_{\ql}$. If $G_{\ql}^{\mathrm{ad}}$ was a simple-group, then we would be done if we could show that $\mathbb{M}$ was not central in $G_{\ql}$. However, in general there are \emph{no} primes $\ell$ such that $G_{\ql}^{\mathrm{ad}}$ is simple and so at this point we have to deviate from the strategy of \cite{ChaiElladicmonodromy}.

Instead, we control $\mathbb{M}$ by studying the centraliser $I_{x,\ell} \subset G_{\ql}$ of the image of Frobenius elements $\operatorname{Frob}_x \in \pi_1^{\text{\'et}}(Z_k^{\circ},z)$ corresponding to points $x \in Z_k(\mathbb{F}_{q})$. Since the paper \cite{KMPS} makes an in-depth study of these Frobenius elements, we can make use of their results about these centralisers. For example, if $x$ is not contained in the basic locus, then $\operatorname{Frob}_x$ is not central. To get more precise results, we need to know that the element $\operatorname{Frob}_x \in G(\ql)$ is \emph{defined over $\mathbb{Q}$}, which is what Hypothesis \ref{Hyp:Levi} makes precise. 

In this way we can show that $\mathbb{M} \subset G_{\ql}$ is a normal subgroup that surjects onto $G_{\ql}^{\mathrm{ad}}$. The result about $\mathbb{M}_{\mathrm{geom}} \subset \mathbb{M}$ will be deduced from this.
\subsection{A sketch of the proof of Theorem \ref{Thm:Ordinary}}
Let $x \in \operatorname{Sh}_G(\ovfp)$ be an ordinary point, and let $Z$ be the Zariski closure inside $\shgord$ of the prime-to-$p$ Hecke orbit of $x$. Let $y \in Z(\ovfp)$ be a smooth point of $Z$. Recall that it follows from the theory of Serre--Tate coordinates that the formal completion $\mathcal{A}_{g,n}^{/y}$ of $\mathcal{A}_{g,n}$ at $y$ is a formal torus. A special case of the main result of \cite{ShankarZhou} tells us that
\begin{align}
    S^{/y}:=\operatorname{Sh}_{G,\ovfp}^{/y} \subset \mathcal{A}_{g,n}^{/y}
\end{align}
is a formal subtorus. Work of Chai on the deformation theory of ordinary $p$-divisible groups \cite{ChaiOrdinary} tells us that the dimension of the smallest formal subtorus of $S^{/y}$ containing $Z^{/y}$, is encoded in the unipotent radical of the $p$-adic monodromy group of the isocrystal $\mathcal{M}$ associated to the universal abelian variety $A$ over $Z$. 

Using Theorem \ref{Thm:MonodromyIntroduction} and results of D'Addezio, \cites{d2020monodromy, DAddezioII}, we compute the monodromy group of $\mathcal{M}$ over $Z$. It follows from this computation that the smallest formal subtorus of $S^{/y}$ containing $Z^{/y}$ is equal to $S^{/y}$.

We conclude by proving that the formal completion $Z^{/y}$ is a formal subtorus of $S^{/y}$. By the rigidity theorem for $p$-divisible formal groups of Chai \cite{ChaiRigidity}, it suffices to give a representation-theoretic description of the Dieudonn\'e module of $S^{/y}$. Unfortunately, the description of the subtorus $S^{/y}$ coming out of the work of \cite{ShankarZhou} does not readily lend itself to understanding its Dieudonn\'e module. 

Instead, we give a different proof that $S^{/y}$ is a subtorus of $\mathcal{A}_{g,n}^{/y}$. We do this by giving a new description of Serre--Tate coordinates in terms of actions of formal unipotent groups on Rapoport--Zink spaces, generalising results of Howe \cite{HoweUnipotent} in the case $g=1$. Once we have this perspective, the results of \cite{KimLeaves} give an explicit description of the Dieudonn\'e module of the torus $\mathcal{A}_{g,n}^{/y}$ as well as the Dieudonn\'e module of the subtorus $S^{/y}$. 
\subsection{Outline} Sections \ref{Sec:IntModels} and \ref{Sec:Monodromy} form the first part of the paper and work in a more general setting than the rest of the paper. In Section \ref{Sec:IntModels} we introduce the integral models of Shimura varieties of Hodge type constructed in \cite{KMPS}. We recall results and notation from loc. cit., in particular, about the Frobenius elements and their centralisers associated to $\ovfp$-points of these models. In Section \ref{Sec:Monodromy} we prove monodromy theorems for Hecke-stable subvarieties of the special fibers of these integral models, combining results of \cite{KMPS} with ideas of \cite{ChaiElladicmonodromy}.

Section \ref{Sec:SerreTate} is a standalone section on Serre--Tate coordinates. In it, we show that the classical Serre--Tate coordinates, as described in \cite{KatzSerreTate}, can be reinterpreted using actions of unipotent formal groups as in \cite{HoweUnipotent}. This section should be of independent interest. 

In Section \ref{Sec:SerreTateHodge}, we specialise to the smooth canonical integral models of Shimura varieties of Hodge type at hyperspecial level, and we moreover assume that the ordinary locus is nonempty. We reprove a result of \cite{ShankarZhou}, which states that the formal completion of the ordinary locus gives a subtorus of the Serre--Tate torus, and give a group-theoretic description of its Dieudonn\'e module. At the end of this section we also give a short interlude on strongly nontrivial actions of algebraic groups on isocrystals, which we will need to confirm the hypotheses of the rigidity theorem of \cite{ChaiRigidity}.

In Section \ref{Sec:ProofOfTheorems}, we put everything together and prove Theorem \ref{Thm:Ordinary}. We end by deducing a result for Shimura varieties of abelian type.
\subsection{Acknowledgements}
It is clear that our approach owes a substantial intellectual debt to the work of Chai and Oort, and in fact our main results were conceived after reading a remark in \cite{ChaiConjecture}. We are very grateful to the referee for pointing out a serious error in a previous version, for a very detailed reading of a second version and for many helpful comments. We thank Sean Cotner for helpful discussions. 
\section{Integral models of Shimura varieties of Hodge type} \label{Sec:IntModels}
Let $(G,X)$ be a Shimura datum of Hodge type. In this section we follow \cite[Sec. 1.3]{KMPS} and construct integral models for the Shimura varieties associated to $(G,X)$ in a very general situation. The main goal is to introduce various Frobenius elements $\gamma_{x,m,\ell} \in G(\ql)$ associated to $\mathbb{F}_{q^m}$-points of these integral models, and to discuss result of \cite{KMPS} about their centralisers $I_{x,m,\ell}$. We end by introducing Hypothesis \ref{Hyp:Levi}, which will be assumed throughout Section \ref{Sec:Monodromy}, and prove that it holds under minor assumptions.

\subsubsection{Hodge cocharacters} \label{Sec:HodgeCocharacters} If $(G,X)$ is a Shimura datum, then for each $x \in X$ there is a cocharacter $\mu_x:\mathbb{G}_{m,\mathbb{C}} \to G_{\mathbb{C}}$, see \cite[Sec. 1.2.3]{KMPS} for the precise definition. The $G(\mathbb{C})$-conjugacy class of $\mu_x$ does not depend on the choice of $x$ and we will write $\{\mu_X\}$ for this conjugacy class, and denote it by $\{\mu\}$ if $X$ is clear from context. This conjugacy class of cocharacters is defined over a number field $E \subset \mathbb{C}$, called the reflex field.

\subsection{The construction of integral models}
For a symplectic space $(V, \psi)$ over $\mathbb{Q}$ we write $\mathcal{G}_V:=\operatorname{GSp}(V, \psi)$ for the group of symplectic similitudes of $V$ over $\mathbb{Q}$. It admits a Shimura datum $\mathcal{H}_V$ consisting of the union of the Siegel upper and lower half spaces. Let $(G,X)$ be a Shimura datum of Hodge type with reflex field $E$ and let $(G,X) \to (\mathcal{G}_V, \mathcal{H}_V)$ be a Hodge embedding.

Fix a prime $p$ and choose a $\mathbb{Z}_{(p)}$-lattice $V_{(p)} \subset V$ on which $\psi$ is $\mathbb{Z}_{p}$-valued, and write $V_p=V_{(p)} \otimes \mathbb{Z}_p$. Write $\mathcal{K}_p \subset \mathcal{G}_V(\qp)$ for the stabiliser of $V_p$ in $\mathcal{G}_V(\qp)$, and similarly write $K_p$ for the stabiliser of $V_p$ in $G(\qp)$.\footnote{It is explained in \cite[Sec. 1.3.2]{KMPS} that the collection of subgroups $K_p \subset G(\qp)$ that can arise from this construction by varying the symplectic space and the Hodge embedding contains all stabilisers of vertices in the extended Bruhat--Tits building of $G_{\qp}$. It is moreover explained in loc. cit. that this collection is stable under finite intersections.} For every sufficiently small compact open subgroup $K^p \subset G(\afp)$ we can find $\mathcal{K}^p \subset \mathcal{G}_V(\afp)$ such that the Hodge embedding induces a closed immersion (see \cite[Lemma 2.1.2]{KisinModels})
\begin{align}
    \mathbf{Sh}_{K}(G,X) \to \mathbf{Sh}_{\mathcal{K}}(\mathcal{G}_V, \mathcal{H}_V)
\end{align}
of Shimura varieties of level $K=K^pK_p$ and $\mathcal{K}=\mathcal{K}^p \mathcal{K}_p$, respectively.
We let $\mathcal{S}_{\mathcal{K}}$ over $\mathbb{Z}_{(p)}$ be the moduli-theoretic integral model of $\mathbf{Sh}_{\mathcal{K}}(\mathcal{G}_V, \mathcal{H}_V)$; it is a moduli space of polarised abelian schemes $(A, \lambda)$ up to prime-to-$p$ isogeny with level $\mathcal{K}^p$-structure.

Fix a prime $v \mid p$ of $E$ and let
\begin{align}
    \mathscr{S}_{K}:=\mathscr{S}_K(G,X) \to \mathcal{S}_{\mathcal{K}} \otimes_{\mathbb{Z}_{(p)}} \mathcal{O}_{E,(v)}
\end{align}
be the normalisation of the Zariski closure of $\mathbf{Sh}_{K}(G,X)$ in $\mathcal{S}_{\mathcal{K}} \otimes_{\mathbb{Z}_{(p)}} \mathcal{O}_{E,(v)}$. This construction is compatible with changing the level away from $p$ and we define
\begin{align}
    \mathcal{S}_{\mathcal{K}_p}:=\varprojlim_{\mathcal{K}^p \subset \mathcal{G}_V(\afp)} \mathcal{S}_{\mathcal{K}^p \mathcal{K}_p}, ; \qquad
    \mathscr{S}_{K_p}:=\varprojlim_{K^p \subset G(\afp)} \mathscr{S}_{K^pK_p}.
\end{align}
Then as discussed in \cite[Sec. 2.1]{KMPS}, the transition maps in both inverse systems are finite \'etale and moreover $G(\afp)$ acts on $\mathscr{S}_{K_p}$. Let $k=\mathbb{F}_q$ be the residue field of $\mathcal{O}_{E,(v)}$, and write $\shginf$ for the special fiber of $\mathscr{S}_{K_p}$ and $\shg$ for the special fiber of $\mathscr{S}_{K^pK_p}$; these are both schemes over $k$ and $G(\afp)$ acts on $\shginf$. We will write $\shgsp$ for the special fiber of $\mathcal{S}_{\mathcal{K}^p \mathcal{K}_p} \otimes_{\mathbb{Z}_{(p)}} \mathcal{O}_{E,(v)}$ and $\shgspinf$ for the special fiber of $\mathcal{S}_{\mathcal{K}_p} \otimes_{\mathbb{Z}_{(p)}} \mathcal{O}_{E,(v)}$.

Let $V^p$ be the prime-to-$p$ adelic Tate module of the universal abelian variety $\mathcal{A}$ over $\mathcal{S}_{\mathcal{K}_p}$; it is a smooth $\afp$-sheaf on $\mathcal{S}_{\mathcal{K}_p}$. As explained in \cite[Sec. 2.1.1]{KMPS} there is a universal isomorphism
\begin{align}
    \epsilon:V \otimes \underline{\mathbb{A}}_f^p \simeq V^p, 
\end{align}
sending the symplectic form $\psi$ to an $\underline{\mathbb{A}}_f^{p, \times}$-multiple of the Weil pairing. Here $\underline{\mathbb{A}}_f^p$ denotes the pro-\'etale sheaf associated to the topological group $\mathbb{A}_f^p$. 

\subsubsection{Tensors} \label{Sec:Tensors} Write $V^{\otimes}$ for the direct sum of $V^{\otimes n} \otimes (V^{\ast})^{\otimes m}$ for all pairs of integers $m \ge 0, n \ge 0$. We will also use this notation later for modules over commutative rings and modules over sheaves of rings. 

As in \cite[Sec. 1.3.4]{KMPS}, we fix tensors $\{s_{\alpha} \in V\} \subset V^{\otimes}$ such that $G$ is their pointwise stabiliser in $\operatorname{GL}(V)$. Then as explained in \cite[Sec. 1.3.4, Sec. 2.1.2]{KMPS}, there are global sections
\begin{align}
    \{s_{\alpha, \afp}\} \in H^0( \mathscr{S}_{K^pK_p}, (V^p)^{\otimes})
\end{align}
such that if we restrict the isomorphism $\epsilon$ via $\mathscr{S}_{K_p} \to \mathcal{S}_{\mathcal{K}_p}$ we get an isomorphism
\begin{align}
    \eta:V \otimes \underline{\mathbb{A}}_f^p \to V^p
\end{align}
taking $s_{\alpha} \otimes 1$ to $s_{\alpha, \afp}$ for all $\alpha$. In particular, for each $x \in \mathscr{S}_{K_p}(\ovfp)$ the stabiliser of the tensors $\{s_{\alpha, \afp}\}$ in $\operatorname{GL}(V^p)$ is canonically identified with $G \otimes \afp$.

\subsubsection{} \label{Sec:NewtonStratification} We will use $\zpbr$ to denote the $p$-typical Witt vectors $W(\ovfp)$ of $\ovfp$ and we set $\qpbr=\zpbr[1/p]$. We let $\sigma:\zpbr \to \zpbr$ be the automorphism induced by Frobenius on $\ovfp$, and also denote by $\sigma$ the induced automorphism of $\qpbr$.

Let $x \in \shg(\ovfp)$ and let $\mathbb{D}_x$ be the rational contravariant Dieudonn\'e module of the $p$-divisible group $A_x[p^{\infty}]$ of the abelian variety $A_x$, equipped with its Frobenius $\phi$. By \cite[Prop. 1.3.7]{KMPS} there are $\phi$-invariant tensors $\{s_{\alpha, \mathrm{cris}} \} \subset \mathbb{D}_x^{\otimes}$ and in \cite[Sec. 1.3.8]{KMPS} it is argued that there is an isomorphism $\qpbr \otimes V \to \mathbb{D}_x$ sending $1 \otimes s_{\alpha}$ to $s_{\alpha, \mathrm{cris}}$. See the statement of \cite[Prop. 1.3.7]{KMPS} for a characterisation of the tensors $s_{\alpha,\mathrm{cris}}$. 

Under such an isomorphism, the Frobenius $\phi$ corresponds to an element $b_x \in G(\qpbr)$, which is well defined up to $\sigma$-conjugacy, where $\sigma:G(\qpbr) \to G(\qpbr)$ is induced by $\sigma:\qpbr \to \qpbr$. In other words, we can associate to $\phi$ a well defined element $[b_x]$ of the Kottwitz set $B(G)=B(G_{\mathbb{Q}_p})$ of \cite{Kottwitz1}. By \cite[Lem. 1.3.9]{KMPS} the element $[b_x]$ is contained in the neutral acceptable set $\bgmu$ consisting of the $\{\mu^{-1}\}$-admissible elements defined in \cite[Sec. 1.1.5]{KMPS}. Here we use $\{\mu\}$ to denote the $G(\qpbar)$ conjugacy class of cocharacters induced by the place $v$ of $E$, where we recall that $\{\mu\}$ was introduced in Section \ref{Sec:HodgeCocharacters}.

It follows from \cite[Thm. 1.3.14]{KMPS} that there are locally closed subschemes $\shgb$ of $\shg$, called \emph{Newton strata}, indexed by $[b] \in \bgmu$, such that
\begin{align}
    \shgb(\ovfp) &= \{x \in \shg(\ovfp) \; | \; [b_x]=[b] \}
\end{align}
and such that
\begin{align}
    \overline{\shgb} \subset  \bigcup_{[b]' \le [b]} \operatorname{Sh}_{G,[b'],K^pK_p}.
\end{align}
Here we are using the partial order on $\bgmu$ defined in \cite[Sec. 2.3]{RapoportRichartz}.
\subsection{Centralisers} \label{Sec:Centralisers} Let $x \in \shginf(\ovfp)$ and choose a sufficiently divisible integer $m$ such that the image of $x$ in $\shg(\ovfp)$ is defined over $\mathbb{F}_{q^m}$. Then the geometric $q^m$-Frobenius $\operatorname{Frob}_{q^m}$ acts on $V^p$ via tensor-preserving automorphisms and therefore determines an element
\begin{align}
    \gamma_{x,m}^p \in G(\afp),
\end{align}
which depends on $x$ and $m$. For $\ell=p$ there is an element $\delta_{x,m} \in G(\mathbb{Q}_{q^m})$, constructed in \cite[Sec. 2.1.7]{KMPS}, whose class in $B(G_{\qp})$ is equal to $[b_x]$. Moreover there is an element $\gamma_{x,m,p} \in G(\mathbb{Q}_{q^m})$ such that 
\begin{align} \label{Eq:Norm}
    \gamma_{x,m,p}=\delta_{x,m} \sigma(\delta_{x,m}) \cdots \sigma^{rm-1}(\delta_{x,m}),
\end{align}
where we write $q=p^r$ and where $\sigma$ denotes the Frobenius on $G(\mathbb{Q}_{q^m})$.

We define $I_{x,\afp} \subset G_{\afp}$ to be the centraliser of $\gamma_{x,m}^p$, which does not depend on $m$ as long as $m$ is sufficiently divisible. We similarly define $I_{x,\ell} \subset G_{\ql}$ for $\ell \not=p$ to be the centraliser of the projection $\gamma_{x,m,\ell}$ of $\gamma_{x,m}^p$ to $G_{\ql}$ for sufficiently divisible $m$.

We define $I_{x,m,p}$ to be the algebraic group over $\qp$ whose functor of points is given by
\begin{align}
    R \mapsto \{g \in G(\mathbb{Q}_{q^m} \otimes_{\qp} R) \;  | \; g \delta_{x,m} = \delta_{x,m} \sigma(g)\},
\end{align}
where $\sigma$ is induced by $\sigma:G(\mathbb{Q}_{q^m}) \to G(\mathbb{Q}_{q^m})$. As explained in \cite[Sec. 2.1.7]{KMPS}, the base change $I_{x,m,p} \otimes \mathbb{Q}_{q^m}$ is naturally identified with the centraliser of the semisimple element $\gamma_{x,m,p}$ in $G(\mathbb{Q}_{q^m})$, and $I_{x,m,p}$ is thus reductive. We similarly define $J_{\delta_{x,m}}$ by its functor of points 
\begin{align}
    R \mapsto \{g \in G(\qpbr \otimes_{\qp} R) \;  | \; g \delta_{x,m} = \delta_{x,m} \sigma(g)\}.
\end{align}

\subsubsection{} Consider the decomposition
\begin{align} \label{eq:ProductDecomposition}
    G^{\mathrm{ad}}=\prod_{i=1}^n G_i
\end{align}
of $G^{\mathrm{ad}}$ into simple groups over $\mathbb{Q}$. Let $\delta_{x,m,i}$ and $\gamma_{x,m,p,i}$ be the images of $\delta_{x,m}$ and $\gamma_{x,m,p}$ in $G_i(\mathbb{Q}_{q^m})$.
\begin{Lem} \label{Lem:ProductIXP}
There is a product decomposition
\begin{align}
    I_{x,m,p}/Z_{G,\qp}= \prod_{i=1}^n I_{x,m,p,i},
\end{align}
where $I_{x,m,p,i}$ represents the functor on $\qp$-algebras sending $R$ to
\begin{align}
    \{g \in G_i(\mathbb{Q}_{q^m} \otimes_{\qp} R) \;  | \; g \delta_{x,m,i} = \delta_{x,m,i} \sigma(g)\}.
\end{align}
Similarly there is a product decomposition
\begin{align}
    J_{\delta_{x,m}}/Z_G \simeq \prod_{i=1}^n J_{\delta_{x,m,i}},
\end{align}
where $J_{\delta_{x,m,i}}$ represents the functor on $\qp$-algebras sending $R$ to
\begin{align}
    \{g \in G_i(\qpbr \otimes_{\qp} R) \;  | \; g \delta_{x,m,i} = \delta_{x,m,i} \sigma(g)\}.
\end{align}
\end{Lem}
\begin{proof}
Consider the commutative diagram
\begin{equation}
    \begin{tikzcd}
    I_{x,m,p} \arrow[r] \arrow[d, hook] & \prod_{i=1}^n I_{x,m,p,i} \arrow[d, hook] \\
    \operatorname{Res}_{\mathbb{Q}_{q^m}/\qp} G \arrow{r} & \prod_{i=1}^n \operatorname{Res}_{\mathbb{Q}_{q^m}/\qp} G_i.
    \end{tikzcd}
\end{equation}
Since the kernel of the bottom map is central and the bottom map is surjective, it follows that the natural map $I_{x,p} \to \prod_{i=1}^n I_{x,m,p,i}$ is surjective. The kernel is given by the intersection of $I_{x,m,p}$ with the kernel of the bottom map and thus has the following functor of points: 
\begin{align}
    R \mapsto \{g \in Z_G(\mathbb{Q}_{q^m} \otimes_{\qp} R)\;  | \; \delta_{x,m,i} = \delta_{x,m,i} \sigma(g)\}.
\end{align}
This forces $g=\sigma(g)$ and so $g \in Z_G(R) \subset Z_G(\mathbb{Q}_{q^m} \otimes_{\qp} R)$. The same proof shows that there is a product decomposition $J_{\delta_{x,m}}/Z_G \simeq \prod_{i=1}^n J_{\delta_{x,m,i}}$.
\end{proof}
Note that $I_{x,m,p,i} \otimes \mathbb{Q}_{q^m}$ can be identified with the centraliser of $\gamma_{x,m,p,i}$ in $G_{i,\mathbb{Q}_{q^m}}$ as in the beginning of Section \ref{Sec:Centralisers}. The centraliser of $\gamma_{x,m,p,i} \in G(\qpbr)$ does not depend on $m$ for $m$ sufficiently divisible, and thus the group $I_{x,m,p}$ does not depend on $m$ for $m$ sufficiently divisible. We will write $I_{x,p}$ for the group $I_{x,m,p}$ for sufficiently divisible $m$ and similarly $I_{x,i,p}$ for the group $I_{x,m,p,i}$. We will identify $I_{x,p} \otimes \qpbr$ with the centraliser of $\gamma_{x,m,p}$ in $G(\qpbr)$ for sufficiently divisible $m$ and similarly identify $I_{x,i,p}$ with the centraliser of $\gamma_{x,m,p,i}$ in $G_i(\qpbr)$. 

\subsubsection{} Let $x \in \shginf(\ovfp)$ and let $\operatorname{Aut}(A_x)$ be the algebraic group over $\mathbb{Q}$ with functor of points
\begin{align}
    \operatorname{Aut}(A_x)(R)=\left(\operatorname{End}(A_x) \otimes_{\mathbb{Z}} R \right)^{\times}.
\end{align}
Following \cite[Sec. 2.1.3]{KMPS}, we define $I_x^p$ to be the largest closed subgroup of $\operatorname{Aut}(A_x)$ that fixes the tensors $s_{\alpha, \afp}$ and $I_x \subset I_x^p$ to be the largest closed subgroup that also fixes the tensors $s_{\alpha, \mathrm{cris}}$. There are natural maps $I_{x,\ql} \to I_{x,\ell}$ for all (including $\ell=p$), see \cite[Sec. 2.1.8]{KMPS} for the $\ell=p$ case. 

The groups $I_{x,\ell}$ are connected reductive subgroups of $G_{\ql}$ and in fact Levi subgroups over $\overline{\mathbb{Q}}_{\ell}$. By \cite[Cor. 2.1.9]{KMPS} for all $\ell$ (including $\ell=p$) the natural map
\begin{align}
    I_{x, \ql} \to I_{x,\ell}
\end{align}
is an isomorphism. This induces closed immersion of groups $I_{x,\qp} \to J_{\delta_{x,m}}$ for some sufficiently divisible $m$. 

\subsection{An assumption}
We will need to assume the following hypothesis to prove our main monodromy theorems in Section \ref{Sec:Monodromy}.
\begin{Hyp} \label{Hyp:Levi}
For all points $x \in \shginf(\ovfp)$ and for sufficiently divisible $m$ depending on $x$, there is an element $\gamma_{x,m} \in G(\qbar)$ that is conjugate to $\gamma_{x,m,\ell}$ in $G(\qlbar)$ for all $\ell$ (including $\ell=p)$. Moreover the $G(\qbar)$-conjugacy class of $\gamma_{x,m}$ is stable under the action of $\operatorname{Gal}(\qbar/\mathbb{Q})$.
\end{Hyp}
If the $G(\qbar)$-conjugacy class of $\gamma_{x,m}$ contains an element of $G(\mathbb{Q})$, then it is clearly Galois stable. However the converse does not necessarily hold. 
\begin{Lem} \label{Lem:HypothesisHyperspecial}
The hypothesis holds when $K_p$ is hyperspecial.
\end{Lem}
\begin{proof}
If $K_p$ is hyperspecial, then \cite[Cor. 2.3.1]{KisinPoints} tells us that there is an element $\gamma_{x,m} \in G(\mathbb{Q})$ that is conjugate to $\gamma_{x,m,\ell}$ in $G(\qlbar)$ for all $\ell$ (including $\ell=p)$. 
\end{proof}

\begin{Rem}
By \cite[Cor. 2.2.14]{KMPS}, an element $\gamma_{x,m} \in G(\mathbb{Q})$ satisfying the requirements of Hypothesis \ref{Hyp:Levi} exists when $G_{\qp}$ is quasi-split and has no factors of type $D$. If $K_p$ is very special, the group $G_{\qp}$ is tamely ramified and satisfies $p \nmid \# \pi_1(G^{\mathrm{der}})$ and $\pi_1(G)_{I}$ is torsion free, where $I \subset \gal(\qpbar/\qp)$ is the inertia group, then the existence of an element $\gamma_{x,m} \in G(\mathbb{Q})$ satisfying the requirements of Hypothesis \ref{Hyp:Levi} follows from Theorem I of \cite{HoftenZhou}

If $K_p$ is a very special parahoric subgroup and the triple $(G,X,K_p)$ is acceptable in the sense of \cite[Def. 5.2.6, Def. 5.2.9]{KisinZhou}, then \cite[Thm. 6.1.4]{KisinZhou} proves the existence of an element $\gamma_{x,m} \in G(\qbar)$ satisfying the requirements of Hypothesis \ref{Hyp:Levi}.
\end{Rem}
\begin{Rem}
When $G_{\qp}$ is not quasi-split, one should probably not expect that the $G(\qbar)$-conjugacy class of $\gamma_{x,m}$ always contains an element of $G(\mathbb{Q})$. This is because CM lifts do not exist in general when $G_{\qp}$ is not quasi-split. However, we expect Hypothesis \ref{Hyp:Levi} to hold in full generality.

For example, let $x \in \shginf(\ovfp)$ be a point corresponding to the good reduction of an abelian variety \emph{defined over a number field} and assume that $p>2$. Then \cite[Thm. 7.2.4]{KisinZhou} tells us that there is an element $\gamma_{x,m} \in G(\qbar)$ satisfying the requirements of Hypothesis \ref{Hyp:Levi}.
\end{Rem}

\subsubsection{} We end by deducing a consequence of Hypothesis \ref{Hyp:Levi} that will be used in Section \ref{Sec:Monodromy}. Let $G^{\ast}$ denote the quasi-split inner form of $G$ over $\mathbb{Q}$ and let $\Psi:G \otimes \qbar \to G^{\ast} \otimes \qbar$ be an inner twisting. This means that every $\tau \in \gal(\qbar/\mathbb{Q})$ satisfies
\begin{align}
    \Psi(\tau(g)) = h_{\tau} \tau(\Psi(g)) h_{\tau}^{-1}
\end{align}
for some element $h_{\tau} \in G^{\ast}(\qbar)$. A direct consequence of this definition is that the image under $\psi$ of a $\gal(\qbar/\mathbb{Q})$-stable $G(\qbar)$-conjugacy class is a $\gal(\qbar/\mathbb{Q})$-stable $G^{\ast}(\qbar)$-conjugacy class.
\begin{Lem} \label{Lem:HypothesisTrue}
Suppose that Hypothesis \ref{Hyp:Levi} holds and let $\gamma_{x,m} \in G(\qbar)$ be the element that is guaranteed to exist by that Hypothesis. Then for sufficiently divisible $m$ the element $\Psi(\gamma_{x,m})$ is $G^{\ast}(\qbar)$-conjugate to an element in $G^{\ast}(\mathbb{Q})$.
\end{Lem}
\begin{proof}
If $m$ is sufficiently divisible, then the centraliser of $\gamma_{x,m}$ is connected because this is true for $\gamma_{x,m,\ell}$ and the formation of centralisers commutes with base change. Since $G^{\ast}$ is quasi-split and the element $\Psi(\gamma_{x,m})$ is semisimple with connected centraliser, we may apply \cite[Thm. 4.7.(2)]{KottwitzConjugacy} which tells us that the $G^{\ast}(\mathbb{Q})$-conjugacy class of $\Psi(\gamma_{x,m})$ contains an element of $G^{\ast}(\mathbb{Q})$.
\end{proof}

\section{Monodromy of Hecke-invariant subvarieties} \label{Sec:Monodromy}
In this section we prove an $\ell$-adic monodromy theorem in the style of Chai \cite{ChaiElladicmonodromy} (c.f. \cites{LXiao,KasprowitzMonodromy}) for prime-to-$p$ Hecke stable subvarieties of Shimura varieties of Hodge type in characteristic $p$. We expect the results in this section to be of independent interest, at least beyond the hyperspecial case that we will use in the rest of this article. \medskip

In Section \ref{Sec:ArithmeticMonodromyI} we establish formal properties of subvarieties $Z$ of Shimura varieties of Hodge type in characteristic $p$ that are stable under prime-to-$p$ Hecke operators. Using techniques from \cite{ChaiElladicmonodromy} we prove that the $\ell$-adic monodromy groups of the universal abelian variety over such $Z$ are normal subgroups of $G_{\ql}$, this is stated as Corollary \ref{Cor:Normal}. 

In Section \ref{Sec:ArithmeticMonodromyII} we use the results from \cite{KMPS} in combination with Hypothesis \ref{Hyp:Levi} to prove Theorem \ref{Thm:ArithmeticMonodromy} and Corollary \ref{Cor:GeometricMonodromy}; the latter is a generalisation of Theorem \ref{Thm:MonodromyIntroduction}. In Section \ref{Sec:PAdicMonodromy} we combine this theorem with results of D'Addezio \cite{d2020monodromy} to deduce results about the $p$-adic monodromy groups of the universal abelian variety over Hecke stable subvarieties.

Finally, in Section \ref{Sec:IrrComp} we prove results about irreducible components of Hecke stable subvarieties in the style of \cite[Prop. 4.5.4]{ChaiElladicmonodromy}. We will not use these results in the rest of this article and so this section can safely be skipped for the reader only interested in the proof of Theorem \ref{Thm:Ordinary}.

\subsection{Arithmetic monodromy groups I} \label{Sec:ArithmeticMonodromyI}
Let the notation be as in Section \ref{Sec:IntModels}. In this section we are going to study arithmetic monodromy groups of Hecke stable subvarieties of $\shg$. For maximal generality, we do not assume that these are defined over $k=\mathbb{F}_q$ and so from now on we will implicitly base change the Shimura variety $\shg$ to an unspecified finite extension of $k$, which we will also denote by $k$.

The morphism $\pi:\shginf \to \shg$ is a pro-\'etale $K^p$-torsor over $\shg$ such that the action of $K^p \subset G(\afp)$ extends to an action of $G(\afp)$. Let $Z \subset \shg$ be a locally closed subscheme and let $\tilde{Z}$ be the inverse image of $Z$ under $\pi$. We say that $Z$ is \emph{stable under the prime-to-$p$-Hecke operators}, or that $Z$ is $G(\afp)$\emph{-stable}, if $\tilde{Z}$ is $G(\afp)$-stable. 

\textbf{For the rest of this section $\ell$ will be used to denote a prime number not equal to $p$.} For such $\ell$ we let $K_{\ell}$ be the image of $K^p$ in $G(\mathbb{Q}_{\ell})$ under the projection $G(\afp) \to G(\mathbb{Q}_{\ell})$. We let 
\begin{align} \label{Def:PiEll}
    \pi_{\ell}:\shginf \times_{K^p} K_{\ell} \to \shg 
\end{align}
be the induced pro-\'etale $K_{\ell}$-torsor. For $Z \subset \shg$ a locally closed subscheme we will write $\tildezl$ for the inverse image of $Z$ under $\pi_{\ell}$. We say that $Z$ is \emph{stable under the $\ell$-adic Hecke operators}, or that $Z$ is $G(\ql)$\emph{-stable}, if $\tildezl$ is $G(\ql)$-stable. When discussing $G(\ql)$-stable $Z$ we will always implicitly work with $\ell \not=p$. If $Z$ is $G(\afp)$-stable, then it is automatically $G(\ql)$-stable for all $\ell \not=p$.

All the results in this section will be stated for smooth $Z$, and the following lemma will be used to reduce to the smooth case in the proof of Theorem \ref{Thm:Ordinary}.
\begin{Lem} \label{Lem:DevSmooth}
Let $Z \subset \shg$ be a locally closed subscheme that is stable under the action of $G(\afp)$ (respectively $G(\ql)$), then the smooth locus $U \subset Z$ is also stable under this action.
\end{Lem}
\begin{proof}
For $g \in G(\afp)$ and $K^p \subset G(\afp)$ there is a finite \'etale correspondence
\begin{equation}
    \begin{tikzcd}
    & \operatorname{Sh}_{G, (K^p \cap g K^p g^{-1}) K_p} \arrow{dl}{p_1} \arrow{dr}{p_2} \\
    \operatorname{Sh}_{G,K^pK_p} & & \operatorname{Sh}_{G, g K^p g^{-1}K_p} \arrow{r}{g} & \operatorname{Sh}_{G,K^pK_p}.
    \end{tikzcd}
\end{equation}
and the assumption that $\tilde{Z}$ is stable under the action of $g$ is equivalent to the statement that the inverse image of $Z$ under $p_1$ is the same as the inverse image of $Z$ under $g \circ p_2$ for all choices of $K^p$. Because all the maps in the diagram are finite \'etale, the same is true for the smooth locus $U$ of $Z$. Therefore the inverse image $\tilde{U}$ of $U$ under $\pi$ is stable under the action of $g \in G(\afp)$.
\end{proof}
\begin{Lem} \label{Lem:ClosureStable}
Let $Z \subset Y \subset \shg$ be locally closed and $G(\afp)$-stable (resp. $G(\ql)$-stable) subvarieties. Then the closure of $Z$ in $Y$ is also stable under $G(\afp)$ (resp. $G(\ql)$).
\end{Lem}
\begin{proof}
This follows in the same way as in the proof of Lemma \ref{Lem:DevSmooth} from the fact that the prime-to-$p$ Hecke correspondences are finite \'etale; indeed finite \'etale maps are open and closed, and thus take closures to closures. 
\end{proof}
\subsubsection{Some general topology} Let $\{X_i\}_{i \in I}$ be a countably indexed cofiltered inverse system of finite type schemes over a field $k$ with surjective affine transition maps. Let $X = \varprojlim_i X_i$ be the inverse limit, it is a nonempty quasicompact scheme by \cite[Lem. 01Z2]{stacks-project}. Recall that for a quasicompact scheme $Y$ there is a profinite topological space $\pi_0(Y)$ of connected components of $Y$.
\begin{Lem} \label{Lem:InverseLimitsConnected}
The natural map
\begin{align} \label{Eq:NaturalLimitMap}
    \pi_0(X) \to \varprojlim_{i} \pi_0(X_i)
\end{align}
is a homeomorphism.
\end{Lem}
\begin{proof}
The left hand side of \eqref{Eq:NaturalLimitMap} is a profinite topological space by \cite[Lem. 0906]{stacks-project} and the right hand side of \eqref{Eq:NaturalLimitMap} is visibly an inverse limit of finite sets. Hence both sides are compact Hausdorff topological spaces and to show that the map is a homeomorphism it suffices to show that it is a bijection.

To show that the natural map is a bijection, we construct an explicit inverse. Any compatible system of connected components $\{V_{i}\}_{i \in I}$ of $\{X_i\}_{i \in I}$ has nonempty and quasicompact inverse limit $V \subset X$ by \cite[Lem. 0A2W]{stacks-project}. To prove that $V$ is connected we suppose that there are nonempty open and closed subsets $W$ and $W'$ of $V$ such that $V=W \coprod W'$. Then $W$ and $W'$ are quasicompact open because $V$ is quasicompact.

Now \cite[Lem. 0A30.(1)]{stacks-project} tells us that we can find $i$ and (nonempty) constructible quasicompact open subsets $Z, Z'$ of $V_{i}$ such that $W$ is the inverse image of $Z$ under $V \to V_i$ and similarly $W'$ is the inverse image of $Z'$ under $V \to V_i$. In particular, the subsets $Z$ and $Z'$ are disjoint nonempty open subsets of $V_{i}$, which gives us a contradiction since $V_{i}$ is connected. 

We have produced a map $\varprojlim_{i} \pi_0(X_i) \to \pi_0(X)$ and it is not hard to check that it is an inverse of the natural map from the lemma; this concludes the proof.
\end{proof}
\begin{Cor} \label{Cor:Continuity}
Let $Z \subset \shg$ be a $G(\afp)$-stable (resp. $G(\ql)$-stable) locally closed subscheme of $\shg$ and let $\tilde{Z}$ be as above. Then $\pi_0(\tilde{Z})$ is equipped with a continuous action of $G(\afp)$ (respectively $\pi_0(\tildezl)$ is equipped with a continuous action of $G(\ql)$). 
\end{Cor}
\begin{proof}
The existence of the action follows from the existence of the action on $\tilde{Z}$ (resp. $\tildezl$). The continuity follows from the continuity of the action of $K^p$ on $\varprojlim_{K^p} Z_{K^p}$ (resp. the continuity of the action of $K_{\ell}$ on $\tildezl$) and Lemma \ref{Lem:InverseLimitsConnected}.
\end{proof}
The following lemma is only a slight generalisation of \cite[Lem. 2.8]{ChaiElladicmonodromy}, but we include a proof for the benefit of the reader.
\begin{Lem} \label{Lem:QuotientTopology}
Let $X$ be a second-countable compact Hausdorff topological space with a transitive and continuous action of a locally profinite topological group $\mathbb{G}$. Let $x \in X$ with stabiliser $\mathbb{G}_x \subset \mathbb{G}$, then the orbit map
\begin{align}
    O:\mathbb{G}/\mathbb{G}_x \to X
\end{align}
is a homeomorphism.
\end{Lem}
\begin{proof}
We can write $\mathbb{G}$ as the increasing union of countably many compact open sets, for example by using finite unions of cosets of a compact open subgroup $\mathbb{K} \subset \mathbb{G}$. Since the quotient map $\mathbb{G} \to \mathbb{G}/\mathbb{G}_x$ is open for any topological group, it follows that $\mathbb{G}/\mathbb{G}_x$ can be written as the increasing union of countably many compact open subsets. 

Since the orbit map is surjective, the topological space $X$ can be written as a countable union of compact subsets $O(U)$ for $U \subset \mathbb{G}/\mathbb{G}_x$ compact open. Because $X$ is second-countable it is metrisable by Urysohn's metrisation theorem and thus the Baire category theorem tells us that there exists a compact open subset $U$ of $\mathbb{G}/\mathbb{G}_x$ such that $O(U)$ contains an open subset $W$ of $X$. 

Choose a compact open subset $V \subset U$ such that $O(V) \subset W$. Then $O:V \to O(V)$ is a continuous bijection between compact Hausdorff topological spaces and hence a homeomorphism. Now note that $\mathbb{G}$ acts transitively on both $\mathbb{G}/\mathbb{G}_x$ and on $X$. Hence by moving around $V$ we see that any point of $y \in \mathbb{G}/\mathbb{G}_x$ has an open neighborhood $V_y$ such that the natural map $O:V_y \to O(V_y)$ is a homeomorphism, and we conclude that $O$ is a homeomorphism.
\end{proof}

\subsubsection{Lie groups over \texorpdfstring{$\ell$}{l}-adic local fields} Recall that a topological group $M$ is called an $\ell$-adic Lie group if it admits the (necessarily unique) structure of an $\ell$-adic Lie group, see \cite[Def. 2.1, Prop. 2.2]{LecturesLieGroups}. If $M$ is an $\ell$-adic Lie group, then by definition there is a finite-dimensional $\ql$-Lie algebra $\operatorname{Lie} M$, an open neighborhood $U \subset \operatorname{Lie} M$ of the identity and an exponential map $\operatorname{Exp}: U \to M$ that is a homeomorphism onto a compact open subgroup of $M$. For example for an algebraic group $H$ over $\ql$ the topological group $H(\ql)$ is an $\ell$-adic Lie group with Lie algebra $\operatorname{Lie} H(\ql) = \operatorname{Lie} H$.
\begin{Lem} \label{Lem:EllAdicLie}
Let $H$ be an algebraic group over $\ql$ and let $M \subset H(\ql)$ be a subgroup closed in the $\ell$-adic topology. Then $M$ is an $\ell$-adic Lie group and the morphism $M \to H(\ql)$ is a morphism of $\ell$-adic Lie groups. Moreover, the induced Lie subalgebra $\operatorname{Lie} M \subset \operatorname{Lie} H(\ql) = \operatorname{Lie} H$ satisfies
\begin{align}
    [\operatorname{Lie} M,\operatorname{Lie} M] = [\operatorname{Lie} \mathbb{M},\operatorname{Lie} \mathbb{M}],
\end{align}
where the bracket notation means the commutator of two Lie subalgebras and where $\mathbb{M}$ is the Zariski closure of $M$.
\end{Lem}
\begin{proof}
The group $M$ is an $\ell$-adic Lie group by \cite[Prop. 2.3]{LecturesLieGroups} and the morphism $M \to H(\ql)$ is a morphism of $\ell$-adic Lie groups by \cite[Prop. 2.2]{LecturesLieGroups}. This implies that there is an induced morphism on Lie algebras $\operatorname{Lie} M \to \operatorname{Lie} H(\ql) = \operatorname{Lie} H$. 

Since $M \subset \mathbb{M}(\ql)$ is Zariski dense, it follows that the smallest algebraic subgroup of $H$ whose Lie algebra contains $\operatorname{Lie} M$ is equal to $\mathbb{M}$; indeed, if there is a smaller algebraic subgroup $\mathbb{M}' \subset \mathbb{M}$ with $\operatorname{Lie} M \subset \operatorname{Lie} \mathbb{M}'$, then we see using the $\ell$-adic exponential map that there is a finite index subgroup of $M$ contained in $\mathbb{M}'(\ql)$. This contradicts the fact that $M$ is Zariski dense in $\mathbb{M}$.

The fact that the smallest algebraic subgroup of $H$ whose Lie algebra contains $\operatorname{Lie} M$ is equal to $\mathbb{M}$ is expressed as $\mathfrak{a}(\operatorname{Lie} M)=\operatorname{Lie} \mathbb{M}$ in the notation of \cite[Sec. 7.1]{Borel}. By \cite[Cor. 7.9]{Borel} we have the following equality of Lie subalgebras of $\operatorname{Lie} H$ 
\begin{align}
    [\operatorname{Lie} M,\operatorname{Lie} M] &= [\mathfrak{a}(\operatorname{Lie} M),\mathfrak{a}(\operatorname{Lie} M)] \\
    &=[\operatorname{Lie} \mathbb{M},\operatorname{Lie} \mathbb{M}]. 
\end{align}
\end{proof}
\begin{Lem} \label{Lem:DenseCompactOpen}
Let $\mathbb{M}$ be a semisimple algebraic group over $\ql$ and let $M \subset \mathbb{M}(\ql)$ be subgroup closed in the $\ell$-adic topology. If $M$ equipped with the subspace topology is compact and $M$ is Zariski dense in $\mathbb{M}$, then $M$ is a compact open subgroup of $\mathbb{M}(\ql)$.
\end{Lem}
\begin{proof}
It follows from Lemma \ref{Lem:EllAdicLie} that $M$ is an $\ell$-adic Lie group, that $M \to \mathbb{M}(\ql)$ is a morphism of $\ell$-adic Lie groups and that the Lie algebra of $M$ is equal to the Lie algebra of $\mathbb{M}$, since $\mathbb{M}$ is semisimple. Now we can use the exponential map for $\ell$-adic Lie groups to show that $M$ contains a compact open subgroup of $\mathbb{M}(\ql)$.
Since $M$ is itself compact, this implies that $M$ is also a compact open subgroup of $\mathbb{M}(\ql)$.
\end{proof}

\subsubsection{} The main theorem of Galois theory for schemes tells us that the category of finite-\'etale covers of a smooth connected scheme $Z$ over $k$ is equivalent to the category of finite sets equipped with a continuous action of $\pi_1^{\text{\'et}}(Z,z)$. Under this equivalence, a finite \'etale cover $f:Y \to Z$ is sent to the finite set $f^{-1}(z)$ equipped with its action of $\pi_1^{\text{\'et}}(Z,z)$. In particular, the set of connected components of $Y$ is in bijection with the set of orbits of $\pi_1^{\text{\'et}}(Z,z)$ on $f^{-1}(z)$.

If $f:Y \to Z$ is a countably indexed inverse limit of finite \'etale covers $f_i:Y_i \to Z$ with surjective transition maps, then we can associate to $f$ the profinite set 
\begin{align}
    f^{-1}(z) = \varprojlim_i f_i^{-1}(z),
\end{align}
equipped with its natural continuous action of $\pi_1^{\text{\'et}}(Z,z)$. By Lemma \ref{Lem:InverseLimitsConnected} it follows that the profinite set of orbits of $\pi_1^{\text{\'et}}(Z,z)$ on $f^{-1}(z)$ is homeomorphic to the topological space of connected components of $Y$.

\subsubsection{} \label{Sec:InfiniteGalois} Now let $Z$ be a smooth $G(\ql)$-stable locally closed subscheme of $\shg$, let $Z^{\circ} \subset Z$ be a connected component of $Z$, and let $z \in \zcirc(\ovfp)$. Let $\pi_{\ell}$ be as in \eqref{Def:PiEll} and write $\tildezl$ for the inverse image of $Z$ under $\pi_{\ell}$ as above; it is stable under the action of $G(\mathbb{Q}_{\ell})$ by assumption. Denote by $\tildezcl$ the inverse image of $\zcirc$ under $\pi_{\ell}$, then $\tildezcl \to \zcirc$ is a profinite \'etale $K_{\ell}$-torsor.

By the Galois theory for schemes discussed above, the cover $\pi_{\ell}:\tildezcl\to \zcirc$ corresponds to the profinite set $\pi_{\ell}^{-1}(z)$ equipped with its natural action of $\pi_1^{\text{\'et}}(Z^{\circ},z)$. In particular, the set of connected components of $\tildezcl$ corresponds to the set of orbits of $\pi_1^{\text{\'et}}(Z^{\circ},z)$ on $\pi_{\ell}^{-1}(z)$.

Choose an element $\tilde{z} \in \pi_{\ell}^{-1}(z)$. Then using the simply transitive action of $K_{\ell}$ on $\pi_{\ell}^{-1}(z)$ we can identify $\pi_{\ell}^{-1}(z)$ with $K_{\ell}$; under this identification the chosen element $\tilde{z}$ is send to $1 \in K_{\ell}$. This defines a continuous group homomorphism
\begin{align}
    \rho_{\ell}:\pi_1^{\text{\'et}}(Z^{\circ},z) \to K_{\ell},
\end{align}
whose conjugacy class does not depend on the choice of $\tilde{z}$. Let $y \in \pi_0(\tildezcl)$ be the connected component containing $\tilde{z}$. Then the stabiliser of $y$ in $K_{\ell}$ is equal to the image of $\rho_{\ell}$. 

Let $P_y \subset G(\mathbb{Q}_{\ell})$ be the stabiliser of $y$ in $G(\mathbb{Q}_{\ell})$. It is a closed topological subgroup by the continuity of the action and the fact that $\pi_0(\tildezl)$ is Hausdorff. Its intersection with $K_{\ell}$ gives us the stabiliser of $y$ in $K_{\ell}$. The action map gives us a continuous map
\begin{align}
    G(\ql)/P_y \to \pi_0(\tildezl) \\
    g \mapsto g \cdot y
\end{align}
with image the orbit $\operatorname{Orb}(y)$ of $y$. 
\begin{Lem} \label{Lem:OrbitClosed}
The orbit $\operatorname{Orb}(y)$ is open and closed inside of $\pi_0(\tildezl)$. Moreover the orbit map induces a homeomorphism $G(\ql)/P_y \simeq \operatorname{Orb}(y)$; in particular, $G(\mathbb{Q}_{\ell})/P_y$ is compact. 
\end{Lem}
\begin{proof}
The identification
\begin{align}
    \pi_0(\tildezl) / K_{\ell} \simeq \pi_0(Z).
\end{align}
tells us that there are finitely many $K_{\ell}$-orbits on $\pi_0(\tildezl)$, and that each of them is open and closed. The $G(\mathbb{Q}_{\ell})$-orbit of a point $y$ is then a union of finitely many $K_{\ell}$-orbits, and thus also open and closed. Lemma \ref{Lem:OrbitClosed} shows that $\operatorname{Orb}(y)$ is open and closed inside a second-countable profinite topological space. Therefore $\operatorname{Orb}(y)$ is profinite and second-countable. The result now follows from Lemma \ref{Lem:QuotientTopology}.
\end{proof}
Let $M$ be the image of $\rho_{\ell}$ and let $\mathbb{M}$ be the neutral component of its Zariski closure inside $G(\mathbb{Q}_{\ell})$. Let $\rho_{\ell,\mathrm{geom}}$ be the restriction of $\rho_{\ell}$ to 
\begin{align}
    \pi_1^{\text{\'et}}(Z^{\circ}_{\ovfp},z) \subset \pi_1^{\text{\'et}}(Z^{\circ},z),
\end{align}
let $M_{\mathrm{geom}}$ be its image and let $\mathbb{M}_{\mathrm{geom}}$ be the neutral component of its Zariski closure inside $G(\ql)$.
\begin{Lem} \label{Lem:Reductive}
The groups $\mathbb{M}$ and $\mathbb{M}_{\mathrm{geom}}$ are connected reductive groups over $\ql$.
\end{Lem}
\begin{proof}
There is a short exact sequence (e.g. by \cite[Prop. 3.2.7]{d2020monodromy})
\begin{align} \label{Eq:ShortExactSequenceGeometricArithmetic}
    1 \to \mathbb{M}_{\mathrm{geom}}' \to \mathbb{M} \to Q \to 1,
\end{align}
where $Q$ is a commutative algebraic group of multiplicative type and where $\mathbb{M}_{\mathrm{geom}}'$ is a closed subgroup of $\mathbb{M}$ with neutral component given by $\mathbb{M}_{\mathrm{geom}}$. In particular, it follows that $\mathbb{M}$ is reductive if $\mathbb{M}_{\mathrm{geom}}$ is reductive. The representation 
\begin{align}
    \pi_1^{\text{\'et}}(Z^{\circ},z) \to K_{\ell} \to G(\ql) \to \operatorname{GL}(V)(\mathbb{Q}_{\ell})
\end{align}
is the monodromy representation of the (rational) $\ell$-adic Tate module of the abelian scheme $\pi:A \to \shg$ coming from the Hodge embedding $\shg \to \shgsp$. This is Tate-module is an $\ell$-adic sheaf $\mathcal{F}_0$ on $Z^{\circ}$ which is pure of weight one. Then \cite[Thm. 3.4.1.(iii)]{WeilII} tells us that the basechange $\mathcal{F}$ of $\mathcal{F}_0$ to $Z^{\circ}_{\ovfp}$ is semi-simple. This base change corresponds to the composition of $\rho_{\ell,\mathrm{geom}}$ with $K_{\ell} \to G(\ql) \to \operatorname{GL}(V)(\ql)$. Now \cite[Cor. 1.3.9]{WeilII} tells us that $\mathbb{M}_{\mathrm{geom}}$ is a semi-simple algebraic group, and thus that it is reductive. \smallskip
\end{proof}
\subsubsection{}
Let $\mathbb{N}$ be the normaliser of $\mathbb{M}$ in $G_{\ql}$ and let $\mathbb{N}^{\circ}$ be its neutral component. The group $\mathbb{N}^{\circ}$ is a connected reductive group because we are working with reductive groups in characteristic zero by e.g. \cite[Prop. A.8.12]{ConradGabberPrasad}.
\begin{Lem}
The group $P_y$ is contained in $\mathbb{N}$.
\end{Lem}
\begin{proof}
Let $\gamma \in P_y$, then we want to show that $\gamma$ normalises $\mathbb{M}$. If $V$ is a compact open subgroup of $G(\ql)$ contained in $K_{\ell}$, then $V \cap P_y \subset M$. Moreover for every $\gamma \in P_y$ we can find an open subgroup $U \subset G(\ql)$ such that $\gamma U \gamma^{-1} \subset  V$. For example, we can just take the intersection of $V$ with $\gamma V \gamma^{-1}$. For such $U$ the open subgroup $M \cap U$ of $M$ satisfies $\gamma (M \cap U) \gamma^{-1} \subset M$. 

Since conjugation by $\gamma$ is a homeomorphism in the Zariski topology, we see that the Zariski closure of $M \cap U$ is moved under conjugation by $\gamma$ into the Zariski closure of $M$. But since $M \cap U$ is an open subgroup of $M$ it is also a closed subgroup and thus compact and thus of finite index in $M$. This means that the Zariski closure of $M \cap U$ and the Zariski closure of $M$ have the same identity component, both of which are equal to $\mathbb{M}$. Since conjugation preserves $1$, this mean it sends $\mathbb{M}$ to $\mathbb{M}$.
\end{proof}
\begin{Cor} \label{Cor:Normal}
The Zariski closure $\mathbb{M}$ of $M$ is a normal subgroup of $G_{\ql}$.
\end{Cor}
\begin{proof}
The group $G/\mathbb{N}(\ql)$ is compact, because it is a quotient of $G(\ql)/P_y$ which is compact. Since $\mathbb{N}^{\circ}(\ql)$ is finite index in $\mathbb{N}(\ql)$, it follows that $G(\ql)/\mathbb{N}^{\circ}(\ql)$ is also compact. Since $\mathbb{N}^{\circ}$ is connected it follows from \cite[Prop. 8.4, Prop. 9.3]{BorelTits} that it contains a parabolic subgroup of $G_{\ql}$ and because it is reductive it follows that $\mathbb{N}^{\circ}=G_{\ql}$. Therefore $\mathbb{N}^{\circ}=\mathbb{N}$ and we find that $\mathbb{M}$ is a normal subgroup of $G_{\ql}$.
\end{proof}
\subsection{Arithmetic monodromy groups II} \label{Sec:ArithmeticMonodromyII} So far we have not excluded the possibility that $\mathbb{M}$ is contained in the center of $G_{\ql}$. In fact, this happens when $Z$ is the supersingular locus inside the modular curve. Thus we will need additional assumptions on $Z$ to prove that $\mathbb{M}$ is not central. 

We will show, using the results of \cite{KMPS}, that if $Z$ contains a point $x \in Z(\mathbb{F}_{q^m})$ not contained in the smallest Newton stratum, then the image of $\operatorname{Frob}_x$ under $\rho_{\ell}$ is noncentral. If $G_{\ql}^{\mathrm{ad}}$ were a simple group over $\ql$, then this would force $\mathbb{M}$ to contain $G^{\mathrm{der}}_{\ql}$. However $G^{\mathrm{ad}}$ is generally not a simple group over $\mathbb{Q}$, and even if it were simple then there would generally be no primes $\ell$ where $G_{\ql}^{\mathrm{ad}}$ is simple. To deal with these issues, we will make use of Hypothesis \ref{Hyp:Levi}.

\subsubsection{} Recall that for a point $x \in Z^{\circ}(\mathbb{F}_q^m)$, there is a Frobenius element
\begin{align}
    \operatorname{Frob}_x \in \pi_1^{\text{\'et}}(Z^{\circ},z),
\end{align}
whose image under $\rho_{\ell}$ is the element $\gamma_{x,m,\ell} \in G(\ql)$ from Section \ref{Sec:Centralisers}. If $m$ is sufficiently divisible, then its centraliser is equal to the group $I_{x,\ell}$. 

\subsubsection{} \label{Sec:ProductDecompositionKottwitzSet} The decomposition $G^{\mathrm{ad}}=\prod_{i=1}^n G_i$ of \eqref{eq:ProductDecomposition} induces maps 
\begin{align} \label{eq:AdmissibleAdjoint}
    B(G_{\qp}) \to B(G^{\mathrm{ad}}_{\qp}) \to \prod_{i=1}^n B(G_{i,\qp}).
\end{align}
For an element $[b] \in B(G_{\qp})$ we will write $[b_i]$ for its image in $B(G_{i,\qp})$ under this map. Recall \cite[Def. 5.3.2]{KretShin} that an element $[b] \in \bgmu$ is called \emph{$\mathbb{Q}$-non-basic} if $[b_i]$ is non-basic for all $i$. A Newton stratum $\shgb$ is called \emph{$\mathbb{Q}$-non-basic} if $[b]$ is $\mathbb{Q}$-non-basic.
\begin{Prop} \label{Prop:NonBasicNonZero}
Let $x \in Z^{\circ}(\mathbb{F}_q^m)$ for some sufficiently divisible $m$ and let $[b]=[b_x] \in \bgmu$. Assume that Hypothesis \ref{Hyp:Levi} holds. If $[b_{i}]$ is non-basic, then the image of $\rho_{\ell}(\operatorname{Frob}_{x,m})$ under
\begin{align}
    G(\mathbb{Q}_{\ell}) \to G^{\mathrm{ad}}(\mathbb{Q}_{\ell}) \to G_i(\mathbb{Q}_{\ell})
\end{align}
is nontrivial. Moreover, the image of $\rho_{\ell}(\operatorname{Frob}_{x,m})$ in each simple factor of $G_{i,\ql}$ over $\ql$ is nontrivial.
\end{Prop}
\begin{proof}
Let $m$ be sufficiently divisible and let $\gamma_{x,m} \in G(\qbar)$ be the element guaranteed to exist by Hypothesis \ref{Hyp:Levi}. Let $G^{\ast}$ denote the quasi-split inner form of $G$ over $\mathbb{Q}$ and let $\Psi:G \otimes \qbar \to G^{\ast} \otimes \qbar$ be an inner twisting. Then by Lemma \ref{Lem:HypothesisTrue} there is an element $\gamma_{x,m}' \in G^{\ast}(\mathbb{Q})$ that is conjugate to $\Psi(\gamma_{x,m})$ in $G^{\ast}(\qbar)$. We will write $I_x' \subset G^{\ast}$ for the centraliser of $\gamma_{x,m}'$. Recall that by Hypothesis \ref{Hyp:Levi} for all $\ell$ (including $\ell=p$) the element $\Psi^{-1}(\gamma_{x,m})$ is conjugate to $\gamma_{x,m,\ell}$ in $G(\qlbar$). 

By the classification of adjoint algebraic groups we can find number fields\footnote{By the classification of Shimura varieties of abelian type in \cite[Appendix B]{Milne}, each $F_i$ is totally real.} $F_1, \cdots, F_n$ and absolutely simple adjoint algebraic groups $H_i$ over $F_i$ for each $i=1, \cdots, n$ such that
\begin{align}
    G^{\mathrm{ad}} = \prod_{i=1}^n \operatorname{Res}_{F_i/\mathbb{Q}} H_i = \prod_{i=1}^n G_i.
\end{align}
We have a similar decomposition for $G^{\ast, \mathrm{ad}}$ with $H_i$ replaced by its quasi-split inner form $H_i^{\ast}$ and we will write $G_i^{\ast}$ for the restriction of scalars from $F_i$ to $\mathbb{Q}$ of $H_i^{\ast}$.

Let $\gamma_{x,m,i} \in G_i^{\ast}(\mathbb{Q})=H_i^{\ast}(F_i)$ be the image of $\gamma_{x,m}$ and let $C_{x,i} \subset H_i^{\ast}$ be its centraliser. Then there is a product decomposition
\begin{align}
    I_x'/Z_G \simeq \prod_{i=1}^n \operatorname{Res}_{F_i/\mathbb{Q}} C_{x,i} = \prod_{i=1}^n I_{x,i}'.
\end{align}
Let $b=b_x \in G(\qpbr)$ for sufficiently divisible $m$ be as in Section \ref{Sec:NewtonStratification}. Then Lemma \ref{eq:ProductDecomposition} shows that there is a product decomposition $J_b/Z_{G,\qp} \simeq \prod_{i=1}^n J_{b_i}$, where each $J_{b,i}$ is the twisted centraliser of the image $b_i$ of $b$ in $G_i(\qpbr)$. Moreover the natural inclusion $I_{x,p} \to J_b$ induces closed immersions $I_{x,i,p} \to J_{b_i}$. As in \cite[Sec. 1.1.4]{KMPS}, there is an inclusion $J_{b,i,\qpbr} \to G_{i,\qpbr}$, identifying its image with the centraliser $M_{\nu_{b_i}}$ of the fractional cocharacter $\nu_{b_i}$ of $G_i \otimes \qpbr$ attached to $b_i$. If $[b_i]$ is not basic, then $\nu_{b_i}$ is not central and so $\operatorname{Dim} J_{b,i} = \operatorname{Dim} M_{\nu_{b_i}} < \operatorname{Dim} G_i$ and so $\operatorname{Dim} I_{x,i,p} < \operatorname{Dim} G_i$. \smallskip

The subgroup $$\Psi^{-1}(I_{x,\qpbar}') \subset G_{\qpbar}$$ can be identified with the centraliser of $\gamma_{x,m,p}$ for $m$ sufficiently divisible, and it follows that $$\Psi^{-1}(I_{x,i,\qpbar}') \subset G_{i,\qpbar}$$ can be identified with the centraliser of the image of $\gamma_{x,m,p,i}$ in $G_{i,\qpbar}$. In particular, the group $I_{x,i,\qpbar}'$ is conjugate to $I_{x,i,p,\qpbar}$ and therefore of the same dimension.

The upshot of the above discussion is that $\operatorname{Dim} I_{x,i}'< \operatorname{Dim} G_i$ if $b_i$ is not basic. It follows that the inclusion $I_{x,i,\ql} \subset G_{i,\ql}$ is not an equality for $\ell \not=p$ and thus that the image of $\rho_{\ell}(\operatorname{Frob}_{x,m})$ in $G_i(\ql)$ is nontrivial.

To deduce the last statement of Proposition \ref{Prop:NonBasicNonZero}, we note that it suffices to show that the image of $\gamma_{x,m,i}$ in every simple factor of $G_{i,\ql}$ is noncentral. For this, we fix $i$ and a prime number $\ell$.

Then we can write $G^{\ast}_{i,\ql}$ as a product indexed by primes $\mathfrak{p}$ of $F_i$ dividing $\ell$
\begin{align}
    G^{\ast}_{i,\ql} &= \prod_{\mathfrak{p} \mid \ell} \operatorname{Res}_{F_{i,\mathfrak{p}}/\ql} H^{\ast}_{i, F_{i,\mathfrak{p}}}.
\end{align}
The element $\gamma_{x,m,i}$ is noncentral in $H_i^{\ast}(F_i)$ and thus also noncentral in $H_i^{\ast}(F_{i,\mathfrak{p}})$ for all primes $\mathfrak{p}$ of $F_i$ dividing $\ell$, and thus we are done. 
\end{proof}
\begin{Rem}
When $b_i$ is basic then the inclusion $I_{x,i,\qp} \subset J_{b,i}$ should be an equality and the image of $\rho_{\ell}(\operatorname{Frob}_{x,m})$ in $G_i(\ql)$ should be trivial for all $\ell \not=p$. This is true when $K_p$ is very special, see the proof of \cite[Prop. 5.2.10]{HeZhouZhu}.
\end{Rem}
We now state and prove our main arithmetic monodromy theorem.
\begin{Thm} \label{Thm:ArithmeticMonodromy}
Let $Z \subset \shg$ be a smooth $G(\ql)$-stable locally closed subvariety. Let $\zcirc \subset Z$ be a connected component and choose a point $z \in \zcirc(\ovfp)$. Let $\mathbb{M}$ be the neutral component of the Zariski closure of the image of
\begin{align}
    \rho_{\ell}:\pi_1^{\text{\'et}}(Z^{\circ},z) \to K_{\ell} \to G(\mathbb{Q}_{\ell}).
\end{align}
Assume that Hypothesis \ref{Hyp:Levi} holds. Then $\mathbb{M}$ is a normal subgroup of $G_{\mathbb{Q}_{\ell}}$ surjecting onto $G_{i,\mathbb{Q}_{\ell}}$ for all $i$ such that there is a point $x \in \zcirc(\ovfp)$ with $b_{x,i}$ non-basic.
\end{Thm}
When $(G,X)=(\mathcal{G}_V,\mathcal{H}_V)$, then this is closely related to \cite[Cor. 3.5]{ChaiElladicmonodromy}.
\begin{proof}
Corollary \ref{Cor:Normal} proves that $\mathbb{M} \subset G_{\ql}$ is a normal subgroup. For $x \in \zcirc(\ovfp)$ we have the Frobenius element $\rho_{\ell}(\operatorname{Frob}_{x,m})=\gamma_{x,m,\ell} \in M$, for all sufficiently divisible $m$, where we recall that $M$ is the image of $\rho_{\ell}$. Thus $\gamma_{x,m,\ell}$ is contained in the Zariski closure of $M$ and thus after replacing $m$ by a power we may assume that $\gamma_{x,m,\ell} \in \mathbb{M}$, the neutral component of the Zariski closure of $M$. \smallskip

If $x \in \zcirc(\ovfp)$ is a point with $b_{x,i}$ non-basic, then Proposition \ref{Prop:NonBasicNonZero} tells us that the image of $\gamma_{x,m,\ell}$ in $G_{i,\ql}$ is nonzero and thus the image of $\mathbb{M}$ in $G_{i,\ql}$ is nonzero. This image is moreover normal, so to show that it is equal to $G_{i,\ql}$ it suffices to show that it maps nontrivially to every simple factor of $G_{i,\ql}$ over $\ql$. But this follows from the last line of the statement of Proposition \ref{Prop:NonBasicNonZero}.
\end{proof}
\begin{Cor} \label{Cor:GeometricMonodromy}
Let $Z \subset \shgov$ be a smooth $G(\ql)$-stable locally closed subscheme as before, let $\zcirc \subset Z$ be a connected component and fix $z \in \zcirc(\ovfp)$. Let $\mathbb{M}_{\mathrm{geom}}$ be the neutral component of the Zariski closure of the image of
\begin{align}
    \pi_1^{\text{\'et}}(Z^{\circ}_{\ovfp},z) \to K_{\ell} \to G(\mathbb{Q}_{\ell}).
\end{align}
Assume that Hypothesis \ref{Hyp:Levi} holds. Then $\mathbb{M}_{\mathrm{geom}}$ is a normal subgroup of $G_{\mathbb{Q}_{\ell}}$ surjecting onto $G_{i,\mathbb{Q}_{\ell}}$ for all $i$ such that there is a point $x \in \zcirc(\ovfp)$ with $b_{x,i}$ non-basic. If we can find such a point for all $i$, then $\mathbb{M}_{\mathrm{geom}}=G_{\ql}^{\mathrm{der}}$.
\end{Cor}
\begin{proof}
The subscheme $Z$ is defined over a finite extension of $k$, and so we can speak of its arithmetic monodromy group $\mathbb{M}$. Theorem \ref{Thm:ArithmeticMonodromy} tells us that $\mathbb{M}$ surjects onto $G_{i,\mathbb{Q}_{\ell}}$ for all $i$ such that there is a point $x \in \zcirc(\ovfp)$ with $b_{x,i}$ non-basic. We now claim that  $\mathbb{M}_{\mathrm{geom}}$ and $\mathbb{M}$ have the same image in  $G_{i,\mathbb{Q}_{\ell}}$ for all such $i$. \smallskip 

It follows from the short exact sequence \eqref{Eq:ShortExactSequenceGeometricArithmetic} that $\mathbb{M}_{\mathrm{geom}} \subset \mathbb{M}$ is a normal subgroup with abelian cokernel. Let $\mathbb{M}_{\mathrm{geom},i}$ be the image of $\mathbb{M}_{\mathrm{geom}}$ in $G_{i,\ql}$ and let $\mathbb{M}_i$ be the image of $\mathbb{M}$ in $G_{i,\ql}$. Then $\mathbb{M}_{\mathrm{geom},i} \subset \mathbb{M}_{i}$ is a normal subgroup with abelian cokernel. Given an integer $i$ with $1 \le i \le n$ such that there is a point $x \in \zcirc(\ovfp)$ with $b_{x,i}$ non-basic, then $\mathbb{M}_{i} = G_{i,\mathbb{Q}_{\ell}}$ and therefore $\mathbb{M}_{i}$ has no nontrivial abelian quotients. Thus it follows that the inclusion $\mathbb{M}_{\mathrm{geom},i} \subset \mathbb{M}_{i}$ is an equality.

If we can find a point $x$ with $b_{x,i}$ non-basic for all $i$, then $\mathbb{M}_{\mathrm{geom}}$ surjects onto $G^{\mathrm{ad}}_{\ql}$ and is moreover semi-simple by \cite[Cor. 1.3.7]{WeilII}. It must therefore be equal to $G_{\ql}^{\mathrm{der}}$.
\end{proof}
\subsection{\texorpdfstring{$p$}{p}-adic monodromy groups} \label{Sec:PAdicMonodromy}
In this subsection we record a consequence of Theorem \ref{Thm:ArithmeticMonodromy} in combination with the main results of \cites{d2020monodromy, DAddezioII}.

Recall the following notions from \cite[Sec. 2.2]{d2020monodromy}. Write $F\mathrm{-Isoc}(S)$ for the $\qp$-linear Tannakian category of $F$-isocrystals over a smooth finite type scheme $S$ over $\ovfp$ and write $F\mathrm{-Isoc}^{\dagger}(S)$ for the $\qp$-linear Tannakian category of overconvergent $F$-isocrystals over $S$. There is a natural fully faithful embedding $F\mathrm{-Isoc}^{\dagger}(S) \subset F\mathrm{-Isoc}(S)$ which sends an overconvergent $F$-isocrystal $\mathcal{M}^{\dagger}$ to the underlying $F$-isocrystal $\mathcal{M}$. Similarly we write $\mathrm{Isoc}^{\dagger}(S)$ and $\mathrm{Isoc}(S)$ for the $\qpbr$-linear category of (overconvergent) isocrystals over $S$. There are natural faithful forgetful functors from (overconvergent) $F$-isocrystals to (overconvergent) isocrystals.

\subsubsection{} The morphism $\shg \to \shgsp$ gives us an abelian scheme $\pi:\mathcal{A} \to \shg$ and we consider the $F$-isocrystal
\begin{align}
    \mathcal{M}:=R^1  \pi_{\ast} \mathcal{O}_{\mathrm{cris},A}[1/p],
\end{align}
which is overconvergent by \cite[Thm. 7]{Etesse}. Then \cite[Cor. 1.3.13]{KMPS} proves that there is an exact $\qp$-linear tensor functor (the \emph{$p$-adic realisation functor})
\begin{align} \label{Eq:FCrystalGStructure}
    \operatorname{Rel}_p:\operatorname{Rep}_{\qp} G \to F\mathrm{-Isoc}(\shg)
\end{align}
such that the representation $G_{\qp} \to \mathcal{G}_V \to \operatorname{GL}_V$ coming from the choice of Hodge embedding is sent to the $F$-isocrystal $\mathcal{M}$. 
\begin{Lem}
This morphism factors via an exact $\qp$-linear tensor functor
\begin{align} \label{Eq:FCrystalGStructureOC}
    \operatorname{Rel}_p:\operatorname{Rep}_{\qp} G \to F\mathrm{-Isoc}^{\dagger}(\shg),
\end{align}
which we will also denote by $\operatorname{Rel}_p$.
\end{Lem}
\begin{proof}
Recall that we have chosen finitely many tensors $s_{\alpha} \in V^{\otimes}$ cutting out $G$ inside $\operatorname{GL}_V$, which correspond to morphisms $s_{\alpha}:1 \to V^{\otimes}$ in $\operatorname{Rep}_{\qp} G$. Applying the tensor functor
\begin{align}
    \operatorname{Rel}_p:\operatorname{Rep}_{\qp} G \to F\mathrm{-Isoc}(\shg)
\end{align}
we get morphism of $F$-isocrystals
\begin{align} \label{Eq:ConvergentOC}
    s_{\alpha}:1 \to \mathcal{M}^{\otimes}.
\end{align}
Since $\mathcal{M}$ is overconvergent and the category of overconvergent isocrystals is stable under tensor products and duals and direct sums by \cite[Rem. 2.3.3.(iii)]{Berthelot}, it follows that $\mathcal{M}^{\otimes}$ is also overconvergent. Thus the morphisms \eqref{Eq:ConvergentOC} live in the full subcategory $F\mathrm{-Isoc}^{\dagger}(\shg)$ and by \cite[Lem. 4.1.8]{KisinShinZhu} we get a $\qp$-linear exact tensor functor
\begin{align}
    \operatorname{Rel}_p:\operatorname{Rep}_{\qp} G \to F\mathrm{-Isoc}^{\dagger}(\shg)
\end{align}
taking $V$ to $\mathcal{M}^{\dagger}$. 
\end{proof}
Given a smooth locally closed subscheme $Z \subset \shgov$ and a point $z \in Z(\ovfp)$, there are monodromy groups
\begin{align}
    \operatorname{Mon}(Z, \mathcal{M},z) \subset \operatorname{Mon}(Z, \mathcal{M}^{\dagger},z),
\end{align}
which are algebraic groups over $\qpbr$, see the introduction of \cite{DAddezioII}. They are defined to be the Tannakian groups corresponding to the smallest Tannakian subcategory of $\mathrm{Isoc}(Z)$ respectively $\mathrm{Isoc}^{\dagger}(Z)$ containing $\mathcal{M}$, via the fiber functor $\omega_z$
\begin{align}
    \omega_z:\mathrm{Isoc}(Z) \to \mathrm{Isoc}(\ovfp)= \operatorname{Vect}_{\qpbr}.
\end{align}
We will often omit the chosen point $z$ from the notation since the monodromy group does not depend on $z$ up to isomorphism. 

Applying Tannakian duality to \eqref{Eq:FCrystalGStructure} and \eqref{Eq:FCrystalGStructureOC} we get inclusions
\begin{align}
    \operatorname{Mon}(Z, \mathcal{M},z) \subset \operatorname{Mon}(Z, \mathcal{M}^{\dagger},z) \subset G_{\qpbr}.
\end{align}
\begin{Cor} \label{Cor:padicMonodromyI}
Let $Z \subset \shg$ be a smooth locally closed subscheme and assume that there is a prime $\ell \not=p$ such that $Z$ is $G(\ql)$-stable. Suppose that $\zcirc$ contains a point $x$ such that $b_{x,i}$ is non-basic for all $i$. If Hypothesis \ref{Hyp:Levi} holds, then there is an equality of subgroups of $G_{\qpbr}$
\begin{align}
    \operatorname{Mon}(\zcirc, \mathcal{M}^{\dagger}) = G^{\mathrm{der}}_{\qpbr}.
\end{align}
\end{Cor}
\begin{proof}
Let $\ell$ be as in the assumptions of the statement of the Corollary. Then it follows from Theorem \ref{Cor:GeometricMonodromy} that the $\ell$-adic monodromy group $\mathbb{M}_{\mathrm{geom}}$ of the abelian variety over $Z$ is equal to $G^{\mathrm{der}}_{\ql}$. It follows from \cite[Thm. 1.2.1]{d2020monodromy} (c.f. \cite{PalMonodromy}) that there is an isomorphism of algebraic groups
\begin{align}
    \operatorname{Mon}(\zcirc, \mathcal{M}^{\dagger}) \otimes \qpbar \simeq G^{\mathrm{der}} \otimes \qpbar.
\end{align}
Therefore $\operatorname{Mon}(\zcirc, \mathcal{M}^{\dagger}) \subset G \otimes \qpbr$ is a subgroup which is isomorphic to $G^{\mathrm{der}}$ over $\qpbar$. It follows that $\operatorname{Mon}(\zcirc, \mathcal{M}^{\dagger})$ is equal to its own derived subgroup and therefore contained in $G^{\mathrm{der}} \otimes \qpbr$. This inclusion has to be an isomorphism for dimensions reasons, because both groups are connected.
\end{proof}
From now on we will assume that $Z$ is contained in a single Newton stratum $\shgb$ of $\shg$. This means that for every representation $W$ of $G_{\qp}$ the Newton polygon of $\operatorname{Rel}_p(W)$ is constant. As explained in \cite[Sec. 2.2]{DAddezioII} (c.f. \cite{KatzSlope}), this means that $\operatorname{Rel}_p(W)$ admits a (unique) slope filtration $\operatorname{Rel}_p(W)_{\bullet}$. There is an induced slope filtration on $\omega(\operatorname{Rel}_p(W))=W \otimes \qpbr$, which gives a fractional cocharacter $\lambda_W$ of $\operatorname{GL}_{W,\qpbr}$. Since this construction is functorial in $W$, it defines a fractional cocharacter $\lambda$ of $G_{\qpbr}$. 

Recall from \cite[Sec. 1.1.2]{KMPS} that associated to $[b] \in B(G_{\qp})$ there is a conjugacy class of fractional cocharacters $\{\nu_{[b]}\}$ of $G_{\qpbr}$ that is defined over $G_{\qp}$.\footnote{The conjugacy class $\{\nu_{[b]}\}$ is defined over $\qp$, but it is not necessarily true that $\{\nu_{[b]}\}$ has a representative defined over $\qp$ unless $G_{\qp}$ is quasi-split}
\begin{Lem} \label{Lem:WhichFractionalCocharacter}
The conjugacy class of $\lambda$ agrees with $\{\nu_{[b]}\}$.
\end{Lem}
\begin{proof}
To identify the conjugacy class of the cocharacter $\lambda$, it suffices to identify the conjugacy class of $\lambda$ composed with all representations of $G_{\qp}$. Then the lemma comes down proving that the (conjugacy class of the) Newton cocharacter of the isocrystal $\operatorname{Rel}_p(W)$ corresponding to a representation $W$ of $G_{\qp}$ given by $r:G_{\qp} \to \operatorname{GL}(W)$, is given by $\{r \circ \nu_{[b]}\}$. But this is true by construction of the Newton cocharacter $\nu_{[b]}$ associated to $[b]$, see \cite[Sec. 4]{Kottwitz1}.
\end{proof}
Under our assumption that $Z$ is contained in a single Newton stratum $\shgb$ of $\shg$ we note that the monodromy group
\begin{align}
    \operatorname{Mon}(\zcirc, \mathcal{M}) \subset G_{\qpbr}
\end{align}
of a connected component $\zcirc$ of $Z$ is contained in the parabolic subgroup $P(\lambda)$ associated to $\lambda$, as explained in \cite[Sec. 4.1]{DAddezioII}. 
\begin{Cor} \label{Cor:padicMonodromyII}
Let $Z \subset \shg$ be a smooth locally closed subscheme and assume that there is a prime $\ell \not=p$ such that $Z$ is $G(\ql)$-stable. Let $Z^{\circ}$ be a connected component of $Z$ and suppose that $\zcirc$ is contained in a single $\mathbb{Q}$-non-basic Newton stratum $\shgb$. If Hypothesis \ref{Hyp:Levi} holds, then the $p$-adic monodromy group
\begin{align}
    \operatorname{Mon}(\zcirc, \mathcal{M}) \subset \operatorname{Mon}(\zcirc, \mathcal{M}^{\dagger})=G^{\mathrm{der}}_{\qpbr} \subset G_{\qpbr}
\end{align}
is equal to the intersection
\begin{align}
    G^{\mathrm{der}}_{\qpbr} \cap P(\lambda).
\end{align}
In particular, the unipotent radical of $\operatorname{Mon}(\zcirc, \mathcal{M})$ is isomorphic to the unipotent radical of the parabolic subgroup $P_{\nu_{[b]}}$ of $G_{\qpbr}$, for some representative $\nu_{[b]}$ of $\{\nu_{[b]}\}$.
\end{Cor}
\begin{proof}
The first assertion is a direct consequence of Corollary \ref{Cor:padicMonodromyI} and \cite[Thm 1.1.1]{DAddezioII}. The second assertion follows from Lemma \ref{Lem:WhichFractionalCocharacter}
\end{proof}
\subsection{Irreducible components of Hecke stable subvarieties} \label{Sec:IrrComp}
In this section we will study irreducible components of Hecke stable subvarieties and prove results in the style of \cite[Prop. 4.5.4]{ChaiElladicmonodromy}.\footnote{Our results do not literally generalise Chai's results because he works with $\operatorname{Sp}_{2g}(\afp)$-stable subvarieties while we work with $\operatorname{GSp}_{2g}(\afp)$-stable subvarieties.}The results proved in this section will not be used in the rest of this article, but they will be used to prove irreducibility results for EKOR strata in an upcoming version of \cite{HoftenZhou}. 

Let $\rho:\gsc \to \gder$ be the simply connected cover of the derived group $\gder$ of $G$ and note that $\rho$ induces an action of $\gsc(\afp)$ on $\shginf$. From now on we will need another assumption:
\begin{Hyp} \label{Hyp:LocInt}
For each finite extension $F$ of the reflex field $E$ and any place $w$ of $F$ extending $v$, the natural maps
\begin{align}
    \pi_0(\mathbf{Sh}_{K^pK_p}(G,X) \otimes_E F) \leftarrow \pi_0(\mathscr{S}_{K^pK_p} \otimes_{\mathcal{O}_{E,(v)}} \mathcal{O}_{F,(w)}) \to \pi_0(\shg \otimes_k k(w))
\end{align}
are isomorphisms.
\end{Hyp}
\begin{Lem}
Hypothesis \ref{Hyp:LocInt} holds if $\shg$ has geometrically integral connected components. 
\end{Lem}
\begin{proof}
This is \cite[Cor. 4.1.11]{MP}.
\end{proof}
\begin{Rem}
The variety $\shg$ has geometrically integral connected components if $K_p$ is hyperspecial because then the integral models are smooth by work of Kisin \cite{KisinModels}. More generally the Kisin--Pappas integral models of \cite{KisinPappas} have geometrically integral connected components $K_p$ is very special.
\end{Rem}
\subsubsection{Connected components} The following result is well known.
\begin{Lem} \label{Lem:ConnectedComponentsAreTorsorsComplexI}
Let $Y_{\infty}$ be a connected component of the scheme
\begin{align} \label{Eq:InverseLimit}
    \varprojlim_{U \subset G(\af)} \mathbf{Sh}_{U,\mathbb{C}}(G,X).
\end{align}
Then $Y_{\infty}$ is stable under the action of $\gsc(\af)$.
\end{Lem}
\begin{proof}
This is direct consequence of the description of connected component of Shimura varieties and strong approximation for $\gsc(\mathbb{Q})$, see \cite[Sec. 5.5.1, Lem. 5.5.4]{KisinShinZhu}.
\end{proof}
\begin{Cor} \label{Cor:ConnectedComponentsAreTorsorsComplexII}
Let $Y_{\infty}$ be a connected component of 
\begin{align}
    \varprojlim_{U^p \subset G(\afp)} \mathbf{Sh}_{U^pK_p, \mathbb{C}}(G,X).
\end{align}
Then $Y_{\infty}$ is stable under the action of $\gsc(\afp)$.
\end{Cor}
\begin{proof}
We consider 
\begin{align}
    \varprojlim_{U \subset G(\afp)} \mathbf{Sh}_{U}(G,X) \to \varprojlim_{U^p \subset G(\afp)} \mathbf{Sh}_{U^pK_p}(G,X) \to \mathbf{Sh}_{K^p K_p}(G,X)
\end{align}
and we let $Y_{\infty}'$ be a connected component of the left hand side mapping to $Y_{\infty}$. Then $Y_{\infty}'$ is stable under the action of $\gsc(\af)$ and thus $Y_{\infty}$ is stable under the action of $\gsc(\afp)$. 
\end{proof}
\begin{Lem} \label{Lem:ConnectedComponentsAreTorsors}
Suppose that Hypothesis \ref{Hyp:LocInt} holds and let $Y_{\infty} \subset \shginfov$ be a connected component. Then $Y_{\infty}$ is stable under the action of $\gsc(\afp)$. 
\end{Lem}
\begin{proof}
It suffices to prove this for Shimura varieties over $\mathbb{C}$, because the connected component are defined over an algebraic closure $\overline{E}$ of the reflex field $E$ and the result can be transported to the special fiber using Hypothesis \ref{Hyp:LocInt}. The result over $\mathbb{C}$ is proved in Corollary \ref{Cor:ConnectedComponentsAreTorsorsComplexII}.
\end{proof}
\subsubsection{} Let $Z \subset \shg$ be a $G(\afp)$-stable locally closed subscheme with inverse image $\tilde{Z} \subset \shginf$. A finite \'etale cover $X \to Z$ is called $G(\afp)$-equivariant if $\tilde{X}:=\tilde{Z} \times_Z X$ has an action of $G(\afp)$ making the natural map $\tilde{X} \to \tilde{Z}$ equivariant for the action of $G(\afp)$. If Hypothesis \ref{Hyp:LocInt} is satisfied, then $\gsc(\afp)$ acts on the fibers of
\begin{align}
    \pi_0(\tilde{X}) \to \pi_0(\shginfov).
\end{align}
\begin{Lem} \label{Lem:ClosedOrbitsII}
If Hypothesis \ref{Hyp:LocInt} holds, then $\gsc(\afp)$ acts continuously on $\pi_0(\tilde{X})$.
\end{Lem}
\begin{proof}
The assumption that $X \to Z$ is finite \'etale implies that $\pi_0(\tilde{X}) \to \pi_0(\tilde{Z})$ is a finite map with discrete fibers, and therefore the action of $G(\afp)$ on $\pi_0(\tilde{X})$ is continuous because the action on $\pi_0(\tilde{Z})$ is, see Corollary \ref{Cor:Continuity}.
\end{proof}

Let $\Sigma$ be a finite set of places of $\mathbb{Q}$ containing the infinite place, containing $p$, and containing all places $\ell$ where $G_{\ell}^{\mathrm{ad}}$ has a compact factor. From now on we will work with $G(\afp)$-stable subvarieties $Z$ defined over $\ovfp$ and with geometric monodromy groups.
\begin{Thm} \label{Thm:Irreducibility}
Let $X \to Z$ be a $G(\afp)$-equivariant finite \'etale cover of a smooth $G(\afp)$-stable locally closed subscheme $Z \subset \shgov$ and suppose that each connected component of $Z$ intersects a $\mathbb{Q}$-non-basic Newton stratum. If Hypotheses \ref{Hyp:Levi} and \ref{Hyp:LocInt} hold, then $\gsc(\afs)$ acts trivially on the fibers of
\begin{align} \label{eq:MaponPinaught}
    \pi_0(\tilde{X}) \to \pi_0(\shginfov).
\end{align}
\end{Thm}
For a prime $\ell \not \in \Sigma$ we will write $K_{\ell}$ for the image of $K^p \to G(\ql)$ and $\pi_{\ell}:\operatorname{Sh}_{G,K_pK^{p},\ell,\ovfp} \to \shgov$ for the induced $K_{\ell}$-torsor over $\shgov$.
\begin{Lem} \label{Lem:Lem:ConnectedComponentsAreTorsorsII}
Suppose that Hypothesis \ref{Hyp:LocInt} holds and let $Y_{\infty} \subset \operatorname{Sh}_{G,K_pK^{p},\ell,\ovfp}$ be a connected component with image $Y \subset \shgov$. Then $Y_{\infty} \to Y$ is a torsor for a compact open subgroup of $\gder(\ql)$.
\end{Lem}
\begin{proof}
It follows from profinite Galois theory for schemes, see Section \ref{Sec:InfiniteGalois}, that the stabiliser $K_{\infty}$ of $Y_{\infty}$ in $G(\ql)$ can be identified with the image of
\begin{align}
    \pi_1^{\text{\'et}}(Y,y) \to G(\ql)
\end{align}
for some point $y \in Y(\ovfp)$. If we apply Corollary \ref{Cor:GeometricMonodromy} and Lemma \ref{Lem:DenseCompactOpen} to $Z=\shg$, it follows that this image contains a compact open subgroup of $\gder(\ql)$.
\end{proof}
\begin{proof}[Proof of Theorem \ref{Thm:Irreducibility}]
We write $\tildezl \to Z$ for the induced $K_{\ell}$ torsor and $X_{\ell} \to \tildezl$ for $\tildezl \times_Z X$. Then the action of $G(\afp)$ on $\tilde{X}$ and $\tilde{Z}$ induces an action of $G(\ql)$ on $X_{\ell}$, and it suffices to show that $\gsc(\ql)$ acts trivially on the fibers of
\begin{align}
     a_{\ell}:\pi_0(X_{\ell}) \to \pi_0\left(\operatorname{Sh}_{G,K_pK^{p},\ell,\ovfp}\right)
\end{align}
for all $\ell \not \in \Sigma$. \smallskip

Let $x \in \pi_0(X_{\ell})$ and let $\zcirc$ be a connected component of $Z$ containing the image of $x$. Moreover let $\tildezcl \subset \zl$ be the inverse image of $\zcirc$. Fix a point $z \in \zcirc(\ovfp)$, then Hypothesis \ref{Hyp:Levi} and Corollary \ref{Cor:GeometricMonodromy} tell us that the image of
\begin{align}
    \rho_{\ell}:\pi_1^{\text{\'et}}(\zcirc,z) \to K_{\ell}
\end{align}
is a compact subgroup $M_{\mathrm{geom},\ell}$ whose Zariski closure $\mathbb{M}_{\mathrm{geom},\ell}$ has neutral component equal to $G_{\ql}^{\mathrm{der}}$. It follows from Lemma \ref{Lem:DenseCompactOpen} that
the image of $\rho_{\ell}$ contains a compact open subgroup $V_{\ell} \subset G^{\mathrm{der}}(\ql)$. The upshot of this discussion is that the stabiliser in $G(\ql)$ of a connected component of $Z_{\ell}$ contains a compact open subgroup of $\gder(\ql)$ and this implies that the stabiliser in $G(\ql)$ of $x$ contains a compact open subgroup of $\gder(\ql)$. \medskip

Let $Y_{\infty}$ be a connected component of $\operatorname{Sh}_{G,K_pK^{p},\ell,\ovfp}$ such that the image $Y$ of $Y_{\infty}$ in $\shg$ contains $\zcirc$. Then it follows from Hypothesis \ref{Hyp:LocInt} and Lemma \ref{Lem:Lem:ConnectedComponentsAreTorsorsII} that $Y_{\infty} \to Y$ is a pro-\'etale torsor for a compact open subgroup $U_{\ell} \subset \gder$, and from Lemma \ref{Lem:ConnectedComponentsAreTorsors} that $Y_{\infty}$ is stable under the action of $\gsc(\ql)$. 

We will write $X_{\infty} \subset X_{\ell}$ for the inverse image of $Y_{\infty}$ in $X_{\ell}$ and let $X' \subset X$ be its image. Note that $x \in \pi_0(X_{\infty})$ by construction. Then $X_{\infty} \to X'$ is pro-\'etale $U_{\ell}$ torsor and $X_{\infty}$ is stable under the action of $\gsc(\ql)$. This action is moreover continuous by Lemma \ref{Lem:ClosedOrbitsII} and the inclusion
\begin{align}
    \pi_0(X_{\infty}) \subset \pi_0(X_{\ell})
\end{align}
is closed since $\{ Y_{\infty} \} \subset \pi_0\left(\operatorname{Sh}_{G,K_pK^{p},\ell,\ovfp}\right)$ is closed. In particular, the topological space $\pi_0(X_{\infty})$ is compact Hausdorff. \smallskip

Let $U_{\ell}'$ be the inverse image of $U_{\ell}$ in $\gsc(\ql)$. Then the quotient
\begin{align}
    U_{\ell}' \backslash \pi_0(X_{\infty})
\end{align}
is finite since $U_{\ell} \backslash \pi_0(X_{\infty}) = \pi_0(X')$
is finite and since the map $U_{\ell}' \to U_{\ell}$ has finite cokernel. This means that there are finitely many (open and closed) $U_{\ell}'$ orbits on $\pi_0(X_{\infty})$. Therefore the $\gsc(\ql)$ orbit of $x$ on $\pi_0(X_{\infty})$ is a union of finite many $U_{\ell}'$-orbits and thus closed; in particular it is compact Hausdorff. It then follows from Lemma \ref{Lem:QuotientTopology} that the $\gsc(\ql)$ orbit of $x$ is homeomorphic to $\gsc(\ql)/P_x$, where $P_x \subset \gsc(\ql)$ is the stabiliser of $x$. In particular, it follows that $\gsc(\ql)/P_x$ is compact. \smallskip

The group $P_x$ contains a compact open subgroup of $\gsc(\ql)$ because the stabiliser of $x$ in $G(\ql)$ contains a compact open subgroup of $\gder(\ql)$ and $\gsc(\ql) \to \gder(\ql)$ has finite fibers. This implies that $\gsc(\ql)/P_x$ has the discrete topology, and we conclude that $\gsc(\ql)/P_x$ is a finite set or equivalently $P_x$ is a finite index subgroup. The assumption that $\gsc_{\ql}$ has no compact factors implies, by \cite[Thm. 7.1, Thm. 7.5]{PR}, that the group $\gsc(\ql)$ has no finite index subgroups. Therefore $\gsc(\ql)/P_x$ is a singleton which is precisely what we wanted to prove. 
\end{proof} 
\section{Serre--Tate coordinates and unipotent group actions} \label{Sec:SerreTate}
In this section we show that the classical Serre--Tate coordinates, as described in \cite{KatzSerreTate}, can be reinterpreted using the action of a  unipotent formal group, as in \cite{HoweUnipotent}. Our results are more-or-less a direct generalisation of the results of \cite{HoweUnipotent}, except that we construct the action of unipotent formal groups using Rapoport--Zink spaces, while in loc. cit. this action is constructed using Igusa varieties. \medskip

In Section \ref{Sec:ClassicalSerreTate}, we recall the classical theory of Serre--Tate coordinates following \cite{KatzSerreTate}, which shows that the formal deformation space $\operatorname{Def}(Y)$ of an ordinary $p$-divisible group $Y$ over $\ovfp$ has the structure of a commutative formal group. We then compute the scheme-theoretic $p$-adic Tate-module of the $p$-divisible group $\mathcal{H}_{0,1}$ associated to this formal group. In Section \ref{Sec:RapoportZink} we use Rapoport--Zink spaces to describe an action of the universal cover $\tilde{\mathcal{H}}_{0,1}$ of $\mathcal{H}_{0,1}$ on the formal scheme $\widehat{\operatorname{Def}}(Y)$ associated to $\operatorname{Def}(Y)$. In Section \ref{Sec:ProofOfSerreTateVSRapoportZink} we identify this action with the projection from the universal cover to $\mathcal{H}_{0,1}$ followed by the left-translation action of $\mathcal{H}_{0,1}$ on $\widehat{\operatorname{Def}}(Y)$.

\subsubsection{} We consider the category $\art$ of Artin local $\zpbr$-algebras $R$ such that the natural map $\ovfp \to R/\mathfrak{m}_R$ is an isomorphism. Here $\mathfrak{m}_R$ is the unique maximal ideal of $R$ and we write $\alpha:R \to \ovfp$ for the composition of the natural map $R \to R/\mathfrak{m}$ with the inverse of the natural isomorphism $\ovfp \to R/\mathfrak{m}$. Note that $\alpha$ is functorial for morphisms in $\art$. We similarly consider the category $\nilp$ of $\zpbr$-algebras in which $p^n=0$ for some $n$. The category $\art$ is naturally a full subcategory of $\nilp$.

For a $p$-divisible group $\mathscr{G}$ over an algebra $R \in \nilp$ we define the $p$-adic Tate module to be the functor $
    T_p \mathscr{G}:=\varprojlim_n \mathscr{G}[p^n],
$ which is representable by a flat affine scheme over $\spec R$ by Proposition \cite[Prop. 3.3.1]{ScholzeWeinstein}.
\subsection{Classical Serre--Tate theory} \label{Sec:ClassicalSerreTate} Let $Y$ be an ordinary $p$-divisible group of dimension $g$ and height $2g$ over $\ovfp$. In other words, let $Y$ be a $p$-divisible group isomorphic to $\tfrac{\qp}{\zp}^{\oplus g} \oplus \mu_{p^{\infty}}^{\oplus g}$. \smallskip

Let $\operatorname{Def}(Y)$ be the functor on $\art$ sending $(R,\alpha)$ to the set of isomorphism classes of pairs $(X,\beta)$ where $X$ is a $p$-divisible group over $\spec R$ and $\beta:X \otimes_{R, \alpha} \ovfp \to Y_{\ovfp}$ is an isomorphism. This functor is (pro)-representable by a formally smooth formal scheme $\operatorname{Def}(Y)$ of relative dimension $g^2$ over $\spf \zpbr$. By \cite[Thm. 2.1]{KatzSerreTate}, this functor lifts to a functor valued in abelian groups such that the formal group $\operatorname{Def}(Y)$ is a $p$-divisible.\footnote{Recall that a commutative formal group $X$ is called \emph{$p$-divisible} if $[p]:X \to X$ is finite flat.} 

There is a canonical direct sum decomposition $Y=Y_0 \oplus Y_1$ where $Y_0$ is the maximal \'etale quotient and where $Y_1$ is equal to the formal completion of $Y$ at the origin. Since $Y_0$ is \'etale there is a unique lift to a $p$-divisible formal group over $\zpbr$, which we will denote by $Y_0^{\mathrm{can}}$. Similarly $Y_1$ has a unique lift to a $p$-divisible formal group over $\zpbr$, for example because the Serre dual of $Y_1$ is \'etale. We will denote this lift by $Y_1^{\mathrm{can}}$ and we will use $Y^{\mathrm{can}}:=Y_0^{\mathrm{can}} \oplus Y_1^{\mathrm{can}}$ to denote the canonical lift of $Y$ to $\zpbr$.

Let $Y^{\vee}$ be the Serre-dual of $Y$ and consider the free $\mathbb{Z}_p$-modules of rank $g$ given by $T_p Y(\ovfp)$ and $T_p Y^{\vee}(\ovfp)$. By \cite[Thm. 2.1]{KatzSerreTate}, the formal group $\operatorname{Def}(Y)$ is isomorphic to the functor on $\art$ sending $R$ to 
\begin{align}
    \hom\left(T_p Y(\ovfp) \otimes_{\zp} T_p Y^{\vee}(\ovfp), \widehat{\mathbb{G}}_m(R) \right).
\end{align}
Let $S$ be the complete Noetherian local $\zpbr$-algebra representing $\operatorname{Def}(Y)$ on $\art$. Then the abelian group structure on $\operatorname{Def}(Y)$ induces a (continuous) co-commutative Hopf algebra structure on $S$. In particular the formal scheme $\widehat{\operatorname{Def}}(Y):=\spf S$, considered as a functor on $\nilp$, has the structure of a formal group. We will write $\mathcal{H}_{0,1}:=\varinjlim \spf S[p^n]$ for the corresponding $p$-divisible group over $\spf \zpbr$. Note that it acts via left translation on $\widehat{\operatorname{Def}}(Y)$; we will denote this action by $a_{\mathrm{ST}}$ (for Serre--Tate).

\begin{Rem}
The natural map $\mathcal{H}_{0,1} \to \widehat{\operatorname{Def}}(Y)$ is an isomorphism of formal schemes, since both of them are formally smooth formal schemes of the same dimension. Nevertheless, it useful to treat them as different objects, for example because the notation $\widehat{\operatorname{Def}}(Y)$ is somewhat unwieldy, especially when passing to universal covers of $p$-divisible groups. 
\end{Rem}
\begin{Lem} \label{Lem:TateModuleSerreTate}
   The $p$-adic Tate module of $\mathcal{H}_{0,1}$ is isomorphic to the sheaf $\Hom(Y_0^{\mathrm{can}},Y_1^{\mathrm{can}})$ on $\nilp$ of homomorphisms from $Y_0^{\mathrm{can}}$ to $Y_1^{\mathrm{can}}$.
\end{Lem}
\begin{proof}
Let us prove the stronger assertion that there are isomorphisms $\mathcal{H}_{0,1}[p^n] \simeq \Hom(Y_0^{\mathrm{can}}, Y_1^{\mathrm{can}})[p^n]$ for all $n$, compatible with changing $n$. Note that $\mathcal{H}_{0,1}[p^n]$ is represented by the spectrum of an Artin local $\zpbr$-algebra. The same is true for $\Hom(Y_0^{\mathrm{can}}, Y_1^{\mathrm{can}})[p^n]$, since $\Hom(\underline{\mathbb{Z}/p^n \mathbb{Z}}, \mu_{p^n}) \simeq \mu_{p^n}$. Thus it suffices to show that the functors $\mathcal{H}_{0,1}[p^n]$ and $\Hom(Y_0^{\mathrm{can}}, Y_1^{\mathrm{can}})[p^n]$ are isomorphic as functors on $\art$. \smallskip

In \cite[p. 152]{KatzSerreTate} it is explained that $\operatorname{Def}(Y)$ is isomorphic to the functor (on $\art$) sending $R$ to
\begin{align} \label{Eq:SerreTateFormalGroup}
    \hom(T_p Y(\ovfp), Y_1^{\mathrm{can}}(R)).
\end{align}
Note that $T_p Y(\ovfp)=T_p Y_0(\ovfp)=T_p Y_0^{\mathrm{can}}(\ovfp)$ and that because $T_p Y_0^{\mathrm{can}}$ is an inverse limit of \'etale group schemes, the natural map $T_p Y_0^{\mathrm{can}}(R) \to T_p Y_0(\ovfp)$ is an isomorphism for $R \in \art$. Thus there is a natural isomorphism
\begin{align}
    \hom(T_p Y(\ovfp), Y_1^{\mathrm{can}}(R)) \simeq \hom(T_p Y_0^{\mathrm{can}}(R), Y_1^{\mathrm{can}}(R)).
\end{align}
The $p^n$-torsion of this group is given by 
\begin{align}
    \hom(T_p Y_0^{\mathrm{can}}(R), Y_1^{\mathrm{can}}(R))[p^n] &= \hom(T_p Y_0^{\mathrm{can}}(R), Y_1^{\mathrm{can}}[p^n](R)) \\
    &=\hom(Y_0^{\mathrm{can}}[p^n](R), Y_1^{\mathrm{can}}[p^n](R)).
\end{align}
We see that there is an isomorphism $\operatorname{Def}(Y)[p^n] \simeq \Hom(Y_0^{\mathrm{can}}[p^n],Y_1^{\mathrm{can}}[p^n])$ of functors on $\art$, which induces an isomorphism $\mathcal{H}_{0,1}[p^n] \simeq \Hom(Y_0^{\mathrm{can}}[p^n],Y_1^{\mathrm{can}}[p^n])$ of functors on $\nilp$. It is straightforward to check that these isomorphisms are compatible with increasing $n$, which concludes the proof. 
\end{proof}
\subsection{Rapoport--Zink spaces and unipotent formal groups} \label{Sec:RapoportZink} Let $\tilde{Y} \to Y$ be the universal cover of $Y$, defined as the inverse limit of the projective system
\begin{align}
    \varprojlim_{p:G \to G} Y.
\end{align}
It is representable by a formal scheme by \cite[Prop. 3.1.3.(iii)]{ScholzeWeinstein}. By the proof of \cite[Proposition 4.2.11]{CaraianiScholze}, the automorphism group functors of $Y$ and $\tilde{Y}$ on $\nilp$ can be described as follows
\begin{align}
    \qquad \; \mathbf{Aut}(Y)=\begin{pmatrix}
     \mathbf{Aut}(Y_0) & 0 \\
     \Hom(Y_0, Y_1) & \mathbf{Aut}(Y_1)
    \end{pmatrix}, \mathbf{Aut}(\tilde{Y})=\begin{pmatrix}
     \mathbf{Aut}(\tilde{Y}_0) & 0 \\
     \Hom(Y_0, Y_1)[1/p] & \mathbf{Aut}(\tilde{Y}_1)
    \end{pmatrix}.
\end{align}
Moreover the functors $\mathbf{Aut}(Y_i)$ are pro-\'etale group schemes which are non-canonically isomorphic to the group schemes associated to the profinite group $\operatorname{GL}_g(\zp)$. The functors $\mathbf{Aut}(\tilde{Y}_i)$ are non-canonically isomorphic to the group schemes associated to the locally profinite group $\operatorname{GL}_g(\qp)$. 

Let $\tilde{\mathcal{H}}_{0,1}$ be the universal cover of $\mathcal{H}_{0,1}$. Then by the discussion after \cite[Def. 4.1.1]{CaraianiScholze}, we can identify the fpqc sheaves
\begin{align}
    \tilde{\mathcal{H}}_{0,1} = \left(T_p \mathcal{H}_{0,1}\right)[1/p].
\end{align}
Moreover, by the proof of \cite[Prop. 4.1.2]{CaraianiScholze}, there is a short exact sequence of fpqc sheaves
\begin{align}
    0 \to T_p \mathcal{H}_{0,1} \to \tilde{\mathcal{H}}_{0,1} \to \mathcal{H}_{0,1} \to 0.
\end{align}
By Lemma \ref{Lem:TateModuleSerreTate}, we can identify this with
\begin{align} \label{Eq:ShortExactSequenceTateModuleUniversalCover}
    0 \to \Hom(Y_0,Y_1) \to \Hom(Y_0,Y_1)[1/p] \to \mathcal{H}_{0,1} \to 0.
\end{align}
Note that $\Hom(Y_0, Y_1)[1/p]$ is isomorphic to $\widetilde{\mathcal{H}}_{0,1}$, and thus representable by a formal scheme by \cite[Prop. 3.1.3.(iii)]{ScholzeWeinstein} as above. In particular, this means that $\mathbf{Aut}(\tilde{Y})$ is representable by a formal scheme.

\subsubsection{} Let $\operatorname{RZ}_Y$ be the Rapoport--Zink space associated to $Y$. It is defined to be the functor on $\nilp$ sending $R$ to the set of isomorphism classes of pairs $(X, f)$ where $X$ is a $p$-divisible group over $\spec R$ and $f:X \dashrightarrow Y_R$ is a quasi-isogeny (or equivalently, by \cite[Lemma 1.1.3.3]{KatzSerreTate}, a quasi-isogeny $f_0:X_{R/pR} \dashrightarrow Y_{R/pR})$. The functor $\operatorname{RZ}_Y$ is representable by a formally smooth formal scheme over $\spf \zpbr$ by \cite[Thm. 2.16]{RapoportZink}. The group functor $\mathbf{Aut}(\tilde{Y})$ acts on $\operatorname{RZ}_Y$ via postcomposition, where we note that an automorphism of $\tilde{Y}$ is the same thing as a self quasi-isogeny of $Y$. \smallskip

Let $y$ be the $\ovfp$-point of $\operatorname{RZ}_Y$ corresponding to the identity map $Y \to Y$ and let 
\begin{align}
    \operatorname{RZ}_Y^{/y} \subset \operatorname{RZ}_Y
\end{align}
be the formal completion of $\operatorname{RZ}_Y$ in $\{y\}$, in the sense of considered as a formal algebraic space as in \cite[Tag0GVR]{stacks-project}. By definition this is the subfunctor of $\operatorname{RZ}_Y$ corresponding to those morphisms $\spec R \to \operatorname{RZ}_Y$ that factor through $\{y\}$ on the level of topological spaces. In other words, it consists of those morphisms $\spec R \to \operatorname{RZ}_Y$ such that the induced morphism $\spec R^{\mathrm{red}} \subset \spec R \to \operatorname{RZ}_Y$ factors through $y:\spec \ovfp \to \operatorname{RZ}_Y$. 

In terms of the moduli description, this means that we are looking at those quasi-isogenies $f:X \dashrightarrow Y_R$ such that: There is a (necessarily unique) isomorphism $\beta:X_{R^{\mathrm{red}}} \to Y_{R^{\mathrm{red}}}$ making the following diagram commute
\begin{equation} \label{Eq:FormalCompletionRZ}
    \begin{tikzcd}
        X_{R^{\mathrm{red}}} \arrow[d, dashed,"f"] \arrow[r, "\beta"] & Y_{R^{\mathrm{red}}} \arrow[d, equals] \\
        Y_{R^{\mathrm{red}}} \arrow[r, equals]\ & Y_{R^{\mathrm{red}}}.
    \end{tikzcd}
\end{equation}
Now restrict this moduli description to the full subcategory $\art \subset \nilp$. Then $\operatorname{RZ}_{Y}^{/y}$ can be described as the functor on $\art$ sending $(R,\alpha)$ to the set of isomorphisms classes of triples $(X,\beta,f)$, where $X$ is a $p$-divisible group over $R$ equipped with an isomorphism $\beta:X \otimes_{R, \alpha} \ovfp \to Y$ and where $f$ is a quasi-isogeny $f:X \dashrightarrow Y_R$ such that \eqref{Eq:FormalCompletionRZ} commutes. 
\begin{Lem} \label{Lem:FormalCompletionRZ}
The natural forgetful map $\operatorname{RZ}_{Y}^{/y} \to \operatorname{Def}(Y)$ sending $(X,\beta,f)$ to $(X,\beta)$ is an isomorphism. In particular, there is an isomorphism of formal schemes $\widehat{\operatorname{Def}}(Y) \simeq \operatorname{RZ}_{Y}^{/y}$.
\end{Lem}
\begin{proof}
The commutativity of \eqref{Eq:FormalCompletionRZ} expresses the fact that $f$ is a lift of the quasi-isogeny $Y \to Y$ given by the identity. But since quasi-isogenies lift uniquely by \cite[Lem. 1.1.3.3]{KatzSerreTate}, the data of $f$ is superfluous and we see that the forgetful map $\operatorname{RZ}_{Y}^{/y}(R) \to \operatorname{Def}(Y)(R)$ is a bijection for all $R \in \art$.
\end{proof}
The subgroup
\begin{align} \label{Eq:SemiDirectProductIntegral}
    \begin{pmatrix}
     \mathbf{Aut}(Y_0) & 0 \\
     \widetilde{\mathcal{H}}_{0,1} & \mathbf{Aut}(Y_1),
    \end{pmatrix} \subset \mathbf{Aut}(\tilde{Y})
\end{align}
preserves the identity quasi-isogeny $Y \to Y$ and therefore acts on $\widehat{\operatorname{Def}}(Y)$. In particular, the profinite group
\begin{align}
   \mathbf{Aut}(Y_0)(\ovfp) \times \mathbf{Aut}(Y_1)(\ovfp) = \mathbf{Aut}(Y)(\ovfp)
\end{align}
acts on $\widehat{\operatorname{Def}}(Y)$. This induces an action of $\mathbf{Aut}(Y)(\ovfp)$ on $\operatorname{Def}(Y)$ because $\ovfp$ is an object of $\art \subset \nilp$.
\begin{Cor} \label{Cor:UsualActionDefY}
This action sends a pair $(X, \beta) \in \operatorname{Def}(Y)(R)$, where $X$ is a $p$-divisible group over $\spec R$ and $\beta:X \otimes_{R, \alpha} \ovfp \to Y_{\ovfp}$ is an isomorphism, to $(X, g \circ \beta)$ for $g \in \mathbf{Aut}(Y)(\ovfp)$. 
\end{Cor}
\begin{proof}
This follows from Lemma \ref{Lem:FormalCompletionRZ} and the uniqueness of the isomorphism $\beta:X_{\ovfp} \to Y$ given $f:X \dashrightarrow Y_R$.
\end{proof}

\subsubsection{} Since the action of $\widetilde{\mathcal{H}}_{0,1}$ on $\operatorname{RZ}_Y$ preserves the point $y$, there is an induced action 
\begin{align} \label{Eq:TildeActionRZ}
    \tilde{a}_{\mathrm{RZ}}: \tilde{\mathcal{H}}_{0,1} \times \widehat{\operatorname{Def}}(Y) \to \widehat{\operatorname{Def}}(Y).
\end{align}

The goal of the rest of this section is to prove the following theorem, our proof of which was heavily inspired by the proof of \cite[Thm. 6.2.1]{HoweUnipotent}, which deals with the $g=1$ case.
\begin{Prop} \label{Prop:SerreTateVSRapoportZink}
The action $\tilde{a}_{\mathrm{RZ}}$ factors through an action of $\mathcal{H}_{0,1}$ via the natural quotient map
$\tilde{\mathcal{H}}_{0,1} \to \mathcal{H}_{0,1}$. Moreover the induced action of $\mathcal{H}_{0,1}$ is given by $a_{\mathrm{ST}}$.
\end{Prop}
\subsection{Proof of Proposition \ref{Prop:SerreTateVSRapoportZink}} \label{Sec:ProofOfSerreTateVSRapoportZink} Choose isomorphisms
\begin{align}
    T_p Y_0 \simeq \mathbb{Z}_p^{\oplus g}, \qquad Y_1 \simeq \left(\mu_{p^{\infty}}\right)^{\oplus g},
\end{align}
which induce isomorphisms of functors on $\art$
\begin{align}
    \operatorname{Def}(Y) \simeq \Hom\left(\mathbb{Z}_p^{\oplus g}, \left(\mu_{p^{\infty}}\right)^{\oplus g}\right)
\end{align}
In fact if we let $x_1, \cdots, x_g \in \mathbb{Z}_p^{\oplus g}$ be the standard basis vectors, then we can in fact identify
\begin{align}
    \operatorname{Def}(Y) \simeq \left(\mu_{p^{\infty}}\right)^{\oplus g^2}
\end{align}
with coordinates $q_{i,j}$ for $1 \le i,j \le g$ and similarly
\begin{align}
    \mathcal{H}_{0,1} \simeq \left(\mu_{p^{\infty}}\right)^{\oplus g^2}
\end{align}
with coordinates $\zeta_{i,j}$ for $1 \le i,j \le g$. For $R$ in $\art$ a morphism
\begin{align}
    \spec R \to \operatorname{Def}(Y)
\end{align}
corresponds to elements $q_{i,j} \in 1 + \mathfrak{m}_R$, and the corresponding deformation of $Y$ is the $p$-divisible group $X_{\underline{q}}$ corresponding to the pushout of (see \cite[p. 152]{KatzSerreTate})
\begin{align}
    0 \to \mathbb{Z}_{p}^{\oplus g} \to \qp^{\oplus g} \to \tfrac{\qp^{\oplus g}}{\zp^{\oplus g}} \to 0
\end{align}
via the morphism $\mathbb{Z}_{p,R}^{\oplus g} \to \mu_{p^{\infty},R}^{\oplus g}$ given by $x_i \mapsto (q_{i,1}, \cdots, q_{i,g})$. In fact, there is a pushout diagram
\begin{equation} \label{eq:PushOutDiagram}
    \begin{tikzcd} 
     0 \arrow{r} & \mathbb{Z}^{\oplus g} \arrow{r} \arrow{d} &\mathbb{Z}[1/p]^{\oplus g} \arrow{r} \arrow{d} & \tfrac{\qp^{\oplus g}}{\zp^{\oplus g}} \arrow{r} \arrow[d, equals] & 0\\
     0 \arrow{r} & \mathbb{Z}_{p}^{\oplus g} \arrow{r} & \qp^{\oplus g} \arrow{r} & \tfrac{\qp^{\oplus g}}{\zp^{\oplus g}} \arrow{r} & 0
\end{tikzcd}
\end{equation}
and so we can also think of $X_{\underline{q}}$ as the quotient of $\mu_{p^{\infty},R}^{\oplus g} \oplus \mathbb{Z}[1/p]^{\oplus g}$ by the image of the map $h_{\underline{q}}:\mathbb{Z}^{\oplus g} \to \mu_{p^{\infty},R}^{\oplus g} \oplus \mathbb{Z}[1/p]^{\oplus g}$ given by $x_i \mapsto ((q_{i,1}, \cdots, q_{i,g}),(x_i))$.

Let $N$ be an integer such that $q_{i,j}^{p^N}=1$ for all $i,j$, which exists since $R$ is Artinian. Then the isogeny
\begin{align}
    \mu_{p^{\infty},R}^{\oplus g} \oplus \mathbb{Z}[1/p]^{\oplus g} &\to \mu_{p^{\infty},R}^{\oplus g} \oplus \mathbb{Z}[1/p]^{\oplus g} \\
    (A,B) &\mapsto (p^N A, p^N B)
\end{align}
maps $h_{\underline{q}}(\mathbb{Z}^{\oplus g})$ into $h_{\underline{1}}(\mathbb{Z}^{\oplus g})$. Thus it induces a quasi-isogeny
\begin{align} \label{Eq:Canonicalisogeny}
    f_{\underline{q},N}:X_{\underline{q}} \dashrightarrow X_{\underline{1}}=Y_R,
\end{align}
and the induced quasi-isogeny $X_{\underline{q},\ovfp}=X_{\underline{1},\ovfp} \to X_{\underline{1},\ovfp}$ is given by $p^N$. It follows that the quasi-isogeny $p^{-N}f_{\underline{q},N}$ is the unique quasi-isogeny lifting the identity $X_{\underline{q},\ovfp}=X_{\underline{1},\ovfp} \to X_{\underline{1},\ovfp}$. Let us write $\underline{q} \in \operatorname{RZ}_Y(R)$ for $p^{-N}f_{\underline{q},N}:X_{\underline{q}} \dashrightarrow Y_R$. \smallskip

A morphism $\spec R \to \mathcal{H}_{0,1}$ corresponds to elements $\zeta_{i,j} \in 1+\mathfrak{m}_R$. The left translation action of $\mathcal{H}_{0,1}$ via the Serre--Tate action is given by
\begin{align}
    a_{\mathrm{ST}}(\underline{\zeta}, \underline{q}) = \underline{\zeta q},
\end{align}
where $(\zeta q)_{i,j}=(\zeta_{i,j} \cdot q_{i,j})$ and where $(\zeta_{i,j} \cdot q_{i,j})$ denotes the multiplication in $\mu_{p^{\infty}}(R)=1+\mathfrak{m}_R$. In terms of $p$-divisible groups, this correspond to the $p$-divisible group $X_{\underline{\zeta q}}$. We will write $\underline{\zeta q} \in \operatorname{RZ}_Y(R)$ for the element corresponding to $X_{\underline{\zeta q}}$. \smallskip
\begin{proof}[Proof of Proposition \ref{Prop:SerreTateVSRapoportZink}]
By definition of the action $\tilde{a}_{\mathrm{RZ}}$, it suffices to show that for every fpqc cover $\spec \tilde{R} \to \spec R$ and every lift
\begin{align}
    \tilde{\underline{\zeta}} \in \tilde{\mathcal{H}}_{0,1}(\tilde{R})
\end{align}
of $\underline{\zeta} \in \mathcal{H}_{0,1}(\tilde{R})$, we have $\tilde{a}_{\mathrm{RZ}}(\tilde{\underline{\zeta}}, \underline{q})=\underline{\zeta q}$. There is a universal such lift over the fpqc cover $\tilde{R}_{'}$ given by formally adjoining all the $p$-power roots of all $\zeta_{i,j}$, and it suffices to prove the result for this choice of $\tilde{R}_{'}$. To be precise, we let 
\begin{align}
    \tilde{R}_{'}&:=\varinjlim_n R[\underline{\zeta}_{i,j}^{1/p^n}], \text{ where} \\
    R[\underline{\zeta}_{i,j}^{1/p^n}]&:= R[Y_{i,j}]/(Y_{i,j}^{p^n}-\zeta_{i,j}),
\end{align}
with transition maps defined by $Y_{i,j} \mapsto Y_{i,j}^p$ or equivalently by $\zeta_{i,j}^{1/p^n} \mapsto \zeta_{i,j}^{1/p^n}$. Then the universal lift $\tilde{\underline{\zeta}}$ is the element with $\tilde{\underline{\zeta}}_{i,j}=(\zeta_{i,j}^{1/p^n})_n$.

By fpqc descent, we can furthermore pass to the fpqc cover $\tilde{R}_{'} \to \tilde{R}$ given by also formally adjoining all the $p$-power roots of all the coordinates $q_{i,j}$. Recall that $X_{\underline{q}, \tilde{R}}$ is defined as the quotient of 
\begin{align}
    \mu_{p^{\infty},\tilde{R}}^{\oplus g} \oplus \mathbb{Z}[1/p]^{\oplus g}
\end{align}
by the image of the map $h_{\underline{q}}$ which sends the standard basis element $x_i$ to
\begin{align}
    \left((q_{i,1}, \cdots, q_{i,g}),(x_i)\right) \in \mu_{p^{\infty},\tilde{R}}^{\oplus g} \oplus \mathbb{Z}[1/p]^{\oplus g}.
\end{align}
The $p$-divisible group $X_{\underline{\zeta q}, \tilde{R}}$ is defined similarly but then using the map $h_{\underline{\zeta q}}$. The compatible sequence of $p$-power roots of $\underline{\zeta}$ defined by $\tilde{\underline{\zeta}}$ defines a map
\begin{align}
    \psi_{\tilde{\underline{\zeta}}}:\mu_{p^{\infty},\tilde{R}}^{\oplus g} \oplus \mathbb{Z}[1/p]^{\oplus g} &\to \mu_{p^{\infty},\tilde{R}}^{\oplus g} \oplus \mathbb{Z}[1/p]^{\oplus g} \\
    (A,B) &\mapsto (A \cdot L_{\tilde{\underline{\zeta}}}(B) , B),
\end{align}
where $L_{\tilde{\underline{\zeta}}}:\mathbb{Z}[1/p]^{\oplus g} \to \mu_{p^{\infty},\tilde{R}}^{\oplus g}$ is the morphism sending 
\begin{align}
    \frac{x_i}{p^n} \mapsto \left(\zeta_{i,1}^{1/p^n}, \zeta_{i,2}^{1/p^n}, \cdots, \zeta_{i,g}^{1/p^n} \right).
\end{align}
It is straightforward to check that this map satisfies
\begin{align}
    \psi_{\tilde{\underline{\zeta}}}(h_{\underline{q}}(\mathbb{Z}^{\oplus g}) = h_{\underline{\zeta q}}(\mathbb{Z}^{\oplus g})
\end{align}
and that it thus induces an isomorphism on quotients
\begin{align}
    \phi_{\tilde{\underline{\zeta}}}:X_{\underline{q},\tilde{R}} \to X_{\underline{\zeta q},\tilde{R}}.
\end{align}
Choose $N$ sufficiently large such that $\zeta_{i,j}^{p^N}=1$ and $q_{i,j}^{p^N}=1$ for all $i,j$. Let 
\begin{align}
    p^{-2N}f_{\underline{\zeta q},2N}:X_{\underline{\zeta q}} \dashrightarrow Y_R
\end{align}
be the unique quasi-isogeny lifting the identity map $X_{\underline{\zeta q},\ovfp}=Y_{\ovfp} \to Y_{\ovfp}$ as described in \eqref{Eq:Canonicalisogeny}. To prove the proposition it suffices to show that the following diagram commutes:
\begin{equation} \label{Eq:CrucialDiagram}
   \begin{tikzcd}
        X_{\underline{q}, \tilde{R}} \arrow{r}{\phi_{\tilde{\underline{\zeta}}}} \arrow[d, dashed,"p^{-N}f_{\underline{q},N}"] & X_{\underline{\zeta q}, \tilde{R}} \arrow[d, dashed,"p^{-2N}f_{\underline{\zeta q},2N}"] \\
        \left(\frac{\qp^{\oplus g}}{\zp^{\oplus g}}\right)_{\tilde{R}} \oplus \mu_{p^{\infty},\tilde{R}}^{\oplus g} \arrow[r, dashed,"\xi_{\tilde{\underline{\zeta}}}"] & \left(\frac{\qp^{\oplus g}}{\zp^{\oplus g}}\right)_{\tilde{R}} \oplus \mu_{p^{\infty},\tilde{R}}^{\oplus g}.
   \end{tikzcd}
\end{equation}
Here $\xi_{\tilde{\underline{\zeta}}}$ is given by the matrix $\TJM{1}{0}{\tilde{\underline{\zeta}}}{1}$, see the beginning of Section \ref{Sec:RapoportZink} for the matrix notation. To show that this diagram commutes we consider the auxiliary commutative diagram
\begin{equation} \label{Eq:TrivialDiagram}
    \begin{tikzcd}
         \mu_{p^{\infty},\tilde{R}}^{\oplus g} \oplus \mathbb{Z}[1/p]^{\oplus g} \arrow{r}{\psi_{\tilde{\underline{\zeta}}}} \arrow{d}{p^{N}} \arrow{r} & \mu_{p^{\infty},\tilde{R}}^{\oplus g} \oplus \mathbb{Z}[1/p]^{\oplus g} \arrow{d}{p^{2N}} \\
         \mu_{p^{\infty},\tilde{R}}^{\oplus g} \oplus \mathbb{Z}[1/p]^{\oplus g} \arrow{r}{p^N \psi_{\tilde{\underline{\zeta}}}} & \mu_{p^{\infty},\tilde{R}}^{\oplus g} \oplus \mathbb{Z}[1/p]^{\oplus g}.
    \end{tikzcd} 
\end{equation}
The diagram of quasi-isogenies \eqref{Eq:CrucialDiagram} is obtained from the diagram \eqref{Eq:TrivialDiagram} by quotienting by the subgroups
\begin{equation}
    \begin{tikzcd}
         h_{\underline{q}}(\mathbb{Z}^{\oplus g}) \arrow[r,"\psi_{\tilde{\underline{\zeta}}}"] \arrow{d}{p^N} & h_{\underline{\zeta q}}(\mathbb{Z}^{\oplus g}) \arrow{d}{p^{2N}} \\
         h_{\underline{1}}(\mathbb{Z}^{\oplus g}) \arrow{r}{p^N\psi_{\tilde{\underline{\zeta}}}} & h_{\underline{1}}(\mathbb{Z}^{\oplus g}),
    \end{tikzcd}
\end{equation}
and formally inverting certain powers of $p$. It follows that \eqref{Eq:CrucialDiagram} is commutative. 
\end{proof}

\section{The formal neighborhood of an ordinary point} \label{Sec:SerreTateHodge}
The goal of this section is to give Serre--Tate coordinates for the formal completions of points in the ordinary locus of Shimura varieties of Hodge type. 

In Section \ref{Sec:IntegralModelsHyperspecial} we specialise to the smooth canonical integral models of Shimura varieties of Hodge type at hyperspecial level, and we moreover assume that the ordinary locus is nonempty. In Section \ref{Sec:Dieudonne} we recall a small amount of covariant Dieudonn\'e theory for semiperfect rings, following \cite{ScholzeWeinstein}.

In Section \ref{Sec:SubtorusHodge} we prove that the formal completion of the ordinary locus gives a subtorus of the Serre--Tate torus, reproving a special case of \cite[Thm. 1.1]{ShankarZhou}. We also give a group-theoretic description the Dieudonn\'e module of the associated $p$-divisible group. In Section \ref{Sec:NontrivialActions} we introduce strongly nontrivial actions of algebraic groups on isocrystals, which we will need to confirm the hypotheses of the rigidity theorem of \cite{ChaiRigidity}.

\subsection{Integral models at hyperspecial level} \label{Sec:IntegralModelsHyperspecial}
Let the notation be as in Section \ref{Sec:IntModels}. In particular, we have a Shimura datum $(G,X)$ of Hodge type with reflex field $E$, a prime $p$ and a prime $v$ of $E$ above $p$. Moreover there is a symplectic space $V$ and a Hodge embedding $(G,X) \to (\mathcal{G}_V, \mathcal{H}_V)$ and for every sufficiently small $K^p \subset G(\afp)$ there is a sufficiently small $\mathcal{K}^p \subset \mathcal{G}_V(\afp)$ and a finite morphism
\begin{align} \label{eq:ClosedImmersion}
    \mathscr{S}_{K}:=\mathscr{S}_K(G,X) \to \mathcal{S}_{\mathcal{K}} \otimes_{Z_{(p)}} \mathcal{O}_{E,(v)}.
\end{align}
Recall that there is a $\mathbb{Z}_{(p)}$-lattice $V_{(p)}$ of $V$ on which the symplectic form is $\mathbb{Z}_{(p)}$-valued, and recall that we have defined $K_p \subset G(\qp)$ to be its stabiliser. \textbf{From now on we will assume that $K_p$ is a hyperspecial subgroup}, in which case $\mathscr{S}_{K}$ is the canonical integral model of $\mathbf{Sh}_{K^pK_p}(G,X)$ over $\mathcal{O}_{E_v}$. Moreover the main theorem of \cite{xu2020normalization} tells us that the map \eqref{eq:ClosedImmersion} is a closed immersion. 

\subsubsection{} \label{Sec:IntegralTensors}
The Zariski closure $\mathcal{G}_{\mathbb{Z}_{(p)}}$ of $G$ inside $\mathcal{G}_{V_{(p)}}$ is a reductive group scheme over $\mathbb{Z}_{(p)}$. By \cite[Prop. 1.3.2]{KisinModels}, we can choose tensors $\{s_{\alpha}\} \subset V_{(p)}^{\otimes}$ whose stabiliser is $\mathcal{G}_{\mathbb{Z}_{(p)}}$. All the results of Section \ref{Sec:IntModels} still go through with this choice of tensors.

For $x \in \shg(\ovfp)$ we have seen in Section \ref{Sec:Tensors} that there are tensors
\begin{align}
    \{s_{\alpha,\mathrm{cris}}\} \subset \mathbb{D}_x^{\otimes},
\end{align}
where $\mathbb{D}_x^{\otimes}$ is the rational contravariant Dieudonn\'e module of $A_x[p^{\infty}]$. Now let $\mathbb{D}(A_x[p^{\infty}])$ be the \emph{integral} contravariant Dieudonn\'e module. Then as explained in \cite[Sec. 6.3]{ShankarZhou}, we can choose the tensors $\{s_{\alpha,\mathrm{cris}}\}$ to lie inside
\begin{align}
    \mathbb{D}(A_x[p^{\infty}])^{\otimes}.
\end{align}
It is moreover explained there that there is an isomorphism
\begin{align} \label{Eq:Basis}
    \mathbb{D}(A_x[p^{\infty}]) \simeq V_{(p)} \otimes_{\mathbb{Z}_{(p)}} \zpbr
\end{align}
taking $s_{\alpha,\mathrm{cris}}$ to $s_{\alpha} \otimes 1$.

\renewcommand{\shg}{\operatorname{Sh}_G} \renewcommand{\shgsp}{\operatorname{Sh}_{\operatorname{GSp}}} \newcommand{\shgspord}{\operatorname{Sh}_{\operatorname{GSp},\mathrm{ord}}} \newcommand{\sgsp}{\mathcal{S}_{\operatorname{GSp}}} \newcommand{\sg}{\mathscr{S}_G}

\subsubsection{}
Let us now drop the level from the notation and write $\sg$ and $\sgsp$ respectively for the base changes of $\mathscr{S}_K$ and $\mathcal{S}_{\mathcal{K}}$ respectively to $\zpbr$ for some choice of $\mathcal{O}_{E,v} \to \zpbr$. Similarly write $\shg$ for the special fiber of $\mathscr{S}_G$ and $\shgsp$ for the special fiber of $\sgsp$. Let $\shgspord \subset \shgsp$ be the dense open ordinary locus and define the ordinary locus of $\shg$ by $\shgord:=\operatorname{Sh}_G \cap \shgspord$. It is an open subset which is nonempty if and only if $E_v = \mathbb{Q}_p$, by \cite[Cor. 1.0.2]{LeeNewton}. \textbf{We will assume from now on that $E_v = \mathbb{Q}_p$.}
\begin{Lem} \label{Lem:OrdinaryLocusNewtonStratum}
The ordinary locus $\shgord$ is open and dense and equal to the Newton stratum $\shgb$ for $[b] \in B(G, \{\mu^{-1}\})$ the maximal element in the partial order introduced in Section \ref{Sec:NewtonStratification}.
\end{Lem}
The maximal element $[b]$ is known as the \emph{$\mu$-ordinary element}, and the maximal Newton stratum is known as the \emph{$\mu$-ordinary locus}. 
\begin{proof}
The $\mu$-ordinary locus and the ordinary locus are equal in this case by the proof of \cite[Cor. 4.3.2]{LeeNewton}, as explained in \cite[Rem. 4.3.3]{LeeNewton}. The density of the $\mu$-ordinary locus is \cite[Thm. 1.1]{Wortmann}, see \cite[Thm. 3]{KMPS} for a published reference. 
\end{proof}

\subsubsection{} Let $x \in \shg(\ovfp)$ be an ordinary point and consider the closed immersions of formal neighbourhoods (considered as functor on the category $\nilp_{\zpbr}$ of $\zpbr$-algebras where $p$ is nilpotent)
\begin{align} \label{Eq:Inclusion}
    \sg^{/x}&:=\spf \widehat{\mathcal{O}}_{\sg,x} \xhookrightarrow{} \spf \widehat{\mathcal{O}}_{\sgsp,x}=:\sgsp^{/x}.
\end{align}
Let $A$ be the universal abelian scheme over $\sgsp$ and let $X=A[p^{\infty}]$ be associated $p$-divisible group over $\sgsp$. Let $\widehat{\operatorname{Def}}(A_x)$ be the formal deformation space of the abelian variety $A_x$, that is, the formal scheme representing the functor $\operatorname{Def}(A_x)$ on the category $\art$ of deformations of the abelian variety $A_x$. Similarly let $\widehat{\operatorname{Def}}(Y)$ be the deformation space of the $p$-divisible group $X_x=:Y$. There are natural morphisms
\begin{align}
    \sgsp^{/x} \to \widehat{\operatorname{Def}}(A_x) \to \widehat{\operatorname{Def}}(Y).
\end{align}
The first is a closed immersion by the moduli description of $\sgsp$, and the second morphism is an isomorphism by \cite[Thm. 1.2.1]{KatzSerreTate}. Now \cite[Thm 1.1]{ShankarZhou} (c.f. \cite{Noot} for closely related results) implies that  closed formal subscheme $$\sg^{/x} \subset \widehat{\operatorname{Def}}(Y)$$ is a $p$-divisible formal subgroup. The goal of this section is to compute the Dieudonn\'e module of $\shg^{/x}$. We do this by giving a new proof that
\begin{align}
    \shg^{/x} \subset \widehat{\operatorname{Def}}(Y)
\end{align}
is a $p$-divisible formal subgroup, using the methods of Section \ref{Sec:SerreTate} and results of \cite{KimLeaves}.

\subsection{Some covariant Dieudonn\'e theory} \label{Sec:Dieudonne}
\subsubsection{A caveat} In the rest of this subsection we are going to recall some covariant Dieudonn\'e theory for semiperfect rings following \cite{ScholzeWeinstein}. The reason we do this is that the references \cite{CaraianiScholze} and \cite{KimLeaves} are written in this language. Moreover we feel that results such as Lemma \ref{Lem:IntegralPHodgeTheoryTateModuleUniversalCover} are most naturally stated using the covariant theory. \smallskip

To avoid potential confusion, we will always write a subscript $\mathrm{cov}$ when using covariant Dieudonn\'e theory. The covariant theory and the contravariant theory will interact only once, in Section \ref{Sec:DieudonneModuleSerreTateTorus}, and we will warn the reader again there.
\subsubsection{} Recall that an $\fp$-algebra $A$ is \emph{semiperfect} if it is the quotient of a perfect $\fp$-algebra $B$ and that it is \emph{f-semiperfect} if it is the quotient of a perfect $\fp$-algebra by a finitely generated ideal. Let $A$ be a semiperfect $\mathbb{F}_p$-algebra and let $A_{\mathrm{cris}}(A)$ be Fontaine's ring of crystalline periods (see \cite[Prop. 4.1.3]{ScholzeWeinstein}) with $\varphi:A_{\mathrm{cris}}(A) \to A_{\mathrm{cris}}(A)$ induced by the absolute Frobenius on $A$. 
\begin{Def} \label{Def:DieudonneI}
A \emph{covariant Dieudonn\'e module} over a semiperfect $\fp$-algebra $A$ is a pair $(M, \varphi_{M})$, where $M$ is a finite locally free $A_{\mathrm{cris}}(A)$-module and where
\begin{align}
    \varphi_{M}: \varphi^{\ast} M[\tfrac 1p] \to M[\tfrac 1p]
\end{align}
is an isomorphism such that  \begin{align}
    M \subseteq \varphi_{M} (M) \subseteq \tfrac{1}{p} M.
\end{align}
\end{Def}
\begin{Rem}
Usually one instead asks that
\begin{align}
    pM \subseteq \varphi_{M} (M) \subseteq  M.
\end{align}
The reasons for our conventions is that they agree with the conventions in \cite{CaraianiScholze} and \cite{KimLeaves}.

A $p$-divisible group $\mathscr{G}$ over $A$ has a covariant\footnote{We write $\mathbb{D}_{\mathrm{cov}}$ to distinguish from the contravariant Dieudonn\'e theory that we used in Section \ref{Sec:Monodromy}.} Dieudonn\'e module $(\mathbb{D}_{\mathrm{cov}}(\mathscr{G}), \varphi_\mathscr{G})$. For $\spec A' \to \spec A$ there is a canonical isomorphism
\begin{align}
    (\mathbb{D}_{\mathrm{cov}}(\mathscr{G}_{A'}), \varphi_{\mathscr{G}_{A'}}) \simeq (\mathbb{D}_{\mathrm{cov}}(\mathscr{G}), \varphi_\mathscr{G}) \otimes_{A_{\mathrm{cris}}(A)} A_{\mathrm{cris}}(A').
\end{align}
Our covariant Dieudonn\'e modules are normalised as in \cite{CaraianiScholze}. In particular, this means that the covariant Dieudonn\'e module of $\mathbb{Q}_p/\mathbb{Z}_p$ over $A$ is $A_{\mathrm{cris}}(A)$ equipped with the trivial Frobenius, and the covariant Dieudonn\'e module of $\mu_{p^{\infty}}$ is $A_{\mathrm{cris}}(A)$ equipped with Frobenius given by $1/p$. This also means that the contravariant Dieudonn\'e module is isomorphic to the dual of the covariant Dieudonn\'e module, see \cite[footnote on page 692]{CaraianiScholze}.
\end{Rem}
Now let $\mathscr{G}$ be a $p$-divisible group over $\fp$ with universal cover $\widetilde{\mathscr{G}}$ in the sense of \cite[Sec. 3.1]{ScholzeWeinstein}. If we consider $\widetilde{\mathscr{G}}$ as a functor on $\nilp$ then it is a a filtered colimit of spectra of f-semiperfect $\fp$-algebras by \cite[Prop. 3.1.3.(iii)]{ScholzeWeinstein} and is thus determined by its restriction to the category of semiperfect $\fp$-algebras. We can describe it explicitly on the category of f-semiperfect $\fp$-algebras as follows:
\begin{Lem} \label{Lem:IntegralPHodgeTheoryTateModuleUniversalCover}
There is a commutative diagram of natural transformation of functors on the category of f-semiperfect $\fp$-algebras, which evaluated at an object $A$ gives
\begin{equation} \label{eq:IntegralDieudonneDiagram}
    \begin{tikzcd}
    \tilde{\mathscr{G}}(A) \arrow[r, "\simeq"] & \left(B_{\mathrm{cris}}^+(A) \otimes_{\qpbr} \mathbb{D}_{\mathrm{cov}}(\mathscr{G})[\tfrac{1}{p}] \right)^{\varphi=1},
    \end{tikzcd}
\end{equation}
where $\varphi$ is given by the diagonal Frobenius and where $B_{\mathrm{cris}}^+(A):=A_{\mathrm{cris}}(A)[1/p]$.
\end{Lem}
\begin{proof}
Let $A$ be f-semiperfect, then \cite[Thm. 4.1.4]{ScholzeWeinstein} tells us that Dieudonn\'e module functor over $A$ is fully faithful after inverting $p$. There is a natural map 
\begin{align}
    T_p \mathscr{G}(A) &= \operatorname{Hom}_A((\mathbb{Q}_p/\mathbb{Z}_p)_A,\mathscr{G}_A)\\
    &\to \operatorname{Hom}_{A_{\mathrm{cris}},F}(A_{\mathrm{cris}}(A), A_{\mathrm{cris}}(A) \otimes_{\zpbr} \mathbb{D}_{\mathrm{cov}}(\mathscr{G})) \\
    &\simeq \left(A_{\mathrm{cris}}(A) \otimes_{\zpbr} \mathbb{D}_{\mathrm{cov}}(\mathscr{G})\right)^{\varphi=1}
\end{align}
where the latter bijection is induced by evaluation at $1$. After inverting $p$ we get a natural isomorphism
\begin{align}
    \tilde{\mathscr{G}}(A) &= \operatorname{Hom}_A((\mathbb{Q}_p/\mathbb{Z}_p)_A,\mathscr{G}_A)[\tfrac{1}{p}]\\
    &\to \left(B_{\mathrm{cris}}^{+}(A) \otimes_{\qpbr} \mathbb{D}_{\mathrm{cov}}(\mathscr{G})[\tfrac{1}{p}] \right)^{\varphi=1}.
\end{align}
\end{proof}

\subsection{The Dieudonn\'e module of the Serre--Tate torus} \label{Sec:DieudonneModuleSerreTateTorus} Let $x \in \shgord(\ovfp)$ be as above and let $Y=A_x[p^{\infty}]$ be the corresponding $p$-divisible group. Recall from Section \ref{Sec:SerreTate} that $Y=Y_0 \oplus Y_1$ and that both $Y_0$ and $Y_1$ lift uniquely to $p$-divisible groups $Y_0^{\mathrm{can}}$ and $Y_1^{\mathrm{can}}$ over $\zpbr$. Let $\operatorname{Def}(Y)$ be the formal deformation space of $Y$, considered as a functor on $\art$ together with its extension $\widehat{\operatorname{Def}}(Y)$ to $\nilp$. We have seen that $\widehat{\operatorname{Def}}(Y)$ has the structure of a $p$-divisible formal group, and we use $\mathcal{H}_{0,1}$ to denote the corresponding $p$-divisible group over $\spf \zpbr$. 

Consider the special fiber $\mathcal{H}_{0,1,\ovfp}$. Then by Lemma \ref{Lem:TateModuleSerreTate} its $p$-adic Tate module is given by $\Hom(Y_0,Y_1)$. Therefore by \cite[Lem. 4.1.7, Lem. 4.1.8]{CaraianiScholze} we have an isomorphism
\begin{align}
    \mathbb{D}_{\mathrm{cov}}(\mathcal{H}_{0,1,\ovfp})[1/p] \simeq \Hom(\mathbb{D}_{\mathrm{cov}}(Y_0)[1/p], \mathbb{D}_{\mathrm{cov}}(Y_1)[1/p])^{\le 0},
\end{align}
where $\Hom$ denotes the internal hom in $F$-isocrystals and where $()^{\le 0}$ denotes the slope at most $0$ part of an $F$-isocrystal. 

\subsubsection{} Choose an isomorphism (here we use contravariant Dieudonn\'e theory!) $\mathbb{D}(Y) \to V_{(p)} \otimes \zpbr$ sending $s_{\alpha} \otimes 1$ to $s_{x,\mathrm{cris}}$ as in Section \ref{Sec:IntegralTensors}. This induces an isomorphism from $V_{(p)}^{\ast} \otimes \zpbr$ to the covariant Dieudonn\'e module $\mathbb{D}_{\mathrm{cov}}(Y)$ and thus gives us Frobenius invariant tensors $\{s_{\alpha,\mathrm{cris}}\} \subset \mathbb{D}_{\mathrm{cov}}(Y)^{\otimes}$. Let $b \in G(\qpbr) \subset \operatorname{GL}(V^{\ast})(\qpbr)$ be the element corresponding to the Frobenius in $\mathbb{D}_{\mathrm{cov}}(Y)[1/p]$. Then there is an inclusion of $F$-isocrystals 
\begin{align} \label{Eq:InternalHomCrystals}
    \Hom(\mathbb{D}_{\mathrm{cov}}(Y_0)[1/p], \mathbb{D}_{\mathrm{cov}}(Y_1)[1/p]) \subset \Hom(\mathbb{D}_{\mathrm{cov}}(Y)[1/p], \mathbb{D}_{\mathrm{cov}}(Y)[1/p]),
\end{align}
which sends $f:\mathbb{D}_{\mathrm{cov}}(Y_0)[1/p] \to \mathbb{D}_{\mathrm{cov}}(Y_1)[1/p]$ to $$\operatorname{Id}+f:\mathbb{D}_{\mathrm{cov}}(Y_0)[1/p] \oplus \mathbb{D}_{\mathrm{cov}}(Y_1)[1/p] \to \mathbb{D}_{\mathrm{cov}}(Y_0)[1/p] \oplus \mathbb{D}_{\mathrm{cov}}(Y_1)[1/p].$$ This map realises the source as the slope $-1$ part of the target.

\subsubsection{} Write $\mathfrak{gl}(V^{\ast})$ for the Lie algebra of the algebraic group $\operatorname{GL}(V^{\ast}) \otimes \qpbr$ and identify it with the vector space of endomorphisms of $V^{\ast} \otimes \qpbr$ equipped with the commutator bracket. We can equip $\mathfrak{gl}(V^{\ast})$ with the structure of an $F$-isocrystal by letting Frobenius act by conjugation by $b \in \operatorname{GL}(V^{\ast}) \otimes \qpbr$. Let us write $\left(\mathfrak{gl}(V^{\ast}), \operatorname{Ad} b \sigma \right)$
for this isocrystal. 

Using the isomorphism $V_{(p)}^{\ast} \otimes \zpbr \simeq \mathbb{D}_{\mathrm{cov}}(Y)$ as above, we can identify the $F$-isocrystal on the right hand side of \eqref{Eq:InternalHomCrystals} with $\left(\mathfrak{gl}(V^{\ast}), \operatorname{Ad} b \sigma \right)$. There is a sub $F$-isocrystal
\begin{align}
    (\mathfrak{g}, \operatorname{Ad} b \sigma) \subset (\mathfrak{gl}(V^{\ast}), \operatorname{Ad} b \sigma),
\end{align}
where $\mathfrak{g}=\operatorname{Lie} G \otimes \qpbr$. By Lemma \ref{Lem:LieAlgebrasTensors} below, the subspace $\mathfrak{g} \subset \mathfrak{gl}(V^{\ast})$ is precisely the subspace of those endomorphisms $g$ of $V^{\ast} \otimes \qpbr$ that satisfy $g^{\ast} (s_{\alpha} \otimes 1) = 0$ for all tensors $s_{\alpha}$. 
\begin{Lem} \label{Lem:LieAlgebrasTensors}
Let $C$ be a field of characteristic zero and let $W$ be a finite dimensional $C$ vector space. Let $H \subset \operatorname{GL}(W)$ be a connected reductive group that is the stabiliser of a collection of tensors $\{t_{\alpha}\}_{\alpha \in \mathscr{A}} \subset W^{\otimes}$. Then the Lie algebra $\mathfrak{h} \subset \mathfrak{gl}(W)$ is given by the subspace
\begin{align}
 \{H \in \mathfrak{gl}(W) : H^{\ast}(t_{\alpha}) = 0  \text{ for all } \alpha \in \mathscr{A}\}.
\end{align}
\end{Lem}
\begin{proof}
The Lie algebra is given by the kernel of the map $G(C[\epsilon]/(\epsilon^2)) \to G(C)$. Thus it consists of matrices of the form $1 + \epsilon M$, where $M \in \operatorname{gl}(W)$, such that for $\alpha \in \mathscr{A}$ we have
\begin{align}
    (1 + \epsilon M)^{\ast}(t_{\alpha} \otimes 1) = t_{\alpha}.
\end{align}
But is equivalent to $(\epsilon M)^{\ast}(t_{\alpha} \otimes 1)=0$ or $M^{\ast}(t_{\alpha})=0$.
\end{proof}
Let us write
\begin{align} \label{Eq:InclusionDieudonne}
    (\mathfrak{g}, \operatorname{Ad} b \sigma)^{-1} \subset \Hom(\mathbb{D}_{\mathrm{cov}}(Y_0)[1/p], \mathbb{D}_{\mathrm{cov}}(Y_1)[1/p]) = \mathbb{D}_{\mathrm{cov}}(\mathcal{H}_{0,1,\ovfp})[1/p]
\end{align}
for the slope $-1$ subspace of the $F$-isocrystal $(\mathfrak{g}, \operatorname{Ad} b \sigma)$. Then by \cite[Lem. 3.1.3]{KimLeaves} and its proof, there is an inclusion of $p$-divisible groups
\begin{align} \label{eq:InclusionBTGroups}
    \mathcal{H}_{0,1,\ovfp}^G \subset \mathcal{H}_{0,1,\ovfp}
\end{align}
inducing \eqref{Eq:InclusionDieudonne} upon taking rational covariant Dieudonn\'e modules. Since both of these $p$-divisible groups have \'etale Serre duals, there is a unique lift $\mathcal{H}_{0,1}^G$ of $\mathcal{H}_{0,1,\ovfp}^G$ to $\zpbr$ and a unique lift
\begin{align}
    \mathcal{H}_{0,1}^G \subset \mathcal{H}_{0,1}
\end{align}
of the inclusion \eqref{eq:InclusionBTGroups}.

\begin{Lem} \label{Lem:KimFixII}
Let $A$ be an f-semiperfect $\mathbb{F}_p$-algebra. Then the inclusion 
\begin{align} \label{Eq:InclusionUniversalCovers}
    \widetilde{\mathcal{H}}_{0,1,\ovfp}^G(A) \subset \widetilde{\mathcal{H}}_{0,1,\ovfp}(A) = \Hom(Y_{0,A},Y_{1,A})[1/p]
\end{align}
identifies $\widetilde{\mathcal{H}}_{0,1,\ovfp}^G(A)$ with the subspace of those quasi-endomorphisms $f:Y_{0,A}\dashrightarrow Y_{1,A}$ such that the induced quasi-endomorphism
\begin{align}
    g=\TJM{0}{0}{f}{0}:Y_A \dashrightarrow Y_A
\end{align}
induces an endomorphism $\mathbb{D}_{\mathrm{cov}}(Y_A)[1/p] \to \mathbb{D}_{\mathrm{cov}}(Y_A)[1/p]$ satisfying $g^{\ast} ( s_{\alpha,\mathrm{cris}} \otimes 1) =0$.
\end{Lem}
\begin{proof}
This follows from Lemma \ref{Lem:IntegralPHodgeTheoryTateModuleUniversalCover} in combination with  Lemma \ref{Lem:LieAlgebrasTensors}.
\end{proof}
\begin{Rem}
The statement of Lemma \ref{Lem:KimFixII} contradicts \cite[Lem 3.1.3]{KimLeaves}, which implies that the inclusion 
\begin{align} 
    \widetilde{\mathcal{H}}_{0,1,\ovfp}^G(A) \subset \widetilde{\mathcal{H}}_{0,1,\ovfp}(A) = \Hom(Y_{0,A},Y_{1,A})[1/p]
\end{align}
identifies $\widetilde{\mathcal{H}}_{0,1,\ovfp}^G(A)$ with the subspace of those quasi-endomorphisms $f:Y_{0,A}\dashrightarrow Y_{1,A}$ such that the induced quasi-endomorphism
\begin{align}
    g=\TJM{0}{0}{f}{0}:Y_A \dashrightarrow Y_A
\end{align}
induces an endomorphism $\mathbb{D}_{\mathrm{cov}}(Y_A)[1/p] \to \mathbb{D}_{\mathrm{cov}}(Y_A)[1/p]$ satisfying $g^{\ast} ( s_{\alpha,\mathrm{cris}} \otimes 1) =s_{\alpha,\mathrm{cris}}$. This cannot be correct because $\widetilde{\mathcal{H}}_{0,1,\ovfp}^G(A)$ is stable under addition and if $g_1, g_2$ both satisfy $g^{\ast} (s_{\alpha} \otimes 1)=(s_{\alpha} \otimes 1)$ then their sum $g_1+g_2$ does not.
\end{Rem}
The following lemma and its corollary essentially follow from \cite[Prop. 3.2.4]{KimLeaves}. However the construction there is incorrect because of the error in \cite[Lem 3.1.3]{KimLeaves} pointed out above. Once the subgroup in the statement of Lemma \ref{Lem:KimFix} has been shown to exist with the properties proved in Corollary \ref{Cor:KimFixIII}, the rest of the arguments in \cite{KimLeaves} go through without further changes.
\begin{Lem} \label{Lem:KimFix}
There is a closed subgroup
\begin{align}
    \mathbf{Aut}_G(\tilde{Y}) \subset \mathbf{Aut}(\tilde{Y})
\end{align}
such that on $A$-points for f-semiperfect $\ovfp$-algebras $A$, it is the subgroup of those quasi-isogenies $g:Y_A\dashrightarrow Y_A$ that induce isomorphisms $g:\mathbb{D}_{\mathrm{cov}}(Y_A)[1/p] \to \mathbb{D}_{\mathrm{cov}}(Y_A)[1/p]$ satisfying $$g^{\ast} (s_{\alpha,\mathrm{cris}} \otimes 1)=s_{\alpha,\mathrm{cris}} \otimes 1.$$ 
\end{Lem}
We will call such quasi-isogenies \emph{tensor-preserving quasi-isogenies}.
\begin{proof}
First of all by \cite[Lem. 4.2.10]{CaraianiScholze} the functor $\mathbf{Aut}(\tilde{Y})$ satisfies
\begin{align}
    \mathbf{Aut}(\tilde{Y})(R) = \mathbf{Aut}(\tilde{Y})(R/p),
\end{align}
for all $R \in \nilp$. Thus we can define a closed subfunctor of $\mathbf{Aut}(\tilde{Y})$ by specifying its values on $\ovfp$-algebras. 

The matrix description of $\mathbf{Aut}(\tilde{Y})$ in Section \ref{Sec:RapoportZink} gives us a semi-direct product decomposition (see \cite[Prop. 4.2.11, Rem. 4.2.12]{CaraianiScholze})
\begin{align}
    \mathbf{Aut}(\tilde{Y})_{\ovfp}:=\Hom(Y_0, Y_1)[1/p] \rtimes \left(\mathbf{Aut}(\tilde{Y_0})_{\ovfp} \times \mathbf{Aut}(\tilde{Y_1})_{\ovfp} \right).
\end{align}
Here we are using the map
\begin{align}
    \Hom(Y_0,Y_1)[1/p] &\to \mathbf{Aut}(\tilde{Y})_{\ovfp} \\
    f \mapsto \TJM{1}{0}{f}{1}
\end{align}
to realise $\Hom(Y_0,Y_1)[1/p]$ as the subgroup of lower triangular automorphisms of $\tilde{Y}$. The condition that $\TJM{1}{0}{f}{1}=1+f$ satisfies $(1+f)^{\ast} (s_{\alpha,\mathrm{cris}} \otimes 1)=s_{\alpha,\mathrm{cris}} \otimes 1$ is equivalent to the condition that $f^{\ast} (s_{\alpha,\mathrm{cris}} \otimes 1)=0$. Thus we see that the intersection $\Hom(Y_0,Y_1)[1/p]$ with 
$\mathbf{Aut}_G(\tilde{Y})$ is given by 
\begin{align}
     \widetilde{\mathcal{H}}_{0,1,\ovfp}^G \subset \widetilde{\mathcal{H}}_{0,1,\ovfp} = \Hom(Y_{0},Y_{1})[1/p].
\end{align}
By Lemma \ref{Lem:KimFixII}, this is representable by a closed subgroup. \medskip

We can identify the group $\left(\mathbf{Aut}(\tilde{Y_1})_{\ovfp} \times \mathbf{Aut}(\tilde{Y_0})_{\ovfp} \right)$ with the locally profinite group scheme associated to the locally profinite group $\mathbf{Aut}(\tilde{Y})(\ovfp)$. Using Dieudonn\'e theory, we can identify this locally profinite group with the $\sigma$-centraliser of $b$ in $\operatorname{GL}(V^{\ast})(\qpbr)$, where we recall that we have fixed an isomorphism $V_{(p)}^{\ast} \otimes \zpbr \simeq \mathbb{D}_{\mathrm{cov}}(Y)$ giving rise to $b \in G(\qpbr)$. The subgroup of tensor-preserving automorphisms of $\tilde{Y}$ over $\ovfp$ can be identified with $J_b(\qp)$, the $\sigma$-centraliser of $b \in G(\qpbr)$, which is a closed subgroup.

Note that $J_b(\qp) \subset G(\qpbr)$ stabilises $(\mathfrak{g}, \operatorname{Ad} b \sigma)^{-1}$ because it acts on $\mathfrak{g}$ via automorphisms that preserve the slope decomposition. Using Lemma \ref{Lem:KimFixII} we see that the closed subgroup
\begin{align}
    \widetilde{\mathcal{H}}_{0,1,\ovfp}^G \rtimes \underline{J_b(\qp)} \subset \mathbf{Aut}(\tilde{Y})_{\ovfp},
\end{align}
has the required properties over $\ovfp$, and so we are done. 
\end{proof}
Since the $R$-points of $\widetilde{\mathcal{H}}_{0,1}^{G}$ and $\underline{J_b(\qp)}$ both only depend on $R/p$, we see that
\begin{align}
    \widetilde{\mathcal{H}}_{0,1}^G \rtimes \underline{J_b(\qp)} = \mathbf{Aut}_G(\tilde{Y})
\end{align}
described the unique lift to $\zpbr$. This identifies $\widetilde{\mathcal{H}}_{0,1}^G$ with the neutral component $\mathbf{Aut}_G(\tilde{Y})^{\circ}$ of $\mathbf{Aut}_G(\tilde{Y})$.
\begin{Cor} \label{Cor:KimFixIII}
The identity component
\begin{align}
    \widetilde{\mathcal{H}}_{0,1,\ovfp}^G = \mathbf{Aut}_G(\tilde{Y})^{\circ} \subset \mathbf{Aut}_G(\tilde{Y})
\end{align}
is isomorphic to $\spf S$ where $S$ is the $p$-adic completion of $\zpbr[[x_1^{1/p^{\infty}}, \cdots, x_d^{1/p^{\infty}}]]$ for some $d$.
\end{Cor}
\begin{proof}
This is true for $\widetilde{\mathcal{H}}_{0,1,\ovfp}^G$ because it is the universal cover of a $p$-divisible group, see \cite[Cor. 3.1.5, Sec. 6.4]{ScholzeWeinstein}.
\end{proof}

\subsection{Serre--Tate coordinates for Hodge type Shimura varieties} \label{Sec:SubtorusHodge} Recall that $x \in \shg(\ovfp)$ is an ordinary point with associated element $b=b_x \in G(\qpbr)$. Recall also from Section \ref{Sec:Tensors} that we have a $G(\qpbar)$ conjugacy class of cocharacters $\{\mu\}$ coming from the Shimura datum $X$ and the fixed place $v$ of $E$.
\begin{Lem} \label{Lem:NewtonCocharacter}
    The conjugacy class of fractional cocharacters $\{\nu_{[b]}\}$ defined by $[b]$ is equal to $\{\mu^{-1}\}$.
\end{Lem}
\begin{proof}
The ordinary locus is equal to the $\mu$-ordinary locus by Lemma \ref{Lem:OrdinaryLocusNewtonStratum}. Therefore we have that $\{\nu_{[b]}\}=\{\overline{\mu}\}$, where $\{\overline{\mu}\}$ is the Galois-average of $\{\mu^{-1}\}$, see \cite[Sec. 2.1]{ShankarZhou}. But since $G_{\qp}$ is unramified and the local reflex field $E_v$ is equal to $\qp$ there is a cocharacter $\mu$ defined over $\qp$ inducing the conjugacy class of cocharacters $\{\mu\}$. It follows that $\{\overline{\mu}\}=\{\mu^{-1}\}$.
\end{proof}
Let $\{\lambda\}$ be a conjugacy class of (fractional) cocharacters of a connected reductive group $H$ over an algebraically closed field $C$. Let $T$ be a maximal torus, let $\lambda$ be a representative of $\{\lambda\}$ factoring through $T$ and let $B \supset T$ be a Borel. Let $\rho \in X^{\ast}(T)$ be the half sum of the positive roots with respect to $B$. Then the pairing $\langle 2 \rho, \lambda \rangle$ does not depend on the choice of $T, B$ or $\lambda$, and we denote it by $\langle 2 \rho, \{\lambda\} \rangle$. 
\begin{Cor} \label{Cor:Dimension} The $p$-divisible formal group $\mathcal{H}_{0,1,\ovfp}^G$ has dimension $\langle 2 \rho, \{\mu\} \rangle$
\end{Cor}
\begin{proof}
The dimension of $\mathcal{H}_{0,1,\ovfp}^G$ is equal to $\langle 2 \rho, \nu_{[b]} \rangle$ by \cite[Prop. 3.1.4]{KimLeaves}, which is equal to $\langle 2 \rho, \{\mu^{-1}\} \rangle=\langle 2 \rho, \{\mu\} \rangle$ by Lemma \ref{Lem:NewtonCocharacter}.
\end{proof}
\begin{Prop} \label{Prop:DieudonneModuleSerreTateTorusHodge}
The closed formal subscheme
\begin{align}
    \sg^{/x} \xhookrightarrow{} \sgsp^{/x} \xhookrightarrow{} \widehat{\operatorname{Def}}(Y)
\end{align}
introduced in \eqref{Eq:Inclusion}, is a $p$-divisible formal subgroup. The induced inclusion of $p$-divisible groups 
\begin{align}
    \sg^{/x}[p^{\infty}] \subset \widehat{\operatorname{Def}}(Y)[p^{\infty}]=\mathcal{H}_{0,1}
\end{align}
induces the inclusion 
\begin{align}
    (\mathfrak{g}, \operatorname{Ad} b \sigma)^{-1} \subset \Hom(\mathbb{D}_{\mathrm{cov}}(Y_0)[1/p], \mathbb{D}_{\mathrm{cov}}(Y_1)[1/p])
\end{align}
from \eqref{Eq:InclusionDieudonne} on rational covariant Dieudonn\'e modules of their special fibers. 
\end{Prop}
\begin{proof}
By \cite[Thm. 4.3.1]{KimLeaves}, the closed formal subscheme $\sg^{/x} \subset \widehat{\operatorname{Def}}(Y)$ is stable under the action of
\begin{align}
        \mathbf{Aut}_{G}(\tilde{Y})^{\circ} \subset \mathbf{Aut}(\tilde{Y})^{\circ}.
\end{align}
We can identify these groups with
\begin{align} \label{Eq:InclusionUniversalCoversLifted}
    \widetilde{\mathcal{H}}_{0,1}^G \subset \widetilde{\mathcal{H}}_{0,1}
\end{align}
By Proposition \ref{Prop:SerreTateVSRapoportZink}, the action of $\widetilde{\mathcal{H}}_{0,1}$ on $\widehat{\operatorname{Def}}(Y)$ factors through the natural action of $\mathcal{H}_{0,1}$ on $\widehat{\operatorname{Def}}(Y)$ by left translation, via the natural quotient map
\begin{align}
    \widetilde{\mathcal{H}}_{0,1} \to \mathcal{H}_{0,1}.
\end{align}
The inclusion $\mathcal{H}_{0,1}^G \subset \mathcal{H}_{0,1}$ induces an inclusion $T_p \mathcal{H}_{0,1}^G \subset T_p \mathcal{H}_{0,1}$ which induces \eqref{Eq:InclusionUniversalCoversLifted}
after inverting $p$. This implies that the action of $\widetilde{\mathcal{H}}_{0,1}^G$ on $\sg^{/x}$ factors through an action of $\mathcal{H}_{0,1}^G$ via the natural quotient map
\begin{align}
    \widetilde{\mathcal{H}}_{0,1}^G \to \mathcal{H}_{0,1}^G.
\end{align}
Now consider the closed point $\{x\} \in \shg^{/x}$. Then the associated orbit map gives a closed immersion
\begin{align}
    \mathcal{H}_{0,1,\ovfp} \xhookrightarrow{} \operatorname{Def}(Y)_{\ovfp}.
\end{align}
This means that we similarly get a closed immersion
\begin{align}
    \mathcal{H}_{0,1,\ovfp}^G \subset \shg^{/x}.
\end{align}
By \cite[Prop. 3.1.4]{KimLeaves}, the formal scheme $\operatorname{Def}(Y)_{G,\ovfp}$ has dimension $\langle 2 \rho, \{\nu_{[b]}\} \rangle$, which is equal to $\langle 2 \rho, \{\mu\} \rangle$ by Lemma \ref{Lem:NewtonCocharacter}, which in turn is equal to the dimension of $\shg^{/x}$. It follows that the orbit map induces an isomorphism
\begin{align}
    \mathcal{H}_{0,1,\ovfp}^G \to \shg^{/x}
\end{align}
and that $\shg^{/x}$ is a formal subgroup of $\operatorname{Def}(Y)_{\ovfp}$ satisfying the conclusions of the proposition. It remains to show that $\sg^{/x} \subset \operatorname{Def}(Y)$ is a formal subgroup, which follows from \cite[Thm. 1.1]{ShankarZhou}.
\end{proof}

\subsubsection{The action of automorphism groups} \label{Sec:ActionOfAutomorphisms} Let the notation be as in Section \ref{Sec:SerreTateHodge}. Recall that we have fixed an isomorphism $\mathbb{D}_{\mathrm{cov}}(Y) \simeq V_{(p)}^{\ast} \otimes \zpbr$ sending $s_{\alpha} \otimes 1$ to $s_{\alpha,\mathrm{cris}}$. This gives us an element $b \in G(\qpbr) \subset \operatorname{GL}(V^{\ast})(\qpbr)$ corresponding to the Frobenius in $\mathbb{D}_{\mathrm{cov}}(Y)[1/p]$.

Recall from Section \ref{Sec:SerreTate} that there is an action of $\mathbf{Aut}(\tilde{Y})$ on $\operatorname{RZ}_Y$. Recall from the discussion before Corollary \ref{Cor:UsualActionDefY}, that $\underline{\mathbf{Aut}(Y)(\ovfp)} \subset \mathbf{Aut}(\tilde{Y})$ preserves the $\ovfp$ points corresponding to the identity from $Y$ to $Y$, and that this induces an action of the profinite group $\mathbf{Aut}(Y)(\ovfp)$ on $\operatorname{Def}(Y)$. This action is described in Corollary \ref{Cor:UsualActionDefY}.
\subsubsection{} Recall that there are locally profinite groups
\begin{equation}
    \begin{tikzcd}
        J_b(\mathbb{Q}_p) \arrow[r, hook] \arrow[d, hook] & G(\qpbr) \arrow{d} \\
        \operatorname{GL}(V^{\ast})_b(\qp) \arrow[r, hook] & \operatorname{GL}(V^{\ast})(\qpbr),
    \end{tikzcd}
\end{equation}
where $\operatorname{GL}(V^{\ast})_b(\qp)=\mathbf{Aut}(\tilde{Y})(\ovfp)$ is the $\sigma$-centraliser of $b$ in $\operatorname{GL}(V^{\ast})(\qpbr)$. Let us write $U_p \subset J_b(\mathbb{Q}_p)$ for the compact open subgroup given by the intersection
\begin{align}
    \mathbf{Aut}_G(\tilde{Y})(\ovfp) \cap \mathbf{Aut}(Y)(\ovfp).
\end{align}
Then $U_p$ acts on $\widetilde{\mathcal{H}}_{0,1}^G \subset \Hom(Y_0,Y_1)[1/p]$ and preserves the action of $T_p \mathcal{H}_{0,1}^G$, and thus acts on the quotient $\mathcal{H}_{0,1} \simeq \sg^{/x}$. By Proposition \ref{Prop:DieudonneModuleSerreTateTorusHodge} and the proof of Lemma \ref{Lem:KimFix}, the induced action on rational Dieudonn\'e modules can be identified with the natural action of $U_p \subset J_b(\mathbb{Q}_p)$ on 
\begin{align}
    (\mathfrak{g}, \operatorname{Ad} b \sigma)^{-1} \subset (\mathfrak{g}, \operatorname{Ad} b \sigma).
\end{align}
In order to apply the rigidity result of Chai \cite{ChaiRigidity} we need to understand this action. We will do this in more generality in the next section. 

\subsection{Strongly non-trivial actions} \label{Sec:NontrivialActions} Let $G$ be a connected reductive group over $\qp$. Let $b \in G(\qpbr)$ be an element and consider the $F$-isocrystal $(\mathfrak{g}, \operatorname{Ad} b \sigma)$, where $\mathfrak{g}=\operatorname{Lie} G \otimes \qpbr$, equipped with its action of $J_b(\qp)$. If we replace $b$ by a $\sigma$-conjugate $b'$, then $J_b(\qp)$ and $J_{b'}(\qp)$ are conjugate in $G(\qpbr)$ and there is an isomorphism of isocrystals $(\mathfrak{g}, \operatorname{Ad} b \sigma) \simeq (\mathfrak{g}, \operatorname{Ad} b' \sigma)$.

Let $\lambda \in \mathbb{Q}$ and let $N_{\lambda} \subset (\mathfrak{g}, \operatorname{Ad} b \sigma)$ be the largest sub $F$-isocrystal of slope $\lambda$. Then because $J_b(\qp)$ acts on $(\mathfrak{g}, \operatorname{Ad} b \sigma)$ via $F$-isocrystal automorphisms, it preserves the subspace $N_{\lambda}$. Let us also denote by $b$ the image of $b$ in $\operatorname{GL}(\mathfrak{g})$, then there is a homomorphism of algebraic groups
\begin{align}
    J_b \to \operatorname{GL}(\mathfrak{g})_b,
\end{align}
where $\operatorname{GL}(\mathfrak{g})_b$ denotes the $\sigma$-centraliser of $b$ in $\operatorname{GL}(\mathfrak{g})$. There is a parabolic subgroup
\begin{align}
    P(\lambda) \subset \operatorname{GL}(\mathfrak{g})
\end{align}
consisting of automorphisms of $\mathfrak{g}$ that preserve the slope filtration on the $F$-isocrystal $(\mathfrak{g}, \operatorname{Ad} b \sigma)$, and after potentially replacing $b$ by a $\sigma$-conjugate, 
the image of $b$ lands in $P(\lambda)$. There is thus a group homomorphism
\begin{align}
    J_b \to P(\lambda)_b,
\end{align}
where $P(\lambda)_b$ denotes the $\sigma$-centraliser of $b$ in $P(\lambda)$. Since $N_{\lambda}$ is a graded quotient of the slope filtration of the $F$-isocrystal $(\mathfrak{g}, \operatorname{Ad} b \sigma)$, there is an induced quotient map $P(\lambda) \to \operatorname{GL}(N_{\lambda})$ and this induces a group homomorphism
\begin{align}
    J_b \to \operatorname{GL}(N_{\lambda})_b,
\end{align}
where $\operatorname{GL}(N_{\lambda})_b$ denotes the $\sigma$-centraliser of $b$ in $\operatorname{GL}(N_{\lambda})$. Let $E$ be the $\qp$-algebra of endomorphisms of the $F$-isocrystal $N_{\lambda}$ and let $E^{\times}$ be the functor on $\qp$-algebras given by $R \mapsto (R \otimes E)^{\times}$. Then there is a natural isomorphism $E^{\times} \simeq \operatorname{GL}(N_{\lambda})_b$.

Let $\operatorname{GL}(E)$ be the general linear group of $E$ considered as $\qp$-vector space and let $E^{\times} \to \operatorname{GL}(E)$ be the natural map corresponding to the action of $E$ on itself by left translation. Consider $E$ as a $\qp$-linear representation of $J_b$ via $J_b \to E^{\times}$, then the goal of this section is to prove the following result:
\begin{Prop} \label{Prop:NonTrivialII}
Let $T \subset J_b$ be a maximal torus. If $\lambda \not=0$, then the trivial representation of $T$ does not occur in the representation of $T$ given by $E$. 
\end{Prop}
\begin{proof}[Proof of Proposition \ref{Prop:NonTrivialII}]
After replacing $b$ by a $\sigma$-conjugate we can arrange for it to satisfy (c.f. \cite[Sec. 4]{Kottwitz1})
\begin{align}
    b \sigma(b) \cdots \sigma^{r-1}(b) = (r \nu_b)(p)
\end{align}
for some $r$. Here $\nu_b$ is the Newton cocharacter of $b$, which is defined over $\mathbb{Q}_{p^r}$. Let $M_{\nu_b} \subset G \otimes \qpbr$ denote the centraliser of the cocharacter $\nu_b$. By \cite[Prop. 2.2.6]{KimLeaves}, there is a unique isomorphism
\begin{align}
    J_b \otimes \qpbr \to M_{\nu_b}
\end{align}
such the composition $J_b(\qp) \subset J_b(\qpbr) \to M_{\nu_b}(\qpbr) \subset G(\qpbr)$ is the defining inclusion of $J_b(\qp)$ as the $\sigma$-centraliser of $b$ in $G(\qpbr)$. \smallskip

After tensoring up to $\qpbr$, there is a commutative diagram (where $L^{\mathrm{reg}}$ is the left regular representation of $\operatorname{GL}(N_{\lambda})$ on $\operatorname{GL}(\operatorname{End}(N_{\lambda}))$)
\begin{equation}
\begin{tikzcd}
    & \operatorname{GL}(E_{\qpbr}) \arrow[r, "\simeq"] & \operatorname{GL}(\operatorname{End}(N_{\lambda})) \\
    J_{b,\qpbr} \arrow{r} \arrow[d, "\simeq"] & E^{\times}_{\qpbr} \arrow{u} \arrow[r, "\simeq"] & \operatorname{GL}(N_{\lambda}) 
    \arrow{u}{\operatorname{L}_{\mathrm{reg}}} . \\
    M_{\nu_b} \arrow{r}{} & P(\lambda) \arrow{ur}
\end{tikzcd} 
\end{equation}
If we show that the trivial representation of $T \otimes \qpbr$ does not occur in $E \otimes \qpbr$, then it follows that the trivial representation of $T$ does not occur in $E$. The representation $W=\operatorname{End}(N_{\lambda})$ of $J_{b,\qpbr}$ defined by composition with the left regular representation is a direct sum of copies of the representation $N_{\lambda}$. Therefore it suffices to show that the representation $N_{\lambda}$ of $T \otimes \qpbr$ does not contain the trivial representation. \medskip

We note that $T \otimes \qpbr=:T'$ is a maximal torus of $M_{\nu_b}$ acting on the associated graded of the slope filtration of the $F$-isocrystal $(\mathfrak{g}, \operatorname{Ad} b \sigma)$. Since $\nu_b$ is a central cocharacter of $M_{\nu_b}$ by definition, we see that $(r \nu_b)(p) \in T'(\qpbr)$. To determine the slope decomposition of the $F$-isocrystal $(\mathfrak{g}, \operatorname{Ad} b \sigma)$, it suffices to determine the slope decomposition of the $F^r$-isocrystal
\begin{align}
    \left(\mathfrak{g},(\operatorname{Ad} b \sigma)^r\right)
\end{align}
for some positive integer $r$. \medskip

Let $C$ be an algebraic closure of $\qpbr$ and consider the action of $T'_C$ on $\mathfrak{g}_C$ via the adjoint action of $G_C$. Then we have a decomposition
\begin{align}
    \mathfrak{g}_C \simeq \mathfrak{t}'_C \oplus \left(\bigoplus_{\alpha \in \Phi} U_{\alpha} \right),
\end{align}
where $\Phi \subset X^{\ast}(T_C')$ consists of the simple roots of $T_C$. There is a similar decomposition
\begin{align} \label{Eq:DecompositionIntoRoots}
    \mathfrak{g} \simeq \mathfrak{t}' \oplus \left(\bigoplus_{\alpha_0 \in \Phi_0} U_{\alpha_0} \right),
\end{align}
where $\Phi_0 \in X^{\ast}(T_C')_I$ is the image of $\Phi$ and where $I=\operatorname{Gal}(C/\qpbr)$ is the inertia group. \medskip

Now we choose an integer $r$ with the following properties: The isomorphism $J_b \otimes \qpbr \to M_{\nu_b}$ is defined over $\mathbb{Q}_{p^r}$, the equation 
\begin{align}
    b \sigma(b) \cdots \sigma^{r-1}(b) = (r \nu_b)(p)
\end{align}
is satisfied, and the decomposition \eqref{Eq:DecompositionIntoRoots} is defined over $\mathbb{Q}_{p^r}$. Then each $U_{\alpha_0}$ is stable under the action of $\sigma^r$ and $(\operatorname{Ad} b \sigma)^r$ acts on it by $(r \nu_b)(p) \sigma^r$. The operator $\operatorname{Ad} b \sigma$ moreover acts trivially on $\mathfrak{t}'$, and thus for nonzero $\lambda$ we have that
\begin{align}
    N_{\lambda} \subset \bigoplus_{\alpha_0} U_{\alpha_0}.
\end{align}
After basechanging to $C$, we see that
\begin{align}
    N_{\lambda} \subset \bigoplus_{\alpha} U_{\alpha}.
\end{align}
Thus $T'_C$ acts on $N_{\lambda}$ via a subset of the nontrivial characters given by the simple roots $\Phi \subset X^{\ast}(T'_C)$, and therefore the trivial representation of $T'$ does not occur in $N_{\lambda}$ and thus it does not occur in $E$.

\end{proof}

\section{Proof of the main theorems} \label{Sec:ProofOfTheorems}
There are two final ingredients that are introduced in this section. In Section \ref{Sec:Rigidity}, we prove the local stabiliser principle of Chai--Oort (c.f. \cite[Thm. 9.5]{ChaiOortNotes}), which shows that the formal completion of the Zariski closure of a prime-to-$p$ Hecke orbit is stable under the action of a large $p$-adic Lie group. In Section \ref{Sec:Coordinates} we give a summary of results of \cite{ChaiOrdinary}, which relates Serre--Tate coordinates of families of ordinary abelian varieties to the $p$-adic monodromy groups of these abelian varieties. Then in Section \ref{Sec:Conclusion} we put everything together to prove Theorem \ref{Thm:Ordinary}. In Section \ref{Sec:AbelianType} we prove Corollary \ref{Cor:AbelianType}, which is a generalisation of Theorem \ref{Thm:Ordinary} to Shimura varieties of abelian type. \medskip

\renewcommand{\shg}{\operatorname{Sh}_{G,K^pK_p}}
\renewcommand{\shgord}{\operatorname{Sh}_{G,\mathrm{ord},K^pK_p}}\renewcommand{\shgsp}{\operatorname{Sh}_{\operatorname{GSp},\mathcal{K}^p\mathcal{K}_p}} \renewcommand{\shgspord}{\operatorname{Sh}_{\operatorname{GSp},\mathrm{ord},\mathcal{K}^p \mathcal{K}_p}} \renewcommand{\sgsp}{\mathcal{S}_{\operatorname{GSp}},\mathcal{K}^p\mathcal{K}_p} We will use the notation introduced in Section \ref{Sec:IntModels} and Section \ref{Sec:IntegralModelsHyperspecial} and moreover we will keep track of the level again. Let $x \in \shg(\ovfp)$ and let $\tilde{x}$ be a lift of $x$ to $\shginf(\ovfp)$. Then the \emph{prime-to-$p$ Hecke orbit} of $x$ is defined to be the image $H_{K^p}(x) \subset \shg(\ovfp)$ of the orbit $G(\afp) \cdot \tilde{x} \subset \shginf(\ovfp)$; it does not depend on the choice of lift $\tilde{x}$. For the rest of this section we fix $x$ as above and we let $Z \subset \shgord$ be the closure of $H_{K^p}(x)$, then $Z$ is again $G(\afp)$-stable by Lemma \ref{Lem:ClosureStable}.

\subsection{Rigidity of Zariski closures of Hecke orbits} \label{Sec:Rigidity}
Let $z \in Z(\ovfp)$ smooth point of $Z$ and let $I_z(\mathbb{Q})$ be the group of self quasi-isogenies of $z$ respecting the tensors, which was introduced in Section \ref{Sec:Centralisers}. Let $Y=A_z[p^{\infty}]$ and fix a choice of isomorphism $\mathbb{D}_{\mathrm{cov}}(Y) \simeq V_{(p)}^{\ast} \otimes \zpbr$ sending $s_{\alpha} \otimes 1$ to $s_{\alpha,\mathrm{cris}}$ as in Section \ref{Sec:DieudonneModuleSerreTateTorus}. This gives rise to an element $b_z=b \in G(\qpbr)$ and we let $U_p \subset J_b(\qp)$ be the compact open subgroup introduced in Section \ref{Sec:ActionOfAutomorphisms}. Let $I_z(\mathbb{Z}_{(p)})$ be the intersection of $I_z(\mathbb{Q})$ with $U_p$ inside $J_b(\mathbb{Q}_p)$. We consider the closed immersion of formal neighbourhoods (where the notation is as in \eqref{Eq:Inclusion})
\begin{align}
    Z^{/z} \subset \shg^{/z} \subset \shgsp^{/z}.
\end{align}
The goal of this section is to prove the following result.
\begin{Prop}[Local stabiliser principle] \label{Prop:LocalStabiliserPrinciple}
The closed subscheme $Z^{/z} \subset \shg^{/z}$ is stable under the action of $I_z(\mathbb{Z}_{(p)})$ via $I_z(\mathbb{Z}_{(p)}) \to U_p$.
\end{Prop}
\subsubsection{} There is a $G(\afp)$-equivariant closed immersion (using the fact that we have a closed immersion at finite level by the main theorem of \cite{xu2020normalization}) $\shginf \to \shgspinf$, where $G(\afp)$ acts on the right hand side via the inclusion $G(\afp) \to \mathcal{G}_V(\afp)$. The space $\shgspinf$ is a moduli space of polarised abelian varieties $(A, \lambda)$ up to prime-to-$p$ isogeny, equipped with an isomorphism $V^p A \to V \otimes \afp$ compatible with the polarisation. 

Let $\tilde{z}$ be a lift of $z$ to $\shginf(\ovfp)$ as above, which defines an inclusion
\begin{align}
    I_z(\mathbb{Z}_{(p)}) \subset I_z(\mathbb{Q}) \subset G(\afp).
\end{align}
The stabiliser in $\mathcal{G}_V(\afp)$ of $\tilde{z} \in \shgspinf$ is given by $\operatorname{End}_{\lambda}(A_z)^{\times}$, 
which is the group of automorphisms of the abelian variety up to prime-to-$p$ isogeny $A$ that take $\lambda$ to a $\mathbb{Z}_{(p)}^{\times}$ multiple of $\lambda$. 
\begin{Lem} \label{Lem:StabiliserAtInfinity}
The stabiliser inside $G(\afp)$ of the point $\tilde{z}$ is equal to $I_z(\mathbb{Z}_{(p)})$. 
\end{Lem}
\begin{proof}
By \cite[Lem. 2.1.4]{KMPS}, the stabiliser is contained in $I_z(\mathbb{Z}_{(p)})$. The stabiliser in $\mathcal{G}_V(\afp)$ of the image of $\tilde{z}$ in $\shgspinf$ is equal to $\operatorname{End}_{\lambda}(A_z)^{\times}$ and thus contains $I_z(\mathbb{Z}_{(p)})$. The result follows.
\end{proof}
In order to prove Proposition \ref{Prop:LocalStabiliserPrinciple}, we first prove it for $\shgsp$. See \cite[Thm. 9.5]{ChaiOortNotes} for closely related results and arguments. \smallskip

Let $\shgspinf^{/\tilde{z}}$ be the formal completion of $\shgspinf$, considered as a formal algebraic space as in \cite[Sec. 0AIX]{stacks-project}, and restrict its functor of points to Artin local $\ovfp$-algebras $R$ with residue field isomorphic to $\ovfp$. Then $\shgspinf^{/\tilde{z}}(R)$ is the set of isomorphism classes of polarised abelian varieties $(A, \lambda)$ up to prime-to-$p$ isogeny, equipped with an isomorphism $\epsilon:V^p A \to V \otimes \afp$ compatible with the polarisation, such that after basechanging to $\ovfp$ we recover the point given by the image of $\tilde{z}$.

This means that there is a (necessarily unique) isomorphism $\beta:A_{\ovfp} \to A_z$ making the following diagram commute:
\begin{equation} \label{Eq:FormalCompletionShgspinf}
\begin{tikzcd}
V^p A_{\ovfp} \arrow{r}{\beta} \arrow{d}{\epsilon} & V^p A_z \arrow{d}{\epsilon_{\tilde{z}}} \\
V \otimes \afp \arrow[r, equals] & V \otimes \afp.
\end{tikzcd}
\end{equation}
The quadruple $(A, \lambda, \beta,\epsilon)$ is uniquely determined by $(A, \lambda, \beta)$ because (pro-)\'etale sheaves on Artin local rings are determined by their restriction to the residue field. In particular, for all $R \in \art$ the natural forgetful map
\begin{align}
    \shgspinf^{/\tilde{z}}(R) \to \shgsp^{/z}(R)
\end{align}
is an isomorphism. This induces an action of $\operatorname{End}_{\lambda}(A_z)^{\times}$ on $\shgsp^{/z}$ that we will now identify.

\subsubsection{} Recall that there is an inclusion $\operatorname{End}_{\lambda}(A_z)^{\times} \subset \mathcal{G}_V(\afp)$ determined by $\tilde{z}$ or rather $\epsilon_{\tilde{z}}$. This means that an automorphism $f$ of $A_z$ acts on $V \otimes \afp$ in a way that makes the following diagram commute:
\begin{equation}
    \begin{tikzcd}
        V^p A_z \arrow{r}{f} \arrow{d}{\epsilon_{\tilde{z}}} & V^p A_z \arrow{d}{\epsilon_{\tilde{z}}} \\
        V \otimes \afp \arrow{r}{f} & V \otimes \afp.
    \end{tikzcd}
\end{equation}
Since $\operatorname{End}_{\lambda}(A_z)^{\times}$ stabilises $\tilde{z}$, it acts on $\shgspinf^{/\tilde{z}}$. This action can be described as follows: An automorphism $f$ sends a triple $(A, \lambda, \epsilon)$ to $(A, \lambda, f \circ \epsilon)$. It is straightforward to check that the unique upgrade $(A, \lambda, f \circ \epsilon)$ to a quadruple $(A, \lambda, \beta', f \circ \epsilon)$ is realised by taking $\beta'=f \circ \beta$. Therefore the induced action of $\operatorname{End}_{\lambda}(A_z)^{\times}$ on $\shgsp^{/z}$ is given by $(A, \lambda, \beta) \mapsto (A, \lambda, f \circ \beta)$. 

\subsubsection{} \label{Sec:IdentificationOfAction} Because deformations of abelian varieties are uniquely determined by deformations of their $p$-divisible groups, we can also identify 
\begin{align}
    \shgspinf^{/\tilde{z}}(R)
\end{align}
with the space of triples $(X, \lambda, \beta)$ where $(X,\lambda)$ is a polarised $p$-divisible group and $\beta$ is an isomorphism $(X,\lambda)_{\ovfp} \to (A_z[p^{\infty}],\lambda_z)$. The action of  $\operatorname{End}_{\lambda}(A_z)^{\times}$ is then given by postcomposing $\beta$ with $f$. There is a similar description of at finite level, and it follows that the natural map
\begin{align}
    \shgsp^{/z} \subset \operatorname{Def}(A_z[p^{\infty}])
\end{align}
is $\operatorname{End}_{\lambda}(A_z)^{\times}$-equivariant, where $\operatorname{End}_{\lambda}(A_z)^{\times}$ acts on the RHS via the inclusion
\begin{align}
    \operatorname{End}_{\lambda}(A_z)^{\times} \subset \operatorname{End}(A_z[p^{\infty}])^{\times},
\end{align}
followed by the natural action of $\operatorname{End}(A_z[p^{\infty}])^{\times}$ on $\operatorname{Def}(A_z[p^{\infty})]$. 

\begin{proof}[Proof of Proposition \ref{Prop:LocalStabiliserPrinciple}]
Let $\tilde{z}$ be a lift of $z$ to $\shginf(\ovfp)$ as above, which defines an inclusion
\begin{align}
    I_z(\mathbb{Z}_{(p)}) \subset I_z(\mathbb{Q}) \subset G(\afp).
\end{align}
It follows from Lemma \ref{Lem:StabiliserAtInfinity} that $I_z(\mathbb{Z}_{(p)}) \subset G(\afp)$ is the stabiliser of the point $\tilde{z}$ under the action of $G(\afp)$. Let $\tilde{Z}$ be the inverse image of $Z$ in $\shginf$, it is stable under the action of $G(\afp)$ by Lemma \ref{Lem:ClosureStable}. There is a commutative diagram
\begin{equation}
    \begin{tikzcd}
        \tilde{Z} \arrow{d} \arrow{r} & \shginf \arrow{r} \arrow{d} & \shgspinf \arrow{d} \\
        Z \arrow{r} & \shg \arrow{r} & \shgsp.
    \end{tikzcd}
\end{equation}
where the top right horizontal map is $G(\afp)$-equivariant via $G(\afp) \to \mathcal{G}_V(\afp)$. 

Let $\tilde{Z}^{/\tilde{z}}$ be the formal completion of $\tilde{Z}$ at the closed point corresponding to $\tilde{Z}$, considered as a formal algebraic space as in \cite[Sec. 0AIX]{stacks-project}. This is per definition the subfunctor of $\tilde{Z}$ consisting of those morphisms $\spec T \to \tilde{Z}$ that factor through $\tilde{z}$ on the level of topological spaces. Since $I_z(\mathbb{Z}_{(p)})$ stabilises $\tilde{z}$, it acts on $\tilde{Z}^{/\tilde{z}}$.

By \cite[Lem. 0CUF]{stacks-project}, there is a homeomorphism $|\tilde{Z}| \simeq \varprojlim_{K^p} |Z_{K^p}|$ and thus we get an isomorphism
\begin{align}
    \tilde{Z}^{/\tilde{z}} \simeq \varprojlim_{K^p \subset G(\afp)} Z_{K^p}^{/z},
\end{align}
where $z \in Z_{K^p}(\ovfp)$ is the image of $\tilde{z}$ under $\tilde{Z} \to Z_{K^p}$. The formal algebraic space $Z_{K^p}^{/z}$ can be identified with $\spf \widehat{\mathcal{O}}_{Z_{K^p},z}$, compatible with changing $K^p$. Since the transition morphisms are all finite \'etale, they induce isomorphisms of complete local rings. Therefore, all the transition maps in the inverse system $\varprojlim_{K^p \subset G(\afp)} Z_{K^p}^{/z}$ are isomorphism. We conclude that
\begin{align}
    \tilde{Z}^{/\tilde{z}} \simeq Z_{K^p}^{/z},
\end{align}
and so there is an action of $I_z(\mathbb{Z}_{(p)})$ on $Z_{K^p}^{/z}$ In the same way we can prove that there is an action of $I_z(\mathbb{Z}_{(p)})$ on $\shg^{/z}$. It remains for us to identify this action with the inclusion $I_z(\mathbb{Z}_{(p)}) \to U_p$ followed by the natural action of $U_p$ on $\shg^{/z}$.

Let $\tilde{z}$ be the image of $\tilde{z}$ in $\shgspinf(\ovfp)$ and let $z \in \shgsp(\ovfp)$ its image. Then the stabiliser of $\tilde{z}$ can be identified with the group
\begin{align}
    \operatorname{End}_{\lambda}(A_z)^{\times} \subset \mathcal{G}_V(\afp)
\end{align}
as before. The discussion above implies that we have an action of $\operatorname{End}_{\lambda}(A_z)^{\times}$ on $\shgsp^{/z}$ such that the closed immersion
\begin{align}
    \shg^{/z} \subset \shgsp^{/z}
\end{align}
is $I_z(\mathbb{Z}_{(p)})$-equivariant for the action of $I_z(\mathbb{Z}_{(p)})$ on the right hand side via the map $I_z(\mathbb{Z}_{(p)}) \to \operatorname{End}_{\lambda}(A_z)^{\times}$. But we have seen in Section \ref{Sec:IdentificationOfAction} that the action of $\operatorname{End}_{\lambda}(A_z)^{\times}$ on $\shgsp^{/z}$ described above agrees with the action of $\operatorname{End}_{\lambda}(A_z)^{\times}$ via the inclusion $\operatorname{End}_{\lambda}(A_z)^{\times} \to \mathbf{Aut}_{\lambda}(A_z[p^{\infty}])(\ovfp)$. 

Note that the following diagram commutes by construction
\begin{equation}
\begin{tikzcd}
    I_z(\mathbb{Z}_{(p)})  \arrow{r} \arrow{d} & \operatorname{End}_{\lambda}(A_z)^{\times} \arrow{d} \\
    U_p \arrow{r} & \mathbf{Aut}_{\lambda}(A_z[p^{\infty}])(\ovfp).
\end{tikzcd} 
\end{equation}
Thus we see that $Z_{K^p}^{/z}$ is stable under the action of $I_z(\mathbb{Z}_{(p)})$ on $\shg^{/x}$ given by the inclusion $I_z(\mathbb{Z}_{(p)}) \to U_p$ followed by the natural action of $U_p$ on $\shg^{/x}$.
\end{proof}
\begin{Cor} \label{Cor:Rigidity}
Assume that $z \in \shg(\ovfp)$ is an ordinary point. Then $Z^{/z}$ is a formal subtorus of $\shg^{/z}$. 
\end{Cor}
\begin{proof}
We know that $U_p$ acts on the deformation space $\shg^{/z}$ and we see that $Z^{/z}$ is stable under the action of $I_z(\mathbb{Z}_{(p)})$ and hence of its closure in $U_p$. The algebraic group $I_{\mathbb{Q}_p} \subset J_b$ has the same rank as $J_b$ by \cite[Cor. 2.1.7]{KisinPoints}. Let $T \subset I_{\mathbb{Q}}$ be a maximal torus, then \cite[Thm. 7.9]{PR} tells us that the topological closure of $T(\mathbb{Q})$ in $T(\qp)$ has finite index in $T(\qp)$. It follows from this that the closure of $I_z(\mathbb{Z}_{(p)})$ in $U_p$ contains a compact open subgroup of a maximal torus $T$ of $J_b(\mathbb{Q}_p)$. 

Proposition \ref{Prop:NonTrivialII} then tells us that the assumptions of \cite[Thm. 4.3]{ChaiRigidity} are satisfied. This theorem implies that $Z^{/z}$ is a $p$-divisible formal subgroup of $\shg^{/z}$, in other words, it is a formal subtorus. 
\end{proof}
\subsection{Monodromy of linear subspaces} The goal of this section is to prove the following result, which is a consequence of results of \cite{ChaiOrdinary}. Recall that the universal abelian variety $A$ over $\shgspord$ gives rise to an $F$-isocrystal $\mathcal{M}$, see Section \ref{Sec:PAdicMonodromy}. Let $Z \subset \shgspord$ be a closed subscheme, then we say that $Z$ is \emph{linear} at a smooth point $z \in Z(\ovfp)$ if $$Z^{/z} \subset \shgsp^{/z}$$ is a $p$-divisible formal subgroup. Let $U_Z$ be the unipotent radical of the monodromy group $\operatorname{Mon}(Z, \mathcal{M},z)$.
\begin{Prop}[Chai] \label{Prop:MonodromyLinear}
Let $z \in Z(\ovfp)$ be a smooth point such that $Z$ is linear at $z$. Then we have the inequality
\begin{align}
    \operatorname{dim} Z_{z} \ge \operatorname{Dim} U_Z,
\end{align}
where $\operatorname{Dim} Z_z$ is the dimension of the local ring $\mathcal{O}_{Z,x}$.
\end{Prop}
Chai proves the stronger statement that this inequality is actually an equality, but we will not need this stronger statement to prove Theorem \ref{Thm:Ordinary}. \smallskip

Our proof of Proposition \ref{Prop:MonodromyLinear} is a straightforward application of the results in \cite[Sec. 2-4]{ChaiOrdinary}. Since \cite{ChaiOrdinary} is an unpublished preprint from 2003, the referee has suggested we include another reference. Thus we give a second proof of Proposition \ref{Prop:MonodromyLinear} based on results of \cite{DAddeziovanHoften}. 

\subsubsection{} \label{Sec:Coordinates} For our first proof of Proposition \ref{Prop:MonodromyLinear}, we need to give a brief summary of \cite[Sec. 2-4]{ChaiOrdinary}. Consider the closed immersion. 
\begin{align}
    Z^{/z} \to \shgsp^{/i(z)} \xhookrightarrow{} \widehat{\operatorname{Def}}(Y)_{\ovfp},
\end{align}
Write $R=\widehat{\mathcal{O}_{Z,z}}$ and write $M$ for the finite free $\zp$-module $T_p Y_0(\ovfp) \otimes_{\zp} T_p Y_1^{\vee}(\ovfp)$. Then the morphism $Z^{/z} \to \widehat{\operatorname{Def}}(Y)$ corresponds to an element of
\begin{align}
    \widehat{\operatorname{Def}}(Y)(R) &= \operatorname{Hom}(M, \widehat{\mathbb{G}}_m(R))\\
    &=\operatorname{Hom}(M, 1+\mathfrak{m}_R)
\end{align}
where the first equality is \cite[Thm. 2.1]{KatzSerreTate}. Thus we get a homomorphism $f:M \to 1+\mathfrak{m}_R$ and we let $N_z^{\vee}$ be its kernel. By \cite[Prop. 4.2.1, Rem. 2.5.1]{ChaiOrdinary}, the $\zp$-module $N_z^{\vee}$ is finite free and the quotient $M/N_z^{\vee}$
is torsion-free. Thus the map
\begin{align}
    Z^{/z} \to \Hom(M, \widehat{\mathbb{G}}_m)
\end{align}
factors through the subtorus 
\begin{align}
    \Hom \left( M/N_z^{\vee} , \widehat{\mathbb{G}}_m \right) \subset \Hom(M, \widehat{\mathbb{G}}_m),
\end{align}
which we can also write as $N_z \otimes_{\zp} \widehat{\mathbb{G}}_m \subset M^{\ast} \otimes_{\zp} \widehat{\mathbb{G}}_m$. Here the $\ast$ denotes taking $\zp$-linear dual and the morphism $N_z \to M^{\ast}$ is the $\zp$-linear dual of the quotient
\begin{align}
    M \to M/N_z^{\vee}.
\end{align}
The following lemma has the same statement as \cite[Rem. 3.14]{ChaiOrdinary}.
\begin{Lem} \label{Lem:SmallestSubtorus}
The subgroup $N_z \otimes_{\zp} \widehat{\mathbb{G}}_m$ is the smallest formal subtorus of $\widehat{\operatorname{Def}}(Y)_{\ovfp}$ through which the map from $\spf R$ factors. 
\end{Lem}
\begin{proof}
A subtorus corresponds to a free $\zp$-submodule $N \subset N_z$ such that the quotient $N_z/N$ is torsion free. Write $N^{\vee}$ for the kernel of the map $$M \to M/N_z^{\vee} = N_z^{\ast} \to N^{\ast}.$$ Then if
 \begin{align}
    \spf R \to N_z \otimes_{\zp} \widehat{\mathbb{G}}_m
\end{align}
factors through $N \otimes_{\zp} \widehat{\mathbb{G}}_m$, it factors through
\begin{align}
    \Hom \left( M/N^{\vee}, \widehat{\mathbb{G}}_m \right) \subset \Hom(M, \widehat{\mathbb{G}}_m).
\end{align}
Since the kernel of the map $M \to \widehat{\mathbb{G}}_m(R)$ is given by $N_z^{\vee}$, it follows that $N_z^{\vee}=N^{\vee}$ and therefore $N=N_z$.
\end{proof}

\begin{proof}[Proof of Proposition \ref{Prop:MonodromyLinear}]
We specialise the discussion of Section \ref{Sec:Coordinates} to the situation of Proposition \ref{Prop:MonodromyLinear}. In particular, since $Z^{/z}$ is assumed to be a formal subtorus, we are in the situation that 
\begin{align}
    Z^{/z}= N_z \otimes_{\zp} \widehat{\mathbb{G}}_m \subset \widehat{\operatorname{Def}}(Y)_{\ovfp}.
\end{align}
Chai proves in \cite[Sec. 4, Thm. 4.4]{ChaiOrdinary} that the dimension of the $U_Z$ is equal to the rank of $N_z$. Thus the rank of $N_z$ is certainly bounded from below by the dimension of $U_Z$. But the rank of $N_z$ is also the dimension of the formal scheme $Z^{/z}$ which equals the Krull dimension of $\widehat{\mathcal{O}}_{Z,z}$ and also the Krull dimension of $\mathcal{O}_{Z,z}$, which proves the theorem.
\end{proof}

\begin{proof}[Second proof of Proposition \ref{Prop:MonodromyLinear}]
By \cite[Theorem II]{DAddeziovanHoften}, the unipotent radical $U_Z$ of $\operatorname{Mon}(Z,\mathcal{M},z)$ is isomorphic to the monodromy group
\begin{align}
    \operatorname{Mon}(Z^{/z},\mathcal{M},z).
\end{align}
This monodromy group is defined as in Section \ref{Sec:PAdicMonodromy} using the Tannakian category of isocrystals over the formal scheme $Z^{/z}$ or equivalently the Tannakian category of isocrystals over the scheme $\spec \widehat{\mathcal{O}}_{Z,z}$, see \cite[Notation 2.2.5]{DAddeziovanHoften}). Thus it suffices to show that the dimension of $Z^{/z}$ is bounded from below by the dimension of $\operatorname{Mon}(Z^{/z},z)$. \medskip

Let $Y=A_z[p^{\infty}]$ as above and write $\mathfrak{a}^+=:\mathbb{D}_{\mathrm{cov}}(Y)$ and $\mathfrak{a}=\mathfrak{a}^+[1/p]$. Write $\mathfrak{b}^+ \subset \mathfrak{a}^+$ for the covariant Dieudonn\'e module of the $p$-divisible group associated to $Z^{/z}$ and $\mathfrak{b}=\mathfrak{b}^+[1/p]$. Then in the notation of \cite[Sec. 5.5]{DAddeziovanHoften} we have
\begin{align}
    Z^{/z}=Z(\mathfrak{b}^+).
\end{align}
Now \cite[Thm. 5.5.3]{DAddeziovanHoften} tells us that there is an inclusion of algebraic groups over $\qpbr$
\begin{align}
    \operatorname{Mon}(Z^{/z},\mathcal{M} ,z) \subset U(\mathfrak{b}):=\mathfrak{b} \otimes_{\qpbr} \mathbb{G}_a.
\end{align}
in particular, the height of the isocrystal $\mathfrak{b}$ is bounded from below by the dimension of $\operatorname{Mon}(Z^{/z},\mathcal{M},z)$. Since $\mathfrak{b}$ has slope $1$, it follows that the dimension of the $p$-divisible group associated to $\mathfrak{b}^+$ is also bounded from below by the dimension of $\operatorname{Mon}(Z^{/z},\mathcal{M},z)$.
\end{proof}

\subsection{Monodromy and conclusion} \label{Sec:Conclusion}
Recall from Section \ref{Sec:ProductDecompositionKottwitzSet} the maps
\begin{align}
    B(G_{\qp}) \to B(G^{\mathrm{ad}}_{\qp}) \to \prod_{i=1}^n B(G_{i,\qp}).
\end{align}
induced by the decomposition $G^{\mathrm{ad}}=\prod_{i=1}^n G_i$ of \eqref{eq:ProductDecomposition}. Let $[b_{\mathrm{ord}}] \in \bgmu$ be the $\sigma$-conjugacy class corresponding to the ordinary locus, and let $[b_{\mathrm{ord},i}]$ be the image of $[b_{\mathrm{ord}}]$ in $B(G_{i,\qp})$.
\begin{Lem} \label{Lem:NonBasic}
    For all $i$ the element $[b_{\mathrm{ord},i}]$ is non-basic.
\end{Lem}
\begin{proof}
    By the axioms of a Shimura datum, the $G_i(\qpbar)$-conjugacy class of cocharacters $\{\mu_i^{-1}\}$ induced by $\{\mu^{-1}\}$ is nontrivial for all $i$. By Lemma \ref{Lem:NewtonCocharacter}, we have an equality $\{\mu_i^{-1}\}=\{\nu_{[b_{\mathrm{ord},i}]}\}$ and so the Newton cocharacter of $[b_{\mathrm{ord},i}]$ is noncentral for all $i$. In other words, the $\sigma$-conjugacy class $[b_{\mathrm{ord},i}]$ is non-basic for all $i$.
\end{proof}

\begin{proof}[Proof of Theorem \ref{Thm:Ordinary}]
Let $x \in \shg(\ovfp)$ be an ordinary point and let $Z$ be the Zariski closure (inside $\shgord$) of its prime-to-$p$ Hecke orbit. Then $Z$ is $G(\afp)$-stable by Lemma \ref{Lem:ClosureStable} and similarly its smooth locus $Z^{\mathrm{sm}} \subset Z$ is $G(\afp)$-stable by Lemma \ref{Lem:DevSmooth}. Let $X$ be the $p$-divisible group over $Z^{\mathrm{sm}}$ of the universal abelian variety and let $\mathcal{M}^{\dagger}$ be the associated overconvergent $F$-isocrystal, see Section \ref{Sec:PAdicMonodromy}.

By Lemma \ref{Lem:NonBasic}, the element $[b_{\mathrm{ord}}]$ is $\mathbb{Q}$-non-basic and by Lemma \ref{Lem:HypothesisHyperspecial}, we know that Hypothesis \ref{Hyp:Levi} is satisfied because $K_p$ is hyperspecial. Therefore Corollary \ref{Cor:padicMonodromyI} tells us that the monodromy group of $\mathcal{M}^{\dagger}$ over $Z^{\mathrm{sm}}$ is isomorphic to $G^{\mathrm{der}} \otimes \qpbr$. Corollary \ref{Cor:padicMonodromyII} tells us that unipotent radical of the monodromy group of $\mathcal{M}$ over $Z^{\mathrm{sm}}$ is isomorphic to the unipotent radical of the parabolic subgroup $P_{\nu_{[b]}} \subset G \otimes \qpbr$ for any choice of $\nu_{[b]} \in \{\nu_{[b]}\}$. \medskip

By Lemma \ref{Lem:NewtonCocharacter}, this unipotent radical is isomorphic to the unipotent radical of the parabolic subgroup $P_{\mu} \subset G$ for any choice of representative $\mu$ of $\{\mu\}$. This unipotent radical has dimension equal to $\langle 2 \rho, \{\mu\} \rangle$ (this notation was introduced after the statement of Lemma \ref{Lem:NewtonCocharacter}).

Corollary \ref{Cor:Rigidity} tells us that $Z^{/z}$ is a formal subtorus.
Applying Proposition \ref{Prop:MonodromyLinear} we see that the Krull dimension of $\mathcal{O}_{Z,z}$ is bounded from below by $\langle 2 \rho, \{\mu\} \rangle$. Since the Shimura variety $\shg$ also has dimension $\langle 2 \rho, \{\mu\} \rangle$, we conclude that $$Z^{/z}=\shg^{/z}.$$ Because this is true for a dense set of points, it follows that $Z$ is a union of connected components of $\shgord$.

By Lemma \ref{Lem:OrdinaryLocusNewtonStratum}, the ordinary locus is dense and thus $\pi_0(\shgord)=\pi_0(\shg)$. Since $G(\afp)$ acts transitively on $\pi_0(\shginf)$, by \cite[Lem. 2.2.5]{KisinModels} in combination with \cite[Cor. 4.1.11]{MP}, it follows that $Z=\shgord$. We conclude that the prime-to-$p$ Hecke orbit of $x$ is dense in $\shg$ since $\shgord$ is dense in $\shg$.
\end{proof}
\subsection{Consequences for Shimura varieties of abelian type} \label{Sec:AbelianType}
Let $(G,X)$ be a Shimura datum of abelian type with reflex field $E$, and let $(G^{\mathrm{ad}}, X^{\mathrm{ad}})$ be the induced adjoint Shimura datum with reflex field $E^{\mathrm{ad}} \subset E$. Let $p$ be a prime number, let $K_p \subset G(\qp)$ be a hyperspecial subgroup and let $K^p \subset G(\afp)$ be a sufficiently small compact open subgroup. Let $\operatorname{Sh}_{G,K^pK_p}$ be the special fiber of the canonical integral model of the Shimura variety of level $K^pK_p$ at a prime $v$ above $p$ of $E$, constructed by Kisin in \cite{KisinModels} (see \cite{MadapusiPeraKim} for the case $p=2$). \medskip

By \cite[Thm. A]{ShenZhang}, there is an open and dense $G(\afp)$-stable Newton stratum $\operatorname{Sh}_{G,K^pK_p,\mu\mathrm{-ord}}$ in $\operatorname{Sh}_{G,K^pK_p}$, called the \emph{$\mu$-ordinary locus.} If $(G,X) \subset (\mathcal{G}_V, \mathcal{H}_V)$ for some symplectic space $V$ and $E_v=\mathbb{Q}_p$, then the $\mu$-ordinary locus is equal to the ordinary locus by Lemma \ref{Lem:OrdinaryLocusNewtonStratum}.
\begin{Cor} \label{Cor:AbelianType}
If $E^{\mathrm{ad}}_v=\qp$, then the prime-to-$p$ Hecke orbit of $x \in \operatorname{Sh}_{G,K^pK_p,\mu\mathrm{-ord}}(\ovfp)$ is dense in $\operatorname{Sh}_{G,K^pK_p}$.
\end{Cor}
\begin{Lem} \label{Lem:ReductionToAdjoint}
Corollary \ref{Cor:AbelianType} holds for $(G,X)$ if and only if it holds for $(G^{\mathrm{ad}}, X^{\mathrm{ad}})$.
\end{Lem}
\begin{proof}
The image $K_{p}^{\mathrm{ad}}$ in $G^{\mathrm{ad}}(\qp)$ is a hyperspecial subgroup. We can choose $K^{p,\mathrm{ad}} \subset G^{\mathrm{ad}}(\afp)$ containing the image of $K^p$ such that there is a morphism
\begin{align}
    \mathbf{Sh}_{G,K^pK_p}(G,X) \to \mathbf{Sh}_{G^{\mathrm{ad}},K^{p,\mathrm{ad}}K_p^{\mathrm{ad}}}(G^{\mathrm{ad}}, X^{\mathrm{ad}}) \otimes_{E^{\mathrm{ad}}} E,
\end{align}
inducing a morphism on special fibers of integral canonical models
\begin{align} \label{Eq:InducedMapAdjointI}
    \operatorname{Sh}_{G,K^pK_p} \to \operatorname{Sh}_{G^{\mathrm{ad}},K^{p,\mathrm{ad}}K_p^{\mathrm{ad}}},
\end{align}
where we are taking the canonical integral model of $(G^{\mathrm{ad}}, X^{\mathrm{ad}})$ at the place $v^{\mathrm{ad}}$ of $E^{\mathrm{ad}}$ induced by $v$. This morphism induces a $G(\afp)$-equivariant morphism
\begin{align}
    \operatorname{Sh}_{G,K_p} \to \operatorname{Sh}_{G^{\mathrm{ad}},K_p^{\mathrm{ad}}},
\end{align}
where $G(\afp)$ acts on the left hand side via the natural map $G(\afp) \to G^{\mathrm{ad}}(\afp)$. Since the Newton stratification on Shimura varieties of abelian type can be constructed using the $F$-crystals with $G$-structure of Lovering \cite{Lovering}, which are functorial for morphisms of Shimura data, it follows that there is an induced map
\begin{align} \label{Eq:InducedMapAdjointII}
    \operatorname{Sh}_{G,K^pK_p,\mu\mathrm{-ord}} \to \operatorname{Sh}_{G^{\mathrm{ad}},K^{p,\mathrm{ad}}K_p^{\mathrm{ad}},\mu\mathrm{-ord}}.
\end{align}
By construction of the integral canonical models of Shimura varieties of abelian type, see \cite[Sec. 3.4.9]{KisinModels} and \cite[Appendix E.7]{KisinPoints}, the connected components of $\operatorname{Sh}_{G^{\mathrm{ad}},K^{p,\mathrm{ad}}K_p^{\mathrm{ad}}}$ are quotients of connected components of $\operatorname{Sh}_{G,K^pK_p}$ by free actions of finite groups. In particular, the map \eqref{Eq:InducedMapAdjointI} is finite \'etale and thus closed. \smallskip

Because the map \eqref{Eq:InducedMapAdjointI} is closed and takes prime-to-$p$ Hecke orbits to prime-to-$p$ Hecke orbits, it must takes Zariski closures of prime-to-$p$ Hecke orbits to Zariski closure of prime-to-$p$ Hecke orbits. We see that the dimensions of Zariski closures of Hecke orbits of points in the $\mu$-ordinary loci are the same for $(G,X)$ and $(G^{\mathrm{ad}},X^{\mathrm{ad}})$.
Moreover in both cases the prime-to-$p$ Hecke operators act transitively on $\pi_0(\shg)$ by \cite[Lem. 2.2.5]{KisinModels} in combination with \cite[Cor. 4.1.11]{MP}.\footnote{The result \cite[Cor. 4.1.11]{MP} states that for Shimura varieties of Hodge type, the geometric special fibre and the geometric generic fibre of the canonical integral model have the same number of connected components. Since the canonical integral models of Shimura varieties of abelian type are constructed from the canonical integral models of Shimura varieties of Hodge type, the statement therefore also holds for Shimura varieties of abelian type.} Thus prime-to-$p$ Hecke orbits in $\operatorname{Sh}_{G,K^pK_p,\mu\mathrm{-ord}}$ are dense if and only if their images under \eqref{Eq:InducedMapAdjointI} are dense. In particular, if the corollary holds for $(G^{\mathrm{ad}},X^{\mathrm{ad}})$, then it holds for $(G,X)$.

To prove the converse, we note that a point in the Shimura variety for $(G^{\mathrm{ad}},X^{\mathrm{ad}})$ can, by \cite[Lem. 2.2.5]{KisinModels} in combination with \cite[Cor. 4.1.11]{MP}, be moved to a connected component which is in the image of \eqref{Eq:InducedMapAdjointI}. Therefore every prime-to-$p$ Hecke orbit can be lifted to the Shimura variety for $(G,X)$, and we are done. 
\end{proof}
\begin{proof}[Proof of Corollary \ref{Cor:AbelianType}]
By Lemma \ref{Lem:ReductionToAdjoint}, we may assume that $G$ is adjoint. Then by the proof of \cite[Lem. 4.6.22]{KisinPappas} we can choose a Shimura datum of Hodge type $(G_2,X_2)$ and a morphism of Shimura data $(G_2,X_2) \to (G,X)$ such that: The group $G_{2,\qp}$ is quasi-split and split over an unramified extension and the prime $v$ of $E$ splits in the reflex field $E_2 \supset E$ of $(G_2,X_2)$. The upshot is that we can choose a prime $w$ of the $E_2$ satisfying $E_{2,w}=\qp$ and thus the $\mu$-ordinary locus in the special fiber of the canonical integral model for $(G_2,X_2)$ at this prime is equal to the ordinary locus for a choice of Hodge embedding $(\mathcal{G}_V, \mathcal{H}_V)$.

Then Theorem \ref{Thm:Ordinary} implies that Corollary \ref{Cor:AbelianType} holds for $(G_2,X_2)$ and Lemma \ref{Lem:ReductionToAdjoint} tells us that it also holds for $(G,X)$ which concludes the proof.
\end{proof}

\DeclareRobustCommand{\VAN}[3]{#3}
\bibliographystyle{amsalpha}
\bibliography{references}

\begin{bibdiv}
\begin{biblist}

\bib{Berthelot}{article}{
      author={Berthelot, P.},
       title={Cohomologie rigide et cohomologie rigide à supports propres.
  {P}remière partie},
        date={1996},
     journal={{P}républication {IRMAR}},
      volume={96-03},
       pages={89 pp.},
        note={available at
  \url{https://perso.univ-rennes1.fr/pierre.berthelot/}},
}

\bib{Borel}{book}{
      author={Borel, Armand},
       title={Linear algebraic groups},
     edition={Second},
      series={Graduate Texts in Mathematics},
   publisher={Springer-Verlag, New York},
        date={1991},
      volume={126},
        ISBN={0-387-97370-2},
         url={https://doi.org/10.1007/978-1-4612-0941-6},
}

\bib{BorelTits}{article}{
      author={Borel, Armand},
      author={Tits, Jacques},
       title={Groupes r\'{e}ductifs},
        date={1965},
        ISSN={0073-8301},
     journal={Inst. Hautes \'{E}tudes Sci. Publ. Math.},
      number={27},
       pages={55\ndash 150},
         url={http://www.numdam.org/item?id=PMIHES_1965__27__55_0},
      review={\MR{207712}},
}

\bib{CaraianiScholze}{article}{
      author={Caraiani, Ana},
      author={Scholze, Peter},
       title={On the generic part of the cohomology of compact unitary
  {S}himura varieties},
        date={2017},
        ISSN={0003-486X},
     journal={Ann. of Math. (2)},
      volume={186},
      number={3},
       pages={649\ndash 766},
         url={https://doi.org/10.4007/annals.2017.186.3.1},
      review={\MR{3702677}},
}

\bib{ChaiOrdinary}{misc}{
      author={Chai, Ching-Li},
       title={Families of ordinary abelian varieties: canonical coordinates,
  $p$-adic monodromy, tate-linear subvarieties and hecke orbits},
        note={available at
  \url{https://www2.math.upenn.edu/~chai/papers_pdf/fam_ord_av.pdf}},
}

\bib{ChaiOrdinarySiegel}{article}{
      author={Chai, Ching-Li},
       title={Every ordinary symplectic isogeny class in positive
  characteristic is dense in the moduli},
        date={1995},
        ISSN={0020-9910},
     journal={Invent. Math.},
      volume={121},
      number={3},
       pages={439\ndash 479},
         url={https://doi.org/10.1007/BF01884309},
      review={\MR{1353306}},
}

\bib{ChaiElladicmonodromy}{article}{
      author={Chai, Ching-Li},
       title={Monodromy of {H}ecke-invariant subvarieties},
        date={2005},
        ISSN={1558-8599},
     journal={Pure Appl. Math. Q.},
      volume={1},
      number={2, Special Issue: In memory of Armand Borel. Part 1},
       pages={291\ndash 303},
         url={https://doi.org/10.4310/PAMQ.2005.v1.n2.a4},
      review={\MR{2194726}},
}

\bib{ChaiConjecture}{incollection}{
      author={Chai, Ching-Li},
       title={Hecke orbits as {S}himura varieties in positive characteristic},
        date={2006},
   booktitle={International {C}ongress of {M}athematicians. {V}ol. {II}},
   publisher={Eur. Math. Soc., Z\"{u}rich},
       pages={295\ndash 312},
      review={\MR{2275599}},
}

\bib{ChaiRigidity}{article}{
      author={Chai, Ching-Li},
       title={A rigidity result for {$p$}-divisible formal groups},
        date={2008},
        ISSN={1093-6106},
     journal={Asian J. Math.},
      volume={12},
      number={2},
       pages={193\ndash 202},
         url={https://doi.org/10.4310/AJM.2008.v12.n2.a3},
      review={\MR{2439259}},
}

\bib{ChaiOortNotes}{incollection}{
      author={Chai, Ching-Li},
      author={Oort, Frans},
       title={Moduli of abelian varieties and {$p$}-divisible groups},
        date={2009},
   booktitle={Arithmetic geometry},
      series={Clay Math. Proc.},
      volume={8},
   publisher={Amer. Math. Soc., Providence, RI},
       pages={441\ndash 536},
      review={\MR{2498069}},
}

\bib{ConradGabberPrasad}{book}{
      author={Conrad, Brian},
      author={Gabber, Ofer},
      author={Prasad, Gopal},
       title={Pseudo-reductive groups},
     edition={Second},
      series={New Mathematical Monographs},
   publisher={Cambridge University Press, Cambridge},
        date={2015},
      volume={26},
        ISBN={978-1-107-08723-1},
         url={https://doi.org/10.1017/CBO9781316092439},
      review={\MR{3362817}},
}

\bib{DAddezioII}{article}{
      author={{D'Addezio}, Marco},
       title={{Parabolicity conjecture of $F$-isocrystals}},
     journal={Ann. of Math. (2), to appear},
      eprint={2012.12879},
}

\bib{d2020monodromy}{article}{
      author={D'Addezio, Marco},
       title={The monodromy groups of lisse sheaves and overconvergent
  {$F$}-isocrystals},
        date={2020},
        ISSN={1022-1824},
     journal={Selecta Math. (N.S.)},
      volume={26},
      number={3},
         url={https://doi.org/10.1007/s00029-020-00569-3},
}

\bib{DAddeziovanHoften}{article}{
      author={{D'Addezio}, Marco},
      author={{\VAN{Hoften}{Van}{van}}~Hoften, Pol},
       title={{Hecke orbits on Shimura varieties of Hodge type}},
        date={2022-05},
     journal={arXiv e-prints},
      eprint={2205.10344},
}

\bib{WeilII}{article}{
      author={Deligne, Pierre},
       title={La conjecture de {W}eil. {II}},
        date={1980},
        ISSN={0073-8301},
     journal={Inst. Hautes \'{E}tudes Sci. Publ. Math.},
      number={52},
       pages={137\ndash 252},
         url={http://www.numdam.org/item?id=PMIHES_1980__52__137_0},
      review={\MR{601520}},
}

\bib{Etesse}{article}{
      author={\'{E}tesse, Jean-Yves},
       title={Descente \'{e}tale des {$F$}-isocristaux surconvergents et
  rationalit\'{e} des fonctions {$L$} de sch\'{e}mas ab\'{e}liens},
        date={2002},
        ISSN={0012-9593},
     journal={Ann. Sci. \'{E}cole Norm. Sup. (4)},
      volume={35},
      number={4},
       pages={575\ndash 603},
         url={https://doi.org/10.1016/S0012-9593(02)01099-6},
      review={\MR{1981173}},
}

\bib{LecturesLieGroups}{article}{
      author={{Glockner}, Helge},
       title={{Lectures on Lie groups over local fields}},
        date={2008-04},
     journal={arXiv e-prints},
      eprint={0804.2234},
}

\bib{HeZhouZhu}{article}{
      author={{He}, Xuhua},
      author={{Zhou}, Rong},
      author={{Zhu}, Yihang},
       title={{Stabilizers of irreducible components of affine Deligne--Lusztig
  varieties}},
        date={2021-09},
     journal={arXiv e-prints},
      eprint={2109.02594},
}

\bib{HoftenZhou}{article}{
      author={{\VAN{Hoften}{Van}{van} Hoften}, Pol},
      author={Zhou, Rong},
       title={{Mod $p$ points on {S}himura varieties of parahoric level}},
        date={2020},
     journal={arXiv e-prints},
      eprint={2010.10496},
}

\bib{HoweUnipotent}{article}{
      author={Howe, Sean},
       title={A unipotent circle action on {$p$}-adic modular forms},
        date={2020},
     journal={Trans. Amer. Math. Soc. Ser. B},
      volume={7},
       pages={186\ndash 226},
         url={https://doi.org/10.1090/btran/52},
      review={\MR{4170572}},
}

\bib{KasprowitzMonodromy}{article}{
      author={{Kasprowitz}, Ralf},
       title={{Monodromy of subvarieties of PEL-Shimura varieties}},
        date={2012-09},
     journal={arXiv e-prints},
       pages={arXiv:1209.5891},
      eprint={1209.5891},
}

\bib{KatzSerreTate}{incollection}{
      author={Katz, N.},
       title={Serre-{T}ate local moduli},
        date={1981},
   booktitle={Algebraic surfaces ({O}rsay, 1976--78)},
      series={Lecture Notes in Math.},
      volume={868},
   publisher={Springer, Berlin-New York},
       pages={138\ndash 202},
      review={\MR{638600}},
}

\bib{KatzSlope}{incollection}{
      author={Katz, Nicholas~M.},
       title={Slope filtration of {$F$}-crystals},
        date={1979},
   booktitle={Journ\'{e}es de {G}\'{e}om\'{e}trie {A}lg\'{e}brique de {R}ennes
  ({R}ennes, 1978), {V}ol. {I}},
      series={Ast\'{e}risque},
      volume={63},
   publisher={Soc. Math. France, Paris},
       pages={113\ndash 163},
      review={\MR{563463}},
}

\bib{KimLeaves}{article}{
      author={Kim, Wansu},
       title={On central leaves of {H}odge-type {S}himura varieties with
  parahoric level structure},
        date={2019},
        ISSN={0025-5874},
     journal={Math. Z.},
      volume={291},
      number={1-2},
       pages={329\ndash 363},
         url={https://doi.org/10.1007/s00209-018-2086-1},
      review={\MR{3936073}},
}

\bib{MadapusiPeraKim}{article}{
      author={Kim, Wansu},
      author={Madapusi~Pera, Keerthi},
       title={2-adic integral canonical models},
        date={2016},
     journal={Forum Math. Sigma},
      volume={4},
       pages={Paper No. e28, 34},
         url={https://doi.org/10.1017/fms.2016.23},
      review={\MR{3569319}},
}

\bib{KisinPappas}{article}{
      author={Kisin, M.},
      author={Pappas, G.},
       title={Integral models of {S}himura varieties with parahoric level
  structure},
        date={2018},
        ISSN={0073-8301},
     journal={Publ. Math. Inst. Hautes \'{E}tudes Sci.},
      volume={128},
       pages={121\ndash 218},
         url={https://doi.org/10.1007/s10240-018-0100-0},
      review={\MR{3905466}},
}

\bib{KisinModels}{article}{
      author={Kisin, Mark},
       title={Integral models for {S}himura varieties of abelian type},
        date={2010},
        ISSN={0894-0347},
     journal={J. Amer. Math. Soc.},
      volume={23},
      number={4},
       pages={967\ndash 1012},
         url={https://doi.org/10.1090/S0894-0347-10-00667-3},
      review={\MR{2669706}},
}

\bib{KisinPoints}{article}{
      author={Kisin, Mark},
       title={{${\rm mod}\,p$} points on {S}himura varieties of abelian type},
        date={2017},
        ISSN={0894-0347},
     journal={J. Amer. Math. Soc.},
      volume={30},
      number={3},
       pages={819\ndash 914},
         url={https://doi.org/10.1090/jams/867},
      review={\MR{3630089}},
}

\bib{KMPS}{article}{
      author={Kisin, Mark},
      author={Pera, Keerthi~Madapusi},
      author={Shin, Sug~Woo},
       title={{Honda–Tate theory for Shimura varieties}},
        date={2022},
     journal={Duke Mathematical Journal},
      volume={171},
      number={7},
       pages={1559 \ndash  1614},
         url={https://doi.org/10.1215/00127094-2021-0063},
}

\bib{KisinShinZhu}{article}{
      author={{Kisin}, Mark},
      author={{Shin}, Sug~Woo},
      author={{Zhu}, Yihang},
       title={{The stable trace formula for Shimura varieties of abelian
  type}},
        date={2021-10},
     journal={arXiv e-prints},
      eprint={2110.05381},
}

\bib{KisinZhou}{article}{
      author={{Kisin}, Mark},
      author={{Zhou}, Rong},
       title={{Independence of $\ell$ for Frobenius conjugacy classes attached
  to abelian varieties}},
        date={2021-03},
     journal={arXiv e-prints},
      eprint={2103.09945},
}

\bib{KottwitzConjugacy}{article}{
      author={Kottwitz, Robert~E.},
       title={Rational conjugacy classes in reductive groups},
        date={1982},
     journal={Duke Math. J.},
      volume={49},
      number={4},
       pages={785\ndash 806},
         url={http://projecteuclid.org/euclid.dmj/1077315531},
      review={\MR{683003}},
}

\bib{Kottwitz1}{article}{
      author={Kottwitz, Robert~E.},
       title={Isocrystals with additional structure},
        date={1985},
        ISSN={0010-437X},
     journal={Compositio Math.},
      volume={56},
      number={2},
       pages={201\ndash 220},
         url={http://www.numdam.org/item?id=CM_1985__56_2_201_0},
      review={\MR{809866}},
}

\bib{KretShin}{article}{
      author={{Kret}, Arno},
      author={{Shin}, Sug~Woo},
       title={{$H^0$ of Igusa varieties via automorphic forms}},
        date={2021-02},
     journal={arXiv e-prints},
      eprint={2102.10690},
}

\bib{LeeNewton}{article}{
      author={Lee, Dong~Uk},
       title={Nonemptiness of {N}ewton strata of {S}himura varieties of {H}odge
  type},
        date={2018},
        ISSN={1937-0652},
     journal={Algebra Number Theory},
      volume={12},
      number={2},
       pages={259\ndash 283},
         url={https://doi.org/10.2140/ant.2018.12.259},
      review={\MR{3803703}},
}

\bib{Lovering}{article}{
      author={{Lovering}, Tom},
       title={{Filtered F-crystals on Shimura varieties of abelian type}},
        date={2017-02},
     journal={arXiv e-prints},
      eprint={1702.06611},
}

\bib{MP}{article}{
      author={Madapusi~Pera, Keerthi},
       title={Toroidal compactifications of integral models of {S}himura
  varieties of {H}odge type},
        date={2019},
        ISSN={0012-9593},
     journal={Ann. Sci. \'{E}c. Norm. Sup\'{e}r. (4)},
      volume={52},
      number={2},
       pages={393\ndash 514},
         url={https://doi.org/10.24033/asens.2391},
      review={\MR{3948111}},
}

\bib{MaulikShankarTang}{article}{
      author={Maulik, Davesh},
      author={Shankar, Ananth~N.},
      author={Tang, Yunqing},
       title={Picard ranks of {K}3 surfaces over function fields and the
  {H}ecke orbit conjecture},
        date={2022},
        ISSN={0020-9910},
     journal={Invent. Math.},
      volume={228},
      number={3},
       pages={1075\ndash 1143},
         url={https://doi.org/10.1007/s00222-022-01097-x},
      review={\MR{4419630}},
}

\bib{Milne}{incollection}{
      author={Milne, J.~S.},
       title={Introduction to {S}himura varieties},
        date={2005},
   booktitle={Harmonic analysis, the trace formula, and {S}himura varieties},
      series={Clay Math. Proc.},
      volume={4},
   publisher={Amer. Math. Soc., Providence, RI},
       pages={265\ndash 378},
      review={\MR{2192012}},
}

\bib{Noot}{article}{
      author={Noot, Rutger},
       title={Models of {S}himura varieties in mixed characteristic},
        date={1996},
        ISSN={1056-3911},
     journal={J. Algebraic Geom.},
      volume={5},
      number={1},
       pages={187\ndash 207},
      review={\MR{1358041}},
}

\bib{PalMonodromy}{article}{
      author={{Pal}, Ambrus},
       title={{The $p$-adic monodromy group of abelian varieties over global
  function fields of characteristic $p$}},
        date={2015-12},
     journal={arXiv e-prints},
      eprint={1512.03587},
}

\bib{PR}{book}{
      author={Platonov, Vladimir},
      author={Rapinchuk, Andrei},
       title={Algebraic groups and number theory},
      series={Pure and Applied Mathematics},
   publisher={Academic Press, Inc., Boston, MA},
        date={1994},
      volume={139},
        ISBN={0-12-558180-7},
        note={Translated from the 1991 Russian original by Rachel Rowen},
      review={\MR{1278263}},
}

\bib{RapoportRichartz}{article}{
      author={Rapoport, M.},
      author={Richartz, M.},
       title={On the classification and specialization of {$F$}-isocrystals
  with additional structure},
        date={1996},
        ISSN={0010-437X},
     journal={Compositio Math.},
      volume={103},
      number={2},
       pages={153\ndash 181},
         url={http://www.numdam.org/item?id=CM_1996__103_2_153_0},
      review={\MR{1411570}},
}

\bib{RapoportZink}{book}{
      author={Rapoport, M.},
      author={Zink, Th.},
       title={Period spaces for {$p$}-divisible groups},
      series={Annals of Mathematics Studies},
   publisher={Princeton University Press, Princeton, NJ},
        date={1996},
      volume={141},
        ISBN={0-691-02782-X; 0-691-02781-1},
         url={https://doi.org/10.1515/9781400882601},
      review={\MR{1393439}},
}

\bib{ScholzeWeinstein}{article}{
      author={Scholze, Peter},
      author={Weinstein, Jared},
       title={Moduli of {$p$}-divisible groups},
        date={2013},
        ISSN={2168-0930},
     journal={Camb. J. Math.},
      volume={1},
      number={2},
       pages={145\ndash 237},
         url={https://doi.org/10.4310/CJM.2013.v1.n2.a1},
      review={\MR{3272049}},
}

\bib{ShankarEtranges}{misc}{
      author={Shankar, Ananth~N.},
       title={The {H}ecke-orbit conjecture for `mod\`eles \'etranges'},
        note={preprint},
}

\bib{ShankarZhou}{article}{
      author={Shankar, Ananth~N.},
      author={Zhou, Rong},
       title={Serre-{T}ate theory for {S}himura varieties of {H}odge type},
        date={2021},
        ISSN={0025-5874},
     journal={Math. Z.},
      volume={297},
      number={3-4},
       pages={1249\ndash 1271},
         url={https://doi.org/10.1007/s00209-020-02556-y},
      review={\MR{4229601}},
}

\bib{ShenZhang}{article}{
      author={{Shen}, Xu},
      author={{Zhang}, Chao},
       title={{Stratifications in good reductions of Shimura varieties of
  abelian type}},
        date={2017-07},
     journal={arXiv e-prints},
      eprint={1707.00439},
}

\bib{stacks-project}{misc}{
      author={{Stacks project authors}, The},
       title={The stacks project},
         how={\url{https://stacks.math.columbia.edu}},
        date={2020},
}

\bib{Wortmann}{article}{
      author={{Wortmann}, Daniel},
       title={{The $mu$-ordinary locus for {S}himura varieties of {H}odge
  type}},
        date={2013-10},
     journal={arXiv e-prints},
      eprint={1310.6444},
}

\bib{LXiao}{article}{
      author={{Xiao Xiao}, Luciena},
       title={{On the {Hecke}-orbit conjecture for PEL type {S}himura
  Varieties}},
        date={2020-06},
     journal={arXiv e-prints},
      eprint={2006.06859},
}

\bib{xu2020normalization}{article}{
      author={{Xu}, Yujie},
       title={{Normalization in integral models of Shimura varieties of Hodge
  type}},
        date={2020-07},
     journal={arXiv e-prints},
      eprint={2007.01275},
}

\bib{ZhouMotivic}{article}{
      author={{Zhou}, Rong},
       title={{Motivic cohomology of quaternionic Shimura varieties and level
  raising}},
        date={2022},
     journal={Ann. Sci. École Norm. Sup},
      eprint={1901.01954},
      status={To appear},
}

\end{biblist}
\end{bibdiv}
\end{document}